\documentclass[12pt]{amsart}
\usepackage{amssymb,amscd}
\usepackage{verbatim}

\let\frak\mathfrak
\let\Bbb\mathbb

\def\>{\relax\ifmmode\mskip.666667\thinmuskip\relax\else\kern.111111em\fi}
\def\<{\relax\ifmmode\mskip-.333333\thinmuskip\relax\else\kern-.0555556em\fi}
\def\vsk#1>{\vskip#1\baselineskip}
\def\vv#1>{\vadjust{\vsk#1>}\ignorespaces}
\def\vvn#1>{\vadjust{\nobreak\vsk#1>\nobreak}\ignorespaces}

  \let\ssize\scriptstyle
\let\sssize\scriptscriptstyle

\let\Medskip\medskip
\def\medskip{\par\Medskip}
\let\Bigskip\bigskip
\def\bigskip{\par\Bigskip}

\let\Maketitle\maketitle
\def\maketitle{\Maketitle\thispagestyle{empty}\let\maketitle\empty}

\newtheorem{thm}{Theorem}[section]
\newtheorem{cor}[thm]{Corollary}
\newtheorem{lem}[thm]{Lemma}
\newtheorem{prop}[thm]{Proposition}
\newtheorem{conj}[thm]{Conjecture}

\numberwithin{equation}{section}

\theoremstyle{definition}

\let\mc\mathcal
\let\nc\newcommand

\let\al\alpha

\let\la\lambda

\let\phi\varphi

\let\Si\Sigma

\let\om\omega

\let\der\partial

\let\geq\geqslant

\let\leq\leqslant

\let\on\operatorname
\let\bi\bibitem
\let\bs\boldsymbol

\def\C{{\mathbb C}}
\def\Z{{\mathbb Z}}

\def\B{{\mc B}}

\def\F{{\mc F}}

\def\+#1{^{\{#1\}}}

\def\End{\on{End}}

\def\Res{\on{Res}}

\def\sln{\mathfrak{sl}_N}

\def\beq{\begin{equation}}
\def\eeq{\end{equation}}
\def\be{\begin{equation*}}
\def\ee{\end{equation*}}

\nc{\bea}{\begin{eqnarray*}}
\nc{\eea}{\end{eqnarray*}}
\nc{\bean}{\begin{eqnarray}}
\nc{\eean}{\end{eqnarray}}
\nc{\Ref}[1]{{\rm(\ref{#1})}}

\nc{\Il}{{\mc I_{\bs\la}}}
\nc{\bla}{{\bs\la}}
\nc{\Fla}{\F_\bla}
\nc{\tfl}{{T^*\Fla}}
\nc{\GL}{{GL_n(\C)}}
\nc{\GLC}{{GL_n(\C)\times\C^*}}

\let\sd s 

\def\ddk_#1{\kk_{#1}\<\>\frac\der{\der\<\>\kk_{#1}}}

\def\bul{\mathbin{\raise.2ex\hbox{$\sssize\bullet$}}}
\def\intt{\mathchoice
{\mathop{\raise.2ex\rlap{$\,\,\ssize\backslash$}{\intop}}\nolimits}
{\mathop{\raise.3ex\rlap{$\,\sssize\backslash$}{\intop}}\nolimits}
{\mathop{\raise.1ex\rlap{$\sssize\>\backslash$}{\intop}}\nolimits}
{\mathop{\rlap{$\sssize\<\>\backslash$}{\intop}}\nolimits}}

\let\kk q 
\let\cc c

\let\Ko K

\def\GZ/{Gelfand-Zetlin}
\def\KZ/{{\slshape KZ\/}}
\def\qKZ/{{\slshape qKZ\/}}
\def\XXX/{{\slshape XXX\/}}

\nc{\slnl}{{\sln (\lambda)}}
\nc{\PCN}{{   (\C[x])^N   }}
\nc{\di}{\on{Diag}}
\nc{\dio}{\on{Diag}_0}
\nc{\Mm}{{\mc M}}
\nc{\Nn}{{\mc N}}

\nc{\A}{{\mc C}}

\nc{\PCr}{{  P  (\C[x])^n   }}

\nc{\Pk}{{(\bs{P}^1)^k}}

\nc{\N}{{\Bbb N}}

\nc{\Ll}{{\mc L}}

\nc{\ord}{{\on{ord}\,}}

\newcommand{\OS}{\mathcal {A}}
\def\FF{{\mathcal F}}

\nc{\Sing}{{\on{Sing}\,}}
\nc{\sing}{{\on{Sing}\,}}

\nc{\Hess}{{\on{Hess}}}

\nc{\R}{{\Bbb R}}
\newcommand{\Pee}{{\mathbb P}}
\let\on\operatorname
\nc{\Kk}{{\bs K}}
\nc{\Ap}{{A_\Phi(z)}}
\nc{\ap}{{A_\Phi(z)}}

\nc{\sv}{{\sing V}}
\nc{\cd}{{\C^n-\Delta}}

\begin{document}

\hrule width0pt
\vsk->

\title[Arrangements and Frobenius like structures]
{Arrangements and Frobenius like structures}

\author
[Alexander Varchenko ]
{ Alexander Varchenko$\>^{\star}$}

\maketitle

\begin{center}
{\it Department of Mathematics, University of North Carolina
at Chapel Hill\\ Chapel Hill, NC 27599-3250, USA\/}
\end{center}

{\let\thefootnote\relax
\footnotetext{\vsk-.8>\noindent
$^\star$\,{\it E-mail}: anv@email.unc.edu,  supported in part by NSF grant DMS--1101508 and DMS--1362924}}

\medskip

\begin{abstract}
We consider a family of generic weighted arrangements of $n$ hyperplanes in $\C^k$ and show that the Gauss-Manin connection for the associated
hypergeometric integrals, the contravariant form
on the space of singular vectors, and the algebra  of functions on the critical set of the master function
define a Frobenius like structure on the base of the family.

 As a result of this construction we show that the matrix elements of the linear operators of the Gauss-Manin
connection are given by the $2k+1$-st derivatives of a single function on the base of the family, the function called the potential of second kind, see formula \Ref{dERpk2}.

\bigskip

On consid\`ere une famille d'arrangements pond\'er\'es g\'en\'eriques de $n$
hyperplans dans $C^k$ et montre que la connexion de Gauss - Manin pour
les int\'egrales hyperg\'eom\'etriques associ\'ees, la forme contravariante sur l'espace
des vecteurs singuliers et l'alg\'ebre de fonctions sur l'ensemble
des points critiques d\'efinissent une structure du type Frobenius sur la base de la famille.

Comme un r\'esultat de cette construction nous montrons que les \'el\'ements matriciels
des op\'erateurs lin\'eaires de la connexion de Gauss - Manin sont donn\'es par
les $(2k+1)$-mes d\'eriv\'ees d'une seule fonction sur la base de la famille,
cf. la formule (6.46).

\end{abstract}

{\small \tableofcontents

}

\setcounter{footnote}{0}
\renewcommand{\thefootnote}{\arabic{footnote}}

\section{Introduction}

There are three places, where a flat connection depending on a parameter appears:

\noindent
$\bullet$\  KZ equations,
\bean
\label{1}
\phantom{aaaa}
\kappa \frac{\der I}{\der z_i}(z) = K_i(z) I(z), \quad z=(z_1,\dots,z_n),\quad i=1,\dots,n.
\eean
Here $\kappa$ is a parameter, $I(z)$  a $V$-valued function, where  $V$ is a vector space from representation theory,
$K_i(z):V\to V$ are linear operators, depending on $z$. The connection is flat for all $\kappa$.

\noindent
$\bullet$\ Quantum differential equations,
\bean
\label{2}
\phantom{aaaa}
\kappa \frac{\der I}{\der z_i}(z) = p_i *_z I(z), \quad z=(z_1,\dots,z_n),\quad i=1,\dots,n.
\eean
Here $p_1,\dots,p_n$ are generators of some commutative algebra $H$ with quantum multiplication $*_z$ depending on $z$.
These equations are part of the Frobenius structure on the quantum cohomology of a variety.

\noindent
$\bullet$\ Differential equations for hypergeometric integrals associated with a family of weighted arrangements with parallelly
 transported hyperplanes,
\bean
\label{3}
\phantom{aaaa}
\kappa \frac{\der I}{\der z_i}(z) = K_i(z) I(z), \quad z=(z_1,\dots,z_n),\quad i=1,\dots,n.
\eean

It is well known that KZ equations are closely related with the differential equations for hypergeometric integrals.
According to \cite{SV} the KZ equations can be presented as equations for hypergeometric integrals for suitable arrangements. Thus
\Ref{1} and \Ref{3} are related.
Recently it was realized that in some cases the KZ equations appear as quantum differential equations, see \cite{BMO} and \cite{GRTV},
and therefore the  KZ equations are related to the Frobenius structures. On Frobenius structures
see, for example, \cite{D1, D2, M}. Hence \Ref{1} and \Ref{2} are related.
In this paper I argue how a Frobenius like structure may appear on the base of a family of weighted arrangements.
The goal is to make equations \Ref{3} related to Frobenius structures.

The main ingredients of a Frobenius structure are a flat connection depending on a parameter, a constant metric, a multiplication on tangent spaces. In our case,
the connection comes from the differential equations for the associated hypergeometric integrals, the flat metric comes from the contravariant form
on the space of singular vectors and the multiplication comes from the multiplication in the algebra of functions on the critical set of the master function.
In this paper I consider the  families  of generic weighted arrangements.

\smallskip
The organization of the paper is as follows.
In Section \ref{Sec Arrangements},  objects associated with a weighted arrangement are recalled (Orlik-Solomon algebra, space of
 singular vectors, contravariant form, master function,  canonical isomorphism
of the space of singular vectors and the algebra of functions on the critical set of the master function).
In Section \ref{sec par trans}, a family of arrangements with parallelly  transported hyperplanes is considered.
The construction of a Frobenius like structure on the base of the family is given. Conjectures \ref{CB}, \ref{CB3}, \ref{CB4}
are formulated and corollaries of the conjectures
are discussed. In Sections  \ref{sec Exa} and \ref{sec lines on plane} the conjectures are proved for the family  of points on the line
and for a family  of generic arrangements  of lines on plane. The corresponding Frobenius like structures are described.
Here are the corresponding potential functions of second kind:
\bean
\label{pot intro}
\tilde P(z_1,\dots,z_n) = \frac 12 \sum_{1\leq i< j\leq n} a_ia_j\,(z_i-z_j)^2\log(z_i-z_j)
\eean
for the family of arrangements of $n$ points on line and
\bean
\label{pot 2 intro}
&&
\\
&&
\notag
\!\!\!
\tilde P(z_1,\dots,z_n) = \frac1{4!}\!\sum_{1\leq i<j<k\leq n}\! \frac{a_ia_ja_k}{d_{i,j}^2d_{j,k}^2d_{k,i}^2}
(z_id_{j,k} + z_jd_{k,i} + z_kd_{i,j})^4\log(z_id_{j,k} + z_jd_{k,i} + z_kd_{i,j})
\eean
for the family of arrangements of $n$ generic lines on plane. The variables $z_1,\dots,z_n$ are the parameters of the family,
$a_1,\dots,a_n$ are weights, $|a|=a_1+\dots+a_n$, the number $d_{k,\ell}$ is the oriented area of the parallelogram
generated by the normal vectors to the $k$-th and $\ell$-th lines, see formulas \Ref{pot} and \Ref{POT2k}. Note that the potential
$\tilde P$ from \Ref{pot intro} appears in \cite{D2} for $a_1=\dots=a_n$ and in \cite{Ri} for $a_1,\dots,a_n\in\Z$.

In Section \ref{sec arrs k}, the conjectures are proved
for a family of generic arrangements of $n$ hyperplanes in $\C^k$ for any $k$. The potential $\tilde P(z_1,\dots,z_n)$ of second kind is defined by formula
\Ref{POT2k} similar to formulas \Ref{pot intro} and \Ref{pot 2 intro}.
It is shown that the matrix elements of the operators $K_i(z_1,\dots,z_n)$ of the Gauss-Manin connection for
 associated hypergeometric integrals are given by  the $2k+1$-st derivatives of the potential of second kind, see
formula   \Ref{dERpk2}.

This fact that the Gauss-Manin differential equations for associated hypergeometric integrals can be described in terms of derivatives of a single function
on the base of the family is an important application of our Frobenius like structure. One may expect that this is a
manifestation of a much more general phenomenon.

\medskip
It should be stressed that that somewhat technical constructions
in Section \ref{sec par trans}
are explained in details in Sections  \ref{sec Exa},
\ref{sec lines on plane}, and \ref{sec arrs k} for the particular situations discussed there. The reader may decide  to read first
the easiest  Section \ref{sec Exa}.

\smallskip
In this paper I followed one of  I.M.\,Gelfand's rules: for a new subject, choose the simplest nontrivial example and write down
everything explicitly for this example, see the introduction to \cite{EFK}.

\smallskip
I thank V. Schechtman and V. Tarasov for useful discussions.

\section{Arrangements}
\label{Sec Arrangements}

\subsection{Affine arrangement}
\label{An affine arrangement}

Let $k,n$ be positive integers, $k<n$. Denote $J=\{1,\dots,n\}$.
Let $\A =(H_j)_{j\in J}$,  be an arrangement of $n$ affine hyperplanes in
$\C^k$. Denote
$
U = \C^k - \cup_{j\in J} H_j,
$
the complement.
An edge $X_\al \subset \C^k$ of  $\A$ is a nonempty intersection of some
hyperplanes  of $\A$. Denote by
 $J_\al \subset J$ the subset of indices of all hyperplanes containing $X_\al$.
Denote  $l_\al = \on{codim}_{\C^k} X_\al$.

A subset $\{j_1,\dots,j_p\}\subset J$ is called {\it independent} if the hyperplanes
$H_{j_1},\dots,H_{j_p}$ intersect transversally.

We
assume that $\A$ is essential, that is, $\A$ has a vertex.
An edge is called {\it dense} if the subarrangement of all hyperplanes containing
the edge is irreducible: the hyperplanes cannot be partitioned into nonempty
sets so that, after a change of coordinates, hyperplanes in different
sets are in different coordinates.

\subsection{Orlik-Solomon algebra}
Define complex vector spaces $\OS^p(\A)$, $p = 0,  \dots, k$.
 For $p=0$ we set $\OS^p(\A)=\C$. For  $p \geq 1$,\
 $\OS^p(\A)$   is generated by symbols
$(H_{j_1},...,H_{j_p})$ with ${j_i}\in J$, such that
\begin{enumerate}
\item[(i)] $(H_{j_1},...,H_{j_p})=0$
if $H_{j_1}$,...,$H_{j_p}$ are not in general position, that is, if the
intersection $H_{j_1}\cap ... \cap H_{j_p}$ is empty or
 has codimension
 less than $p$;
\item[(ii)]
$ (H_{j_{\sigma(1)}},...,H_{j_{\sigma(p)}})=(-1)^{|\sigma|}
(H_{j_1},...,H_{j_p})
$
for any element $\sigma$ of the
symmetric group $ \Si_p$;
\item[(iii)]
$\sum_{i=1}^{p+1}(-1)^i (H_{j_1},...,\widehat{H}_{j_i},...,H_{j_{p+1}}) = 0
$
for any $(p+1)$-tuple $H_{j_1},...,H_{j_{p+1}}$ of hyperplanes
in $\A$ which are
not in general position and such that $H_{j_1}\cap...\cap H_{j_{p+1}}\not = \emptyset$.
\end{enumerate}
The direct sum $\OS(\A) = \oplus_{p=1}^{N}\OS^p(\A)$ is the (Orlik-Solomon)
 algebra with respect to
 multiplication
 \bean
(H_{j_1},...,H_{j_p})\cdot(H_{j_{p+1}},...,H_{j_{p+q}}) =
 (H_{j_1},...,H_{j_p},H_{j_{p+1}},...,H_{j_{p+q}}) .
\eean

\subsection{Orlik-Solomon algebra as an algebra of differential forms}
For  $j\in J$,  fix a defining equation for the hyperplane $H_j$, $f_{j} = 0$,
where $f_j$ is a polynomial of degree one on $\C^k$.
Consider the logarithmic differential form
$\omega_j = df_j/f_j$ on $\C^k$.
Let $\bar{\OS}(\A)$ be the exterior $\C$-algebra of differential forms
generated by 1 and $\omega_j$, $j\in J$.
The map ${\OS}(\A) \to \bar{\OS}(\A), \ (H_j) \mapsto \omega_j$,
is an isomorphism. We identify ${\OS}(\A)$ and $\bar{\OS}(\A)$.

\subsection{Weights}
\label{sec weights}

An arrangement $\A$ is {\it weighted} if a map $a: J\to \C^\times,\ j\mapsto a_j,$ is given;
 $a_j$ is called the {\it weight} of $H_j$.
For an edge $X_\al$, define its weight
$a_\al = \sum_{j\in J_\al}a_j$.

 Denote  $\nu(a) = \sum_{j\in J} a_j (H_j) \in \OS^1(\A)$.
 Multiplication by $\nu(a)$ defines a differential
$d^{(a)} :    \OS^p(\A) \to \OS^{p+1}(\A),$
$ x  \mapsto \nu(a)\cdot x$, on  $\OS(\A)$.

\subsection{Space of flags, see \cite{SV} }
 For an edge $X_\alpha$ of codimension $l_\alpha=p$, a flag starting at $X_\alpha$ is a sequence
\bean
X_{\alpha_0}\supset
X_{\alpha_1} \supset \dots \supset X_{\alpha_p} = X_\alpha
\eean
of edges such that
$ l_{\alpha_j} = j$ for $j = 0, \dots , p$.
 For an edge $X_\alpha$,
 we define $\overline{\FF}_{\alpha}$  as  the
complex vector space  with basis vectors
$\overline{F}_{{\alpha_0},\dots,{\alpha_p}=\alpha}$
 la\-bel\-ed by the elements of
the set of all flags  starting at $X_\alpha$.

 Define  $\FF_{\alpha}$ as the quotient of
$\overline{\FF}_{\alpha}$ by the subspace generated by all
the vectors of the form
\begin{equation}
\label{flagrelation} \sum\limits_{X_\beta,\
X_{\alpha_{j-1}}\supset X_\beta\supset X_{\alpha_{j+1}}}\
\overline {F}_{{\alpha_0},\dots,
{\alpha_{j-1}},{\beta},{\alpha_{j+1}},\dots,{\alpha_p}=\alpha}\ .
\notag
\end{equation}
Such a vector is determined  by  $j \in \{ 1, \dots , p-1\}$ and
an incomplete flag $X_{\alpha_0}\supset...\supset
X_{\alpha_{j-1}} \supset X_{\alpha_{j+1}}\supset...\supset
X_{\alpha_p} = X_\alpha$ with $l_{\alpha_i}$ $=$ $i$.

Denote by ${F}_{{\alpha_0},\dots,{\alpha_p}}$ the image in $\FF_\alpha$ of the basis vector
$\overline{F}_{{\alpha_0},\dots,{\alpha_p}}$.  For $p=0,\dots,k$, we set
\bean
{\FF}^p(\A)\ =\ \oplus_{X_\alpha,\, l_\alpha=p}\ {\FF}_{\alpha}\ .
\eean

\subsection{Duality, see \cite{SV}}
The vector spaces $\OS^p(\A)$ and $\FF^p(\A)$ are dual.
The pairing $ \OS^p(\A)\otimes\FF^p(\A) \to \C$ is defined as follows.
{}For $H_{j_1},...,H_{j_p}$ in general position, we set
$F(H_{j_1},...,H_{j_p})=F_{{\alpha_0},\dots,{\alpha_p}}$
where
$X_{\alpha_0}=\C^k,\quad X_{\alpha_1}=H_{j_1},\quad \dots , \quad
X_{\alpha_p}=
H_{j_1} \cap \dots \cap H_{j_p}.
$
Then we define $\langle (H_{j_1},...,H_{j_p}), F_{{\alpha_0},\dots,{\alpha_p}}
 \rangle$ $ = (-1)^{|\sigma|},$
if $F_{{\alpha_0},\dots,{\alpha_p}}
= F(H_{j_{\sigma(1)}},...,H_{j_{\sigma(p)}})$ for some $\sigma \in \Si_p$,
and $\langle (H_{j_1},...,H_{j_p}), F_{{\alpha_0},\dots,{\alpha_p}} \rangle = 0$ otherwise.

Define a map $\delta^{(a)} :  \FF^p(\A)\to\FF^{p-1}(\A)$ to be
 the map adjoint to
$d^{(a)} :   \OS^{p-1}(\A) \to \OS^{p}(\A)$.
An element $v \in \FF^k(\A)$ is called  {\it singular}  if
$\delta^{(a)} v = 0$. Denote by
\bean
\Sing\,\FF^k(\A) \subset \FF^k(\A)
\eean
the subspace of all singular vectors.

\subsection{Contravariant map and form, see \cite{SV}}
 Weights  $(a_j)_{j\in J}$ determine a contravariant  map
 \bean
\mathcal S^{(a)} : \FF^p(\A) \to \OS^p(\A),
\quad
  {F}_{{\alpha_0},\dots,{\alpha_p}}\ \mapsto \
\sum \ a_{j_1} \cdots a_{j_p}\ (H_{j_1}, \dots , H_{j_p})\ ,
\eean
 where the sum is taken over all $p$-tuples $(H_{j_1},...,H_{j_p})$ such that
\bean
H_{j_1} \supset X_{\al_1},\ {}\ .\ .\ .\ {} , \ {} H_{j_p}\supset X_{\alpha_p}\ .
\eean
Identifying $\OS^p(\A)$ with $\FF^p(\A)^*$, we consider
this map  as a bilinear form,
\bean
S^{(a)} : \FF^p(\A) \otimes \FF^p(\A) \to \C .
\eean
The bilinear form is
called the {\it contravariant form}.
The contravariant form  is symmetric.
For $F_1, F_2 \in \FF^p(\A)$, we have
\bean
S^{(a)}(F_1,F_2) =
\sum_{\{j_1, \dots , j_p\} \subset J} \ a_{j_1} \cdots a_{j_p}
\ \langle (H_{j_1}, \dots , H_{j_p}), F_1 \rangle
\ \langle (H_{j_1}, \dots , H_{j_p}), F_2 \rangle \ ,
\eean
where the sum is over all unordered $p$-element subsets.

\subsection{Arrangement with normal crossings}
\label{An arrangement with normal crossings only}
 An essential arrangement $\A$ is {\it with normal crossings},
if exactly $k$ hyperplanes meet at every vertex of $\A$.
Assume that $\A$ is an essential arrangement with normal crossings only.
A subset $\{j_1,\dots,j_p\}\subset J$ is called {\it independent} if
the hyperplanes $H_{j_1},\dots,H_{j_p}$ intersect transversally.


A basis of $\OS^p(\A)$ is formed by
$(H_{j_1},\dots,H_{j_p})$ where
$\{{j_1} <\dots <{j_p}\}$  are independent ordered $p$-element subsets of
$J$. The dual basis of $\FF^p(\A)$ is formed by the corresponding vectors
$F(H_{j_1},\dots,H_{j_p})$.
These bases of $\OS^p(\A)$ and $\FF^p(\A)$ will be called {\it standard}.

We  have
\bean
\label{skew}
F(H_{j_1},\dots,H_{j_p}) = (-1)^{|\sigma|}
F(H_{j_{\sigma(1)}},\dots,H_{j_{\sigma(p)}}), \qquad \sigma \in \Si_p.
\eean
For an independent subset $\{j_1,\dots,j_p\}$, we have
\bean
S^{(a)}(F(H_{j_1},\dots,H_{j_p}) , F(H_{j_1},\dots,H_{j_p})) = a_{j_1}\cdots a_{j_p}
\eean
and
\bean
S^{(a)}(F(H_{j_1},\dots,H_{j_p}) , F(H_{i_1},\dots,H_{i_k})) = 0
\eean
for distinct elements of the standard basis.

\subsection{If the weights of dense edges are nonzero}
\label{sec rem on gener wei}

\begin{thm}
\label{thm Shap nondeg}
Assume that the weights $(a_j)_{j\in J}$ are such that the weights of all dense edges of $\A$ are nonzero. Then
\begin{enumerate}
\item[(i)] the contravariant form is nondegenerate;

\item[(ii)]
$H^p(\OS^*(\A),d^{(a)}) = 0$ for $ p < k$
and dim $H^k(\OS^*,d^{(a)}) = |\chi(U)|$, where $\chi (U)$ is the Euler
characteristics of $U$.

\end{enumerate}
In particular, these statements hold if all the weights are positive.
\end{thm}

Part (i) is proved in \cite{SV}. Part (ii) is a straightforward corollary of results in
\cite{SV} as explained in Theorem 2.2 in \cite{V6}. Part (ii) is proved in \cite{Y}, \cite{OT2}.

\subsection{Master function}
\label{master function}

Given weights $(a_j)_{j\in J}$, define the (multivalued) {\it master function}
$\Phi : U \to \C$ by the formula:
\bean
\label{def mast fun}
\Phi =  \Phi_{\A,a} = \sum_{j\in J}\,a_j \log f_j .
\eean
A point $t\in U$ is a {\it critical point} if\ $d\Phi\vert_t =\nu(a)\vert_t= 0$.

\begin{thm}[\cite{V3, OT1, Si}]
\label{thm V,OT,S}
For generic weights $(a_j)_{j\in J}$ all the critical points of
$\Phi$ are nondegenerate and the number of critical points equals
  $|\chi(U)|$.
  \qed\end{thm}

\subsection{If the weights are unbalanced}
\label{sec unbalanced}

Let $\A=(H_j)_{j\in J}$ be an essential arrangement in
$\C^k$ with weights $(a_j)_{j\in J}$.  Consider the compactification of the
arrangement $\A$ in the
projective space $\Pee^k$. Assign
the weight $a_\infty=-\sum_{j\in J} a_j$ to the hyperplane
$H_\infty=\Pee^k-\C^k$ and denote by  $\hat\A$
the arrangement $(H_j)_{j\in J\cup \infty}$ in $\Pee^k$.

The weights of the arrangement $\A$ are called {\it unbalanced}
if the weights of all the dense
edges of $\hat\A$ are nonzero, see \cite{V6}.
For example, if all the weights $(a_j)_{j\in J}$ are positive, then the
weights are unbalanced.
The unbalanced weights form a Zarisky open subset in the space of all weight systems on $\A$.

\begin{thm} [\cite{V6}]
\label{lem crit 1}
If the weights $a=(a_j)_{j\in J}$ of $\A$ are unbalanced, then
all the critical points of the master function of the weighted arrangement
$(\A,a)$ are isolated and the sum of Milnor numbers of all the critical points equals
 $|\chi(U)|$.

\end{thm}

\subsection{Hessian and residue bilinear form}
\label{sec hess}

Denote  $\C(U)$ the algebra of rational functions on $\C^k$ regular on $U$ and
$I_\Phi  =\langle
\frac{\partial \Phi} {\partial t_i }
\ |\ i=1,\dots,k\ \rangle  \subset \C(U)$
the ideal generated by first  derivatives of $\Phi$.
Let
\bean
\label{AP}
A_\Phi = \C(U)/ I_\Phi
\eean
 be
the algebra of functions on the critical set  and
$[\,]:\C(t)_U \to  A_\Phi, \ {} f\mapsto [f]$, the canonical homomorphism.

If all critical points are isolated,
then the critical set  is finite and the algebra
$ A_\Phi$ is finite-dimensional. In that case, $ A_\Phi$
is the direct sum of local algebras corresponding to points $p$ of
the critical set,
\bean
\label{loc alg}
 A_\Phi = \oplus_{p}  A_{p,\Phi} \ .
\eean
The local algebra
$A_{p,\Phi}$ can be defined as the quotient of the algebra of germs at $p$
 of holomorphic functions modulo the ideal
$ I_{p,\Phi}$ generated by first derivatives of $\Phi$.

\begin{lem} [\cite{V6}]
\label{lem f_j generate}
The elements
$[1/f_j]$, $j\in J$, generate $ A_\Phi$.
\qed
\end{lem}

We fix  affine coordinates $t_1,\dots,t_k$ on $\C^k$. Let
\bean
f_j = b^0_j +b^1_jt_1+ \dots + b^k_jt_k.
\eean

\begin{lem}
\label{lem on ONE 1}
The identity element $[1] \in \Ap$ satisfies the equation
\bean
\label{ONE 1}
[1] = \frac 1{|a|}\sum_{j\in J} b^0_j\Big[\frac{a_j}{f_j}\Big],
\eean
where $|a|=\sum_{j\in J} a_j$.
\end{lem}

\begin{proof}
The lemma follows from the equality
\bean
\label{one1}
\sum_{i=1}^k t_i\frac{\der\Phi}{\der t_i} = |a| - \sum_{j\in J} b^0_j\Big[\frac{a_j}{f_j}\Big].
\eean
\end{proof}

Surprisingly, formula \Ref{one1} and Lemma \ref{lem f_j generate} play central roles in the constructions of this paper.

\smallskip

We define the rational function $\Hess : \C^k \to \C$, regular on $U$, by the formula
\bean
\Hess (t)\ =\ \det_{1\leq i,j \leq k}
\Big( \frac{\partial^2\Phi} {\partial t_i \partial t_j} \Big)(t) \ .
\eean
The function is called the {\it Hessian} of $\Phi$.

Let
 $\rho_{p} :  A_{p,\Phi} \to \C$, be the
{\it Grothendieck residue},
\bean
\label{res map}
f \ \mapsto \ \frac 1{(2\pi \sqrt{-1})^k}\,\Res_{p}
\ \frac{ f}{\prod_{i=1}^k\, \frac{\der \Phi}{\der t_i}}
=\frac{1}{(2\pi \sqrt{-1})^k}\int_{\Gamma_p}
\frac{f\ dt_1\wedge\dots\wedge dt_k}{\prod_{i=1}^k \frac{\der \Phi}{\der t_i}}\ ,
\eean
where
$\Gamma_p$ is the real $k$ cycle in a small neighborhood of $p$, defined by the equations
$|\frac{\der \Phi}{\der t_i}|=\epsilon_i,\ i=1,\dots,k,$ and  oriented by the condition
$d\arg \frac{\der \Phi}{\der t_1}\wedge\dots\wedge d\arg \frac{\der \Phi}{\der t_k}>0$,
here $\epsilon_s$ are positive  numbers sufficiently small with respect to the size of the neighborhood, see \cite{GH, AGV}.

Let $(\,,\,)_{p}$ be the {\it residue
bilinear form} on $A_{p,\Phi}$,
\bean
\label{Gr form}
( f, g)_{p}\ =\ \rho_{p} (f g) ,
\eean
for $f,g\in A_{p,\Phi}$. This form is nondegenerate.

Let all the critical points of $\Phi$ be isolated and hence, $ A_\Phi = \oplus_{p}  A_{p,\Phi}$. We define
the {\it residue bilinear form}  $(\,,\,)$ on $A_\Phi$  as $\oplus_{p} (\,,\,)_{p}$.
This form is nondegenerate and $(fg,h)=(f,gh)$ for all $f,g,h\in A_\Phi$. In other words,
the pair $(A_\Phi, (\,,\,))$ is a {\it Frobenius algebra}.

\subsection{Canonical isomorphism and algebra structures on $\sing \FF^k(\A)$}
\label{sec Special}

Let $(F_m)_{m\in M}$ be a basis of $\FF^k(\A)$ and $(H^m)_{m\in M} \subset \OS^k(\A)$ the dual basis.
Consider the element $\sum_m H^m\otimes F_m \in \OS^k(\A) \otimes \FF^k(\A)$.
We have $H^m = f^mdt_1\wedge\dots\wedge dt_k$ for some $f^m \in \C(U)$.
The element
\bean
\label{spec elt}
E = \sum_{m\in M} f^m \otimes F_m  \ \in\ \C(U)\otimes \FF^k(\A)
\eean
is called the {\it canonical element} of  $\A$.
Denote  $[E]$ the image of the canonical element in $ A_\Phi\otimes\FF^k(\A)$.

\begin{thm}
[\cite{V6}]
\label{thm E-sing}
We have $[E] \in  A_\Phi \otimes \Sing \FF^k(\A)$.
\end{thm}

Assume that all critical points of $\Phi$ are isolated. Introduce the linear map
\bean
\label{map alpha}
\al :  A_\Phi \to \sing \FF^k(\A), \qquad
[g] \mapsto ([g],[E]).
\eean

\begin{thm}
[\cite{V6}]
\label{thm alpha}
If the weights $(a_j)_{j\in J}$ of $\A$ are unbalanced, then the canonical map $\al$ is an isomorphism of vector
spaces. The isomorphism $\al$ identifies the residue form on $A_\Phi$ and
the contravariant form on $\sing \FF^k(\A)$ multiplied by $(-1)^k$, that is,
\bean
(f,g) = (-1)^kS^{(a)}(\al(f),\al(g))\qquad \text{for all}\ f,g\in A_\Phi.
\eean

\end{thm}

The map $\al$ is called the {\it canonical} map or {\it canonical} isomorphism.

\begin{cor}[\cite{V6}]
\label{cor nondeg}
The restriction of the contravariant form $S^{(a)}$ to the subspace
$\sing \FF^k(\A)$ is nondegenerate.
\qed
\end{cor}

On the restriction of the contravariant form $S^{(a)}$ to the subspace
$\sing \FF^k(\A)$ see \cite{FaV}.

 If all critical points $p$ of the master function are nondegenerate, then
\bean
\label{EX}
\al : [g] \mapsto \sum_p \sum_m\frac {g(p)f^m(p)}{\Hess(p)} F_m .
\eean

If the weights $(a_j)_{j\in J}$ of $\A$ are unbalanced, then the
 canonical isomorphism $\al :  A_\Phi \to \sing \FF^k(\A)$ induces a commutative associative algebra structure
on $\sing \FF^k(\A)$. Together with the contravariant form it is a Frobenius algebra structure.

\subsection{Change of variables and canonical isomorphism}
\label{sec Change}
Assume that we  change coordinates on $\C^n$,\ $t_i =\sum_{j=1}^k c_{i,j} s_j$
with $c_{i,j}\in \C$.

\begin{lem}
\label{lem change}
The canonical map \Ref{map alpha} in coordinates $t_1,\dots,t_k$ equals
the canonical map \Ref{map alpha} in coordinates $s_1,\dots,s_k$ divided by $\det(c_{i,j})$,
$\al_t = \frac 1{\det(c_{i,j})}\al_s$.

\end{lem}

\begin{proof}
We have $H^m = f^m dt_1\wedge\dots\wedge dt_k = \det(c_{i,j}) f^m ds_1\wedge\dots\wedge ds_k$
and $\Hess_t = \det^2(c_{i,j}) \Hess_s$. Now the lemma follows, for example, from \Ref{EX}.
\end{proof}

To make the map \Ref {map alpha} independent of coordinates one needs to consider it as a map
\bean
\label{al forms}
 A_\Phi\otimes dt_1\wedge\dots\wedge dt_k  \to \sing \FF^k(\A), \qquad
[g]\otimes dt_1\wedge\dots\wedge dt_k  \mapsto ([g],[E]).
\eean

\section{A family of parallelly transported hyperplanes}
\label{sec par trans}

This section contains the main constructions of the paper.
These constructions
are explained in details in Sections  \ref{sec Exa},
\ref{sec lines on plane}, and \ref{sec arrs k} for the particular situations discussed there.

\subsection{An arrangement in  $\C^n\times\C^k$}
\label{An arrangement in}
Recall that $J=\{1,\dots,n\}$.
Consider $\C^k$ with coordinates $t_1,\dots,t_k$,\
$\C^n$ with coordinates $z_1,\dots,z_n$, the projection
$\C^n\times\C^k \to \C^n$.
Fix $n$ nonzero linear functions on $\C^k$,
$g_j=b_j^1t_1+\dots + b_j^kt_k,$\ $ j\in J,$
where $b_j^i\in \C$.
Define $n$ linear functions on $\C^n\times\C^k$,
$f_j = z_j+g_j = z_j + b_j^1t_1+\dots + b_j^kt_k,$\ $ j\in J.$
In $\C^n\times \C^k$ we define
 the arrangement $\tilde \A = \{ \tilde H_j\ | \ f_j = 0, \ j\in J \}$.
Denote $\tilde U = \C^n\times \C^k - \cup_{j\in J} \tilde H_j$.

For every $z=(z_1,\dots,z_n)$ the arrangement $\tilde \A$
induces an arrangement $\A(z)$ in the fiber of the projection over $z$. We
identify every fiber with $\C^k$. Then $\A(z)$ consists of
hyperplanes $H_j(z), j\in J$, defined in $\C^k$ by the equations
$f_j=0$. Denote $\label{U(A(z))}
U(\A(z)) = \C^k - \cup_{j\in J} H_j(z)$,  the complement to the arrangement $\A(z)$.
We assume that for every $z$ the arrangement $\A(z)$ has a vertex. This happens
if and only if $\A(0)$ has a vertex.

A point $z\in\C^n$ is called {\it good} if $\A(z)$ has normal
crossings only.  Good points form the complement in $\C^n$ to the union
of suitable hyperplanes called the {\it discriminant}.

\subsection{Discriminant}
\label{Discr}

The collection $(g_j)_{j\in J}$ induces a
matroid structure on $J$.  A subset $C=\{i_1,\dots,i_r\}\subset J$ is
a {\it circuit}  if $(g_i)_{i\in C}$ are linearly dependent but any
proper subset of $C$ gives linearly independent $g_i$'s.

For a circuit $C=\{i_1,\dots,i_r\}$, \  let
$(\la^C_i)_{i\in C}$ be a nonzero collection of complex numbers such that
$\sum_{i\in C}
\la^C_ig_i = 0$. Such a collection  is unique up to
multiplication by a nonzero number.

For every circuit $C$ we fix such a collection
and denote $f_C = \sum_{i\in C} \la^C_iz_i$.
The equation $f_C=0$ defines a hyperplane $H_C$ in
$\C^n$.
It is convenient to assume that $\la^C_i=0$ for $i\in J-C$ and write
$f_C = \sum_{i\in J} \la^C_iz_i$.

For any $z\in\C^n$, the hyperplanes $(H_i(z))_{i\in C}$ in $\C^k$ have nonempty
intersection if and only if $z\in H_C$. If $z\in H_C$, then the
intersection has codimension $r-1$ in $\C^k$.

Denote by $\frak C$ the set of all circuits in $J$.
Denote  $\Delta = \cup_{C\in \frak C} H_C$.
The arrangement $\A(z)$ in $\C^k$
has normal crossings  if and only if $z\in \C^n-\Delta$, see \cite{V6}.

For example, if $k=1$ and $f_j=t_1+z_j, j\in J$, then the discriminant is the
union of hyperplanes in $\C^n$ defined by the equations $z_i-z_j=0$,
$1\leq i<j\leq n$.

\subsection{Good fibers and combinatorial connection}
\label{sec Good fibers}

For any
$z^1, z^2\in \C^n-\Delta$, the spaces $\FF^p(\A(z^1))$, $\FF^p(\A(z^2))$
 are canonically identified. Namely, a vector $F(H_{j_1}(z^1),\dots,H_{j_p}(z^1))$ of the first space
is identified  with
the vector $F(H_{j_1}(z^2),\dots,H_{j_p}(z^2))$ of the second. In other words, we identify the standard bases of
these spaces.

Assume that nonzero weights $(a_j)_{j\in J}$ are given. Then each
arrangement $\A(z)$  is weighted. The identification of spaces $\FF^p(\A(z^1))$,
$\FF^p(\A(z^2))$ for $z^1,z^2\in\C^n-\Delta$ identifies the corresponding subspaces
$\Sing\FF^k(\A(z^1))$, $\Sing\FF^k(\A(z^2))$ and contravariant forms.

For a point $z\in\C^n-\Delta$, we denote $V=\FF^k(\A(z))$, $\Sing V=\Sing\FF^k(\A(z))$.
The triple $(V, \Sing V, S^{(a)})$ does not depend on  $z\in\C^n-\Delta$
under the above identification.

As a result of this reasoning we obtain the canonically trivialized
vector bundle
\bean
\label{combbundle}
\sqcup_{z\in \C^n-\Delta}\,\FF^k(\A(z))\to \C^n-\Delta,
\eean
with the canonically trivialized subbundle $\sqcup_{z\in \C^n-\Delta}\,\Sing\FF^k(\A(z))\to \C^n-\Delta$
and the constant contravariant form on the fibers.
This trivialization identifies the bundle in \Ref{combbundle} with
\bean
\label{cmb Bundle}
(\cd)\times V\to \cd
\eean
and the subbundle with
\bean
\label{comb Bundle}
(\cd)\times(\sv)\to \cd.
\eean
The bundle in \Ref{comb Bundle}
 will be called the {\it combinatorial bundle}, the flat connection on it will be called {\it combinatorial}.

\begin{lem}
\label{lem UNbalanced}
If the weights $(a_j)_{j\in J}$ are unbalanced for the arrangement $\A(z)$ for some $z\in\C^n-\Delta$,
then the weights $(a_j)_{j\in J}$ are unbalanced for $\A(z)$ for all $z\in\C^n-\Delta$.
\qed
\end{lem}

\subsection{Bad fibers}
\label{sec Bad fibers}

Points of  $\Delta\subset \C^n$ are called {\it bad}.
Let $z^0\in\Delta$ and $z\in \C^n-\Delta$.
By definition, for any $p$
the space $\OS^p(\A(z^0))$ is obtained from $\OS^p(\A(z))$ by
adding new relations.
Hence $\OS^k(\A(z^0))$ is canonically identified
with the quotient space of $V^* = \OS^k(\A(z))$ and
$\FF^p(\A(z^0))$ is  identified
with a subspace of $V = \FF^p(\A(z))$.

\subsection{Operators $K_j(z):V\to V$, $j\in J$}
\label{sec key identity}

For any circuit $C=\{i_1, \dots, i_r\}\subset J$, we
define the linear operator $L_C : V\to V$ as follows.

For $m=1,\dots,r$, we define $C_m=C-\{i_m\}$.
Let $\{{j_1}<\dots <{j_k}\}
\subset J$ be an independent ordered subset and
$F(H_{j_1},\dots,H_{j_k})$
the corresponding element of the standard basis.
We define $L_C : F(H_{j_1},\dots,H_{j_k}) \mapsto 0$ if
$|\{{j_1},\dots,{j_k}\}\cap C| < r-1$.
If $\{{j_1},\dots,{j_k}\}\cap C = C_m$ for some
$1\leq m\leq r$, then by using the skew-symmetry property \Ref{skew}
we can write
\bean
F(H_{j_1},\dots,H_{j_k})
\,=\,
\pm\, F(H_{i_1},H_{i_2},\dots,\widehat{H_{i_{m}}},\dots,H_{i_{r-1}}H_{i_{r}},H_{s_1},\dots,H_{s_{k-r+1}})
\eean
with $\{{s_1},\dots,{s_{k-r+1}}\}=
\{{j_1},\dots,{j_k}\}-C_m$.
We set
\bean
\label{L_C}
L_C
&:&
F(H_{i_1},\dots,\widehat{H_{i_{m}}},\dots,H_{i_{r}},H_{s_1},\dots,H_{s_{k-r+1}})
 \mapsto
\\
\notag
&&
\phantom{aaaaaaaaa}
(-1)^m \sum_{l=1}^{r} (-1)^l a_{i_l}
F(H_{i_1},\dots,\widehat{H_{i_{l}}},\dots,H_{i_{r}},H_{s_1},\dots,H_{s_{k-r+1}}) .
\eean
Consider on $\C^n\times\C^k$ the logarithmic differential one-forms
$\omega_C = \frac {df_C}{f_C}, \ C\in \frak C.$
Recall that  $f_C = \sum_{j\in J}\la^C_jz_j$. We define
\bean
\label{K_j}
K_j(z) \ = \ \sum_{C\in \frak C}
\,\frac{\la_j^C}{f_C(z)} \,L_C \,,
\qquad
j\in J .
\eean
The operators $K_j(z)$ are rational functions on $\C^n$ regular on $\C^n-\Delta$
and
\bean
\label{LK}
\sum_{C\in \frak C}
\omega_{C} \otimes L_C = \sum_{j\in J} dz_j\otimes K_j(z) .
\eean

\begin{thm} [\cite{V6}]
\label{thm K sym}
For any $j\in J$ and $z\in \C^n-\Delta$, the operator $K_j(z)$ preserves
the subspace $\Sing V\subset V$ and is a symmetric operator,
$S^{(a)}(K_j(z)v,w)= S^{(a)}(v, K_j(z)w)$ for all $v,w\in V$.
\end{thm}

\subsection{Gauss-Manin connection on  $(\cd)\times (\Sing V)\to \cd$}
\label{Construction}
Consider the master function
\bean
\label{Mast}
\Phi(z,t) = \sum_{j\in J} a_j \log f_j(z,t)
 \eean
as a function on $\tilde U\subset \C^n\times \C^k$.
Let $\kappa\in\C^\times$.
The function $e^{\Phi(z,t)/\kappa}$
defines a rank one local system $\mc L_\kappa$
on $\tilde U$ whose horizontal sections
over open subsets of $\tilde U$
 are univalued branches of $e^{\Phi(z,t)/\kappa}$ multiplied by complex numbers,
see for example \cite{SV, V2}.

The vector bundle
\bean
\sqcup_{z\in \C^n-\Delta}\,H_k(U(\A(z)), \mc L_\kappa\vert_{U(\A(z))})
 \to  \C^n-\Delta
 \eean
will be called the {\it homology bundle}. The homology bundle has a canonical  flat Gauss-Manin connection.

For a fixed $z$, choose any $\gamma\in H_k(U(\A(z)), \mc L_\kappa\vert_{U(\A(z))})$.
The linear map
\bean
\label{int map}
\{\gamma\} \ {}:\ {} \OS^k(\A(z)) \to \C,  \qquad \om \mapsto \int_{\gamma} e^{\Phi(z,t)/\kappa} \om,
\eean
is an element of $\Sing\FF^k(\A(z))$ by Stokes' theorem.
It is known that for generic  $\kappa$
any element of $\Sing\FF^k(\A(z))$ corresponds to a certain
$\gamma$ and  in that case this construction  gives an isomorphism
\bean
\label{iSO}
H_k(U(\A(z)), \mc L_\kappa\vert_{U(\A(z))}) \to \Sing\FF^k(\A(z)),
\eean
 see \cite{SV}. This isomorphism will be called the {\it integration isomorphism}. The precise values of $\kappa$ for which \Ref{iSO} is an isomorphism can be deduced
 from the determinant formula in \cite{V1}.

For generic $\kappa$ the fiber isomorphisms \Ref{iSO} defines an isomorphism of the homology bundle and
the combinatorial bundle. The  Gauss-Manin connection induces a flat connection on the combinatorial bundle.
This connection on the combinatorial bundle will  be also called the {\it Gauss-Manin connection}.

Thus, there are two connections on the combinatorial bundle: the combinatorial connection and the Gauss-Manin
connection depending on  $\kappa$. In this situation we can consider the differential equations
for flat sections of the Gauss-Manin connection with respect to the combinatorially flat standard basis. Namely,
let $\gamma(z) \in H_k(U(\A(z)), \mc L_\kappa\vert_{U(\A(z))})$ be a flat section of the Gauss-Manin connection.
Let us write the corresponding section $I_\gamma(z)$ of the bundle $\C^n\times \Sing V\to \C^n$
in the combinatorially flat standard basis,
\bean
\label{Ig}
&&
{}
\\
\notag
&&
I_\gamma(z) =\!\!\!
\sum_{{\rm independent } \atop \{j_1 < \dots < j_k\} \subset J }\!\!\!
I_\gamma^{j_1,\dots,j_k}(z)
 F(H_{j_1}, \dots , H_{j_k}),
 \quad
 I_\gamma^{j_1,\dots,j_k}(z)
=
\int_{\gamma(z)} e^{\Phi(z,t)/\kappa}
 \omega_{j_1} \wedge \dots \wedge \omega_{j_k}.
\eean
For $I=\sum I^{j_1,\dots,j_k} F(H_{j_1}, \dots , H_{j_k})$ and $j\in J$, we denote
\bean
\label{deriv}
\frac{\der I}{\der z_j} = \sum \frac{\der I^{j_1,\dots,j_k}}{\der z_j} F(H_{j_1}, \dots , H_{j_k}).
\eean

\begin{thm} [\cite{V2, V6}]
\label{thm ham normal}
The section $I_\gamma(z)$ satisfies the differential equations
\bean
\label{dif eqn}
\kappa \frac{\der I}{\der z_j}(z) = K_j(z)I(z),
\qquad
j\in J,
\eean
where  $K_j(z): V\to V$ are the linear operators defined in \Ref{K_j}.

\end{thm}

From this formula we see, in particular, that the  combinatorial connection on the combinatorial bundle
is the limit of the Gauss-Manin connection as $\kappa \to\infty$.

\subsection{Bundle of algebras}
\label{sec bundle of alg}

For $z\in \C^n$, denote $A_\Phi(z)$ the algebra of functions on the critical set of
the master function $\Phi(z,\cdot) : U(\A(z))\to \C$.
Assume that the weights $(a_j)_{j\in J}$ are unbalanced for all $\A(z)$, $ z\in\C^n-\Delta$.
Then the dimension of  $A_{\Phi}(z)$  does not depend on $z\in\C^n-\Delta$ and equals $\dim \sv$.
Denote $|a|=\sum_{j\in J}a_j$.

\begin{lem}
\label{lem on ONE}
The identity element $[1](z) \in \Ap$ satisfies the equation
\bean
\label{ONE}
[1](z) = \frac 1{|a|}\sum_{j\in J} z_j\Big[\frac{a_j}{f_j}\Big].
\eean
\end{lem}
\begin{proof} The lemma follows from Lemma \ref{lem on ONE 1}.
\end{proof}

The vector bundle
\bean
\label{algebra bundle}
\sqcup_{z\in\C^n-\Delta} A_{\Phi}(z) \to \C^n-\Delta
\eean
will be called the {\it bundle of algebras} of functions on the critical set.
The  fiber isomorphisms \Ref{map alpha},
\bea
\label{fiber iso}
\al(z) : A_{\Phi}(z) \to \sing V,
\eea
establish an isomorphism $\al$ of the bundle of algebras and the combinatorial bundle.
The isomorphism $\al$ and the connections on the combinatorial bundle (combinatorial and Gauss-Manin connections)
induce connections on the bundle of algebras which will be called also the {\it combinatorial and Gauss-Manin
connections} on the bundle of algebras.

\smallskip

The canonical isomorphism $\al(z)$ induces a Frobenius algebra structure on $\sing V$ which depends on $z$. The multiplication $*_z$
is described by
the following theorem.

\begin{thm}[\cite{V6}]
\label{K/f}
The elements $\al(z)\Big[\frac{a_j}{f_j}\Big]\in \sing V$, $j\in J$, generate the algebra. We have
\bean
\al(z)\Big[\frac{a_j}{f_j}\Big] *_z v = K_j(z) v,
\eean
for all $v\in \Sing V$ and $j\in J$.
\end{thm}

\subsection{Quantum integrable model of the arrangement $(\A(z),a)$}
For $z\in \cd$, the (commutative) subalgebra $\B(z)\subset \End(\sv)$ generated by $K_j(z),j\in J$, is called
the {\it  algebra of geometric Hamiltonians}, the triple
$(\sv, S^{(a)}, \B(z))$ is called the {\it quantum integrable model of the weighted arrangement}
$(\A(z),a)$, see \cite{V6}.

The canonical isomorphism $\al(z)$ identifies the triple
$(\sv, S^{(a)}, \B(z))$ with the triple $(\ap,(-1)^k(\,,\,)_z,\ap)$, see Theorems \ref{thm alpha} and \ref{K/f}.

\smallskip
Notice that the operators $K_j(z)$ are defined in combinatorial terms, see Section \ref{sec key identity},
 while the algebra $\ap$ is an analytic object, see \Ref{AP}. C.f. Corollaries \ref{an comb2} and \ref{an combk}.

\subsection{A remark. Asymptotically flat sections}
\label{sec Ass}
Assume that the weights $(a_j)_{j\in J}$ are unbalanced for all $\A(z)$, $ z\in\C^n-\Delta$.
Let $B\subset \C^n-\Delta$ be an open real $2n$-dimensional ball. Let $\Psi : B\to \C$ be a holomorphic function.
Let $s_j, j\in \Z_{\geq 0}$ be holomorphic sections over $B$ of the bundle of algebras, see \Ref{algebra bundle}.
We say that
\bean
\label{A section}
s(z,\kappa) = e^{\Psi(z)/\kappa}\sum_{j\geq 0} \kappa^js_j(z)
\eean
is an {\it asymptotically flat section} of the Gauss-Manin connection on bundle of algebras as $\kappa \to 0$
if $s(z,\kappa)$ satisfies the flat section equations formally, see, for example, \cite{RV, V5}.

Assume that $B$ is such that for any $z\in B$, all the critical points of $\Phi(z,\cdot) : U(\A(z))\to\C$
are nondegenerate. Let us order them: $p_1(z), \dots, p_d(z)$, where $d=\dim A_\Phi(z)=\dim \sing V$.
We may assume that every $p_i(z)$ depends on $z$ holomorphically. Then the function
\bean
z\mapsto
\Hess(z,p_i(z)) = \det_{1\leq i,j \leq k}
\Big( \frac{\partial^2\Phi} {\partial t_i \partial t_j} \Big)(z,p_i(z))
\eean
is a nonzero holomorphic function on $B$. We fix a square root
$\Hess(z,p_i(z))^{1/2}$.
We denote $w_i(z)$ the element of $A_\Phi(z)$ which equals $\Hess(z,p_i(z))^{1/2}$ at $p_i(z)$ and equals
zero at all other critical points. Let $(\,,\,)_z$ be the residue form on $A_\Phi(z)$. Then
\bean
\label{orth}
(w_i(z),w_j(z))_z = \delta_{ij} \quad\text{and}\quad w_i(z)\cdot w_j(z) = \delta_{ij} \Hess(z,p_i(z))^{1/2} w_i(z)
\eean
for all $i,j$.

\begin{thm}
\label{thm As}
For every $i$, there exists a unique asymptotically flat section
 $s(z,\kappa)$
 \linebreak
$= e^{\Psi(z)/\kappa}\sum_{j\geq 0} \kappa^js_j(z)$
of the Gauss-Manin connection on the bundle of algebras  such that
\bean
\label{property ass}
\Psi(z)=\Phi(z,p_i(z))\qquad\text{and}
\qquad
s_0(z)= w_i(z).
\eean

\end{thm}

\begin{proof}
We first write asymptotically flat sections of the Gauss-Manin connection
on the bundle $(\C^n-\Delta)\times\sing V\to (\C^n-\Delta)$ by using the steepest descent method as in \cite{RV, V5}
and then observe that the leading terms of those sections are nothing else but
\linebreak
$\al(z)\!\!\left(e^{\Phi(z,p_i(z))/\kappa}w_i(z)\right)$.
\end{proof}

\subsection{Conformal blocks, period map, potential functions}
\label{CB conj}

Denote by
\bean
\label{oNE}
\{1\}(z) = \al(z)([1](z))
\eean
the identity element of the algebra structure on $\sv$ corresponding to a point $z\in\cd$.
An analog of this element was studied in \cite{MTV} in a situation related to the geometric Langlands correspondence, see element $v_1$ in \cite[Section 8]{MTV}.

For $r<k$ and $m_1,\dots,m_r\in J$, denote
\bean
I_{m_1,\dots,m_r}(z) = \frac{\der^r\{1\}}{\der z_{m_1}\dots\der z_{m_r}}(z).
\eean

\begin{conj}
\label{CB}

The $\sv$-valued function $\{1\}(z)$ satisfies the Gauss-Manin differential equations with parameter
$\kappa=\frac{|a|} k$,
\bean
\label{KZ ka=a}
\frac{|a|}k\frac{\der \{1\}}{\der z_j}(z) = K_j(z) \{1\}(z), \qquad j\in J,
\eean
where the derivatives are defined with respect to a combinatorially flat basis as in \Ref{deriv}.
More generally, for $r<k$ and $m_1,\dots,m_r\in J$,
the $\sv$-valued function $I_{m_1,\dots,m_r}(z)$ satisfies the Gauss-Manin differential equations with parameter
$\kappa=\frac{|a|}{k-r}$,
\bean
\label{KZ ka=ar}
\frac{|a|}{k-r}\,\frac{\der I_{m_1,\dots,m_r}}{\der z_j}(z) = K_j(z) I_{m_1,\dots,m_r}(z), \qquad j\in J.
\eean
\end{conj}

\begin{conj}
\label{CB3}

If we write the $\sv$-valued function $\{1\}(z)$ in coordinates with respect to a combinatorially flat basis, then  $\{1\}(z)$
is a homogeneous polynomial in $z$ of degree $k$.

\end{conj}

The conjectures describes the interrelations of four objects: the identity element in $\ap$, the canonical isomorphism,
the integration isomorphism, and the Gauss-Manin connection on the homology bundle. In the next sections we will
prove this conjecture for families of generic arrangements.

\begin{thm}
\label{thm der one}
If Conjecture \ref{CB} holds, then for $r\leq k$ and $m_1,\dots,m_r\in J$, we have
\bean
 \frac{\der^r\{1\}}{\der z_{m_1}\dots\der z_{m_r}}(z)
 = \frac{k(k-1)\dots(k-r+1)} {|a|^r} \,\al(z) \big(\prod_{i=1}^r \Big[ \frac {a_{m_i}}   {f_{m_i}}\Big]\big).
\eean

\end{thm}
\begin{proof}
The proof is by induction on $r$. For $r=0$, the statement is true: $\{1\}=\{1\}$. Assuming   the statement is true for
some $r$, we prove the statement for $r+1$. By \Ref{KZ ka=ar} and Theorem \ref{K/f}, we have
\bean
&&
\frac{\der^{r+1}\{1\}}{\der z_{m_1}\dots\der z_{m_r}\der z_j}(z) = \frac {k-r}{|a|} K_j(z) \frac{\der^{r}\{1\}}{\der z_{m_1}\dots\der z_{m_r}}(z)
=
\\
\notag
&&
\phantom{aaa}
=\frac {k-r}{|a|} \al(z)\big(\Big[\frac{a_j}{f_j}\Big]\big) *_z \frac{k(k-1)\dots(k-r+1)} {|a|^r} \,\al(z) \big(\prod_{i=1}^r \Big[ \frac {a_{m_i}}   {f_{m_i}}\Big]\big)
=
\\
\notag
&&
\phantom{aaa}
=
 \frac{k(k-1)\dots(k-r+1)(k-r)} {|a|^{r+1}} \al(z) \big(\Big[\frac{a_j}{f_j}\Big]\prod_{i=1}^r \Big[ \frac {a_{m_i}}   {f_{m_i}}\Big]\big).
\eean

\end{proof}

\smallskip
For given $r<k$, the sections $I_{m_1,\dots,m_r}(z), \,m_1,\dots,m_r\in J$, generate a subbundle of the combinatorial bundle.
We will call it the {\it subbundle of conformal blocks} at level $\frac{|a|}{k-r}$ and denote by $CB_{\frac{|a|}{k-r}}$.
The subbundle of conformal blocks at level $\frac{|a|}{k-r}$ is invariant with respect to the Gauss-Manin connection with $\kappa=\frac{|a|}{k-r}$.
On conformal blocks in conformal field theory
see, for example, \cite{FSV1, FSV2, V2} and Section 3.6 in \cite{V4}.

One may show that
\bean
CB_{\frac{|a|}{k}}\subset CB_{\frac{|a|}{k-1}}\subset \dots \subset CB_{\frac{|a|}{1}}.
\eean

\smallskip
Let us consider $\sv$ as a complex manifold. At every point of $\sv$, the tangent space is identified with the vector
space $\sv$. We will consider the manifold $\sv$ with the constant holomorphic metric defined by the contravariant form $S^{(a)}$.
We will denote this metric by the same symbol  $S^{(a)}$.

Define the {\it period map} $q : \cd \to \sv$ by the formula
\bean
\label{P MAP}
q : z \mapsto \{1\}(z).
\eean
The period map is a polynomial map.
Define the {\it potential function of first kind} $P:\cd \to \C$, by the formula
\bean
\label{op function}
P(z) = S^{(a)}(q(z),q(z)).
\eean
The potential function of first kind is a polynomial.

\subsection{Tangent bundle and a Frobenius like structure}
\label{sec diff forms}

Let $T(\C^n-\Delta)\to \C^n-\Delta$ be the tangent bundle on $\C^n-\Delta$. Denote $\der_j=\frac{\der}{\der z_j}$ for $j\in J$.
Consider the morphism  $\beta$ of the tangent bundle to the bundle of algebras defined by the formula,
\bean
\label{tangent map}
\beta(z) \ :\ \der_j \in T_z(\C^n-\Delta)\quad \mapsto\quad \Big[\frac{\der \Phi}{\der z_j}\Big]
= \Big[\frac{a_j}{f_j}\Big]
\in A_\Phi(z) .
\eean
The morphism $\beta$ will be called the {\it tangent morphism}.

The residue form on the bundle of algebras induces a holomorphic bilinear form $\eta$
on fibers of the tangent bundle,
\bean
\label{ETA}
\phantom{aaaaa}
\eta(\der_i, \der_j)_z
&=&
 (\beta(z)(\der_i),\beta(z)(\der_j))_z =  (-1)^k S^{(a)}(\al(z)\beta(z)(\der_i),\al(z)\beta(z)(\der_j))=
\\
\notag
&=&
 \Big(
\Big[\frac{a_i}{f_i}\Big],\Big[\frac{a_j}{f_j}\Big]\Big)_z
 = (-1)^k S^{(a)}\Big(\al(z)\big(\Big[\frac{a_i}{f_i}\Big]\big),\al(z)\big(\Big[\frac{a_j}{f_j}\Big]\big)
\Big).
\eean

\begin{thm}
\label{thm ETAA}
If Conjecture \ref{CB} holds, then the bilinear form $\eta$ is induced by the period map
$q:\cd \to\sv$ from the flat metric $S^{(a)}$ multiplied by $(-1)^k\frac{|a|^2}{k^2}$,
\bean
\eta(\der_i,\der_j)_z = \frac{|a|^2}{k^2} (-1)^k S^{(a)}(\frac{\der q}{\der z_i}(z), \frac{\der q}{\der z_j}(z)).
\eean

\end{thm}

\begin{proof} By Theorem \ref{thm der one},
 we have $\frac{|a|}k \frac {\der q}{\der z_j}=
 \al(z)(\big[\frac{a_j}{f_j}\big])$.
 Hence
\bean
\frac{|a|^2}{k^2} (-1)^k S^{(a)}(\frac{\der q}{\der z_i}, \frac{\der q}{\der z_j}) = (-1)^kS^{(a)}(\al(z)\big(\Big[\frac{a_i}{f_i}\Big]\big),
\al(z)\big(\Big[\frac{a_j}{f_j}\Big]\big)) = \eta(\der_i,\der_j)_z.
\eean
\end{proof}

For $r \leq 2k$, introduce the constant $A_{k,r}$ by the formula
\bean
&&
A_{k,r} = \sum_{i=0}^r{r\choose i}\frac{(k!)^2}{(k-i)! (k-r+i)!},\qquad \text{if}\ r\leq k,
\\
\notag
&&
A_{k,r} = \sum_{i=r-k}^k {r\choose i}\frac{(k!)^2}{(k-i)! (k-r+i)!},\qquad \text{if}\ r> k.
\eean
For example, $A_{2,3} = 24$ and $A_{k,2k}=(2k)!$\,.

\begin{thm}
\label{thm MULTI}
If Conjectures \ref{CB} and \ref{CB3} hold, then for any $r \leq 2k$, we have
\bean
\label{deriP}
(\beta(z)(\der_{m_1})*_z\dots *_z\beta(z)(\der_{m_{r}}), [1](z))_z = \frac{(-1)^k |a|^r}{A_{k,r}} \frac{\der^rP}{\der z_{m_1}\dots\der z_{m_r}}(z),
\eean
for all $m_1,\dots,m_r\in J$. Here $(\,,\,)_z$ is the residue bilinear form on $\ap$.
\end{thm}

\begin{proof}
We have
\bean
&&
(\beta(z)(\der_{m_1})*_z\dots *_z\beta(z)(\der_{m_{r}}),[1](z))_z =
(\prod_{i=1}^r\Big[\frac{a_{m_i}}{f_{m_i}}\Big], [1](z))_z =
\\
\notag
&&
\phantom{aaaaaaaaaaa}
=(-1)^kS^{(a)}(\al(z)\big(\prod_{i=1}^r\Big[\frac{a_{m_i}}{f_{m_i}}\Big]\big), \{1\}(z)).
\eean
Consider the example $r=2, k\geq 2$. Then
\bean
&&
\frac{\der^2}{\der z_{i}\der z_{j}} S^{(a)}(q(z),q(z)) =
S^{(a)}(\frac{\der^2q}{\der z_{i}\der z_{j}},q)
+ S^{(a)}(\frac{\der q}{\der z_{i}},\frac{\der q}{\der z_{j}}) +
\\
\notag
&&
+
S^{(a)}(\frac{\der q}{\der z_{j}},\frac{\der q}{\der z_{i}}) +
S^{(a)}(q,\frac{\der^2q}{\der z_{i}\der z_{j}}) = \frac{k(k-1)}{|a|^2} S^{(a)}(\al(z)\big(\Big[\frac{a_i}{f_i}\Big]\Big[\frac{a_i}{f_i}\Big]\big),q)
+
\\
\notag
&&
+\frac{k^2}{|a|^2} S^{(a)}(\al(z)\big(\Big[\frac{a_i}{f_i}\Big]\big), \al(z)\big(\Big[\frac{a_j}{f_j}\Big]\big))
+ \frac{k^2}{|a|^2} S^{(a)}(\al(z)\big(\Big[\frac{a_j}{f_j}\Big]\big), \al(z)\big(\Big[\frac{a_i}{f_i}\Big]\big)) +
\\
\notag
&&
+\frac{k(k-1)}{|a|^2} S^{(a)}(q,\al(z)\big(\Big[\frac{a_i}{f_i}\Big]\Big[\frac{a_i}{f_i}\Big]\big)) =
\frac{A_{k,2}}{|a|^2} S^{(a)}(\al(z)\big(\Big[\frac{a_i}{f_i}\Big]\Big[\frac{a_i}{f_i}\Big]\big),q).
\eean
Here we used Theorem \ref{thm der one} and the fact that $S^{(a)}$ is constant with respect to the combinatorial connection. This calculation
proves the theorem for  $r=2, k\geq 2$.
The general case for $r\leq k$ is proved exactly in the same way. If $r>k$, then we need to take into account that
$q(z)$ is a polynomial of degree $k$.
\end{proof}

For  $v \in \sv$, define the differential one-form $\psi_v$ on $\cd$ by the formula
\bean
\psi_v : \der_i\in T_z(\cd) \mapsto S^{(a)}(v, \al(z)\beta(z)(\der_i)).
\eean

\begin{thm}
\label{flat co}
 If Conjecture \ref{CB} holds, then the differential form $\psi_v$ is exact,
\bean
\psi_v = \frac{|a|}k dS^{(a)}(v,q(z)).
\eean

\end{thm}

\begin{proof}  By Theorem \ref{thm der one},
 we have $\frac{|a|}k \frac {\der q}{\der z_j} = \al(z)(\big[\frac{a_j}{f_j}\big])$. Hence
 \bea
\psi_v(\der_i) = S^{(a)}(v,\al(z)\beta(z)(\der_i)) = S^{(a)}(v,\al(z)\big(\Big[\frac{a_i}{f_i}\Big]\big))=
S^{(a)}(v,\frac{|a|}k \frac {\der q}{\der z_j}) =  \frac{|a|}k \frac {\der }{\der z_j}
S^{(a)}(v,q).
\eea
\end{proof}

For $\kappa\in\C^\times$, let $I(z)\in\Sing V$ be a flat (multivalued) section of the Gauss-Manin connection
with parameter $\kappa$. Define the (multivalued) differential one-form $\psi_I$ on $\cd$ by the formula
\bean
\psi_I : \der_i\in T_z(\cd) \mapsto S^{(a)}(I(z), \al(z)\beta(z)(\der_i)).
\eean

\begin{thm}
\label{twised per}
If Conjecture \ref{CB} holds and $\kappa\ne\frac{|a|}{k}$, then the differential form $\psi_I$ is exact,
\bean
\label{PTW}
\psi_I = \Big(\frac 1{\kappa} + \frac {k}{|a|}\Big)^{-1} dS^{(a)}(I(z),q(z)).
\eean

\end{thm}

\begin{proof}
We have
\bea
&&
\frac {\der }{\der z_j}S^{(a)}(I(z),q(z))
 = S^{(a)}(\frac {\der }{\der z_j}I(z),q(z))
+ S^{(a)}(I(z),\frac {\der }{\der z_j}q(z)) =
\\
\notag
&&
\phantom{aaa}
=
 S^{(a)}(\frac 1{\kappa}K_j(z)I(z),q(z))
+ S^{(a)}(I(z), \frac {k}{|a|}K_j(z)q(z))) =
\\
\notag
&&
\phantom{aaa}
= \Big(\frac 1{\kappa} + \frac {k}{|a|}\Big) S^{(a)}\Big(I(z), \al(z)\beta(z)\big(\Big[\frac{a_j}{f_j}\Big]\big)\Big).
\eea

\end{proof}

The functions $\cd\to\C, z \mapsto S^{(a)}(v,q(z))$, of Theorem \ref{flat co}
 are nothing else but the coordinate functions of the period map.
We will call them {\it flat periods}. The functions $\cd\to\C, z \mapsto S^{(a)}(I(z),q(z)),$
of Theorem \ref{twised per} will be called {\it twisted periods}.

\begin{conj}
\label{CB4}
There exists a function $\tilde P(z_1,\dots,z_n)$ such that
\bean
\label{dRP}
&&
\\
&&
\notag
\frac{\der^{2k+1}\tilde P}{\der z_{m_0}\dots\der z_{m_{2k}}}(z)
=
(-1)^k
(\beta(z)(\der_{m_0})*_z\dots *_z\beta(z)(\der_{m_{2k}}), [1](z))_z
\eean
for all $m_0,\dots,m_{2k}\in J$.
\end{conj}

The function $\tilde P(z)$ with this property will be called the {\it potential function of second kind}.

Notice that formula \Ref{deriP} does not hold for $r=2k+1$.

The potential function of second kind $\tilde P(z)$ determines the potential function of first kind $P(z)$. Indeed, by formula
\Ref{lem on ONE} we have
\bean
P(z) = \frac 1{|a|^{2k+1}}\sum_{m_0,m_1,\dots,m_{2k}\in J}
 z_{m_0}z_{m_1}\dots z_{m_{2k}}\frac{\der^{2k+1}\tilde P}{\der z_{m_0}\der z_{m_1}\dots\der z_{m_{2k}}}(z).
\eean
More generally, for any $r\leq 2k$, we have
\bean
\frac{\der^{r}P}{\der z_{m_0}\dots\der z_{m_{r-1}}}(z) = \frac {A_{k,r}}{|a|^{2k+1}}\sum_{m_r,\dots,m_{2k}\in J}
 z_{m_r}\dots z_{m_{2k}}\frac{\der^{2k+1}\tilde P}{\der z_{m_0}\dots\der z_{m_{2k}}}(z).
\eean


\smallskip

\medskip

We will call the collection of our objects\,--\, the combinatorial bundle $(\cd)\times(\sv)\to\cd$
with the contravariant form $S^{(a)}$ and connections (combinatorial and Gauss-Manin);
the bundle of algebras $\sqcup_{z\in\C^n-\Delta} A_{\Phi}(z) \to \C^n-\Delta$;
the period map $q: \cd \to \sv$, the potential functions $P(z)$ and $\tilde P(z)$, flat periods $S^{(a)}(v,q(z))$, twisted periods
$S^{(a)}(I(z),q(z))$\,--\, {\it a Frobenius like structure} on $\cd$.

The situation here reminds the structure induced on a submanifold of a Frobenius manifold, cf. \cite{St}. From that point of view
one may expect that $\sing V$ has an honest Frobenius structure and our Frobenius like structure on  $\cd$  is what can be induced from the
Frobenius structure on $\sv$ by the period map, cf. with constructions in \cite{Do, HM}.

Numerous variations of the definition of the Frobenius structure see, for example,  in \cite{D1, D2, M, St, FV}.

\smallskip

In the next sections we will prove Conjectures \ref{CB}, \ref{CB3}, \ref{CB4} for families of generic arrangements and will describe our
structure more precisely.

\section{Points on line}
\label{sec Exa}

\subsection{An arrangement in  $\C^n\times\C$}
\label{An arrangement in C}
Recall that $J=\{1,\dots,n\}$.  Consider $\C$ with coordinate $t$ and
$\C^n$ with coordinates $z_1,\dots,z_n$.
Consider  $n$ linear functions on $\C^n\times\C$,
$f_j = z_j+t,$\ $ j\in J.$
In $\C^n\times \C$ we define
the arrangement $\tilde \A = \{ \tilde H_j\ | \ f_j = 0, \ j\in J \}$.

For every $z=(z_1,\dots,z_n)\in \C^n$ the arrangement $\tilde \A$
induces an arrangement $\A(z)$ in the fiber over $z$ of the projection $\C^n\times\C \to \C^n$. We
identify the fiber with $\C$.
The arrangement $\A(z)$ is  the arrangement of
points $\{-z_1,\dots,-z_n\}$.
Denote $U(\A(z)) = \C - \{-z_1,\dots,-z_n\}$ the complement.

A point $z\in\C^n$ is {\it good} if the points $-z_1,\dots,-z_n$ are distinct.
  Good points form the complement in $\C^n$ to the {\it discriminant} $\Delta$,
  which is the union of hyperplanes
$H_{ij}=\{(z_1,\dots,z_n)\in\C^n\ |\ z_i=z_j\}$ labeled by two-element subsets $\{i,j\} \subset J$.

\subsection{Good fibers}
\label{sec Good fibers C}

For any $z\in\C^n-\Delta$, the space $\OS^1(\A(z))$ has the standard basis
$H_1(z)$, \dots, $H_n(z)$, the space $\FF^1(\A(z))$ has the standard dual basis
$F(H_1(z))$, \dots, $F(H_n(z))$.
For $z^1, z^2\in \C^n-\Delta$, the combinatorial connection identifies the spaces
$\OS^1(\A(z^1))$, $\FF^1(\A(z^1))$   with the spaces
$\OS^1(\A(z^2))$, $\FF^1(\A(z^2))$, respectively, by identifying the corresponding standard bases.

Assume that nonzero weights $(a_j)_{j\in J}$ are given. Then each
arrangement $\A(z)$  is weighted. For $z\in\C^n-\Delta$, the arrangement $\A(z)$ is unbalanced
if $|a|=\sum_{j\in J}a_j\ne 0$.  We
assume  $|a|\ne 0$.

For $z\in\C^n-\Delta$, we denote $V=\FF^1(\A(z))$. We also denote
 $F_j = F(H_j(z))$ for $j\in J$. We have
\bean
\label{sing S}
S^{(a)}(F_i,F_j)=\delta_{ij} a_i,
\qquad
\sing V = \Big\{ \sum_{j\in J} c_jF_j\ | \ \sum_{j\in J} c_ja_j = 0\Big\}.
\eean
For $j\in J$, we define the vector $v_j\in V$ by the formula
\bean
\label{ v_j}
v_j = -F_j+\frac{a_j}{|a|}\sum_{i\in J} F_i.
\eean

\begin{lem}
\label{lem sing}
We have the following properties.

\begin{enumerate}
\item[(i)]
$\dim \Sing V=n-1$.

\item[(ii)]

For $j\in J$, we have $v_j\in\Sing V$ and $\sum_{j\in J}v_j=0$.

\item[(iii)]
Any $n-1$ vectors of $(v_j)_{j\in J}$ are linearly independent.

\item[(iv)] We have
\bean
\label{SF jj}
&&
S^{(a)}(v_j,v_j) = a_j-\frac{a_j^2}{|a|},
\qquad j\in J,
\\
\notag
\label{SF ij}
&&
S^{(a)}(v_i,v_j) = - \frac{a_ia_j}{|a|},
\qquad i,j\in J, \ i\ne j.\
\eean
\end{enumerate}
\qed
\end{lem}

\begin{lem}
\label{lem det}
We have
\bean
\label{deT}
\det_{1\leq i,j\leq n-1} (S^{(a)}(v_i,v_j)) = \frac 1{|a|}\prod_{j\in J} a_j.
\eean
\end{lem}

\begin{proof}
Denote $M$ the transition matrix from the standard basis $F_1,\dots, F_n$ of $V$ to the basis $v_1,\dots,v_{n-1},$ $ \sum_{j\in J} F_j$.
It is easy to see that $\det M=(-1)^{n-1}$. The vector $\sum_{j\in J} F_j$
is orthogonal to $\sing V$ and $S^{(a)}(\sum_{j\in J} F_j,\sum_{j\in J} F_j)= |a|$.
The determinant of $S^{(a)}$ on $V$ with respect to the standard basis $F_1,\dots,F_n$ equals $\prod_{j\in J} a_j$.
These remarks  imply  \Ref{deT}.
\end{proof}

\subsection{Operators  $K_j(z): V\to V$}
\label{sec ham ex C}

For any pair $\{i,j\}\subset J$, we
define the linear operator
$L_{i,j} : V\to V$ by the formula
\bean
F_i \mapsto a_jF_i-a_iF_j, \qquad
F_j \mapsto a_iF_j-a_jF_i, \qquad
F_m \mapsto 0,\quad\on{if}\ m\notin \{i,j\},
\eean
see formula \Ref{L_C}. Define the operators $K_j(z): V\to V$, $j\in J$, by the formula
\bean
\label{K ex C}
K_j(z) = \sum_{i\ne j} \frac{L_{j,i}}{z_j-z_i} ,
\eean
see formula \Ref{K_j}.
For any $j\in J$ and $z\in \C^n-\Delta$, the operator $K_j(z)$ preserves
the subspace $\Sing V\subset V$ and is a symmetric operator, that is
$S^{(a)}(K_j(z)v,w)= S^{(a)}(v, K_j(z)w)$ for all $v,w\in V$, see Theorem \ref{thm K sym}.

\begin{lem}
\label{lem Kv C}
For $j\in J$, we have
\bean
\label{KJ}
K_j(z) v_i &=& \frac {a_j}{z_j-z_i}v_i + \frac {a_i}{z_i-z_j}v_j, \qquad i\ne j,
\\
\notag
K_j(z) v_j &=& -\sum_{i\ne j}K_j(z) v_i .
\eean
\qed

\end{lem}

\begin{cor}
We have $K_j(z) v_i = K_i(z)v_j$ for all $i,j$.
\end{cor}

The differential equations \Ref{dif eqn} for flat sections of the Gauss-Manin connection on
\linebreak
 $(\C^n-\Delta)\times \sing V\to \C^n-\Delta$
take the form
\bean
\label{dif eqn ex 1}
\kappa \frac{\der I}{\der z_j}(z) = K_j(z)I(z),
\qquad
j\in J.
\eean
For generic $\kappa$ all the  flat sections are given by the formula
\bean
\label{I(z)}
I_\gamma(z) =
\sum_{i\in J}
\Big(\int_{\gamma(z)} \prod_{m\in J}(z_m+t)^{a_m/\kappa} \frac {dt}{z_i+t}\Big) F_i ,
\eean
see formula \Ref{Ig}. More precisely, all the flat sections are given by \Ref{I(z)} if
$1+\frac{|a|}{\kappa}\notin \Z_{\leq 0}$ and
$1+\frac{a_j}{\kappa}\notin \Z_{\leq 0}$ for all $j\in J$, see \cite{V1} or Theorem 3.3.5 in \cite{V4}.

\smallskip
Notice that equations \Ref{dif eqn ex 1} are a particular case of the KZ equations, see Section 1.1-1.3 in \cite{V4}.

\subsection{Conformal blocks}

\begin{lem}
\label{lem conf block}
If $\kappa = |a|$, then the Gauss-Manin connection has a one-dimensional invariant subbundle,
generated by the section
\bean
\label{CBl}
q  : z \mapsto \frac 1{|a|}\sum_{j\in J} z_jv_j = \frac 1{|a|} \sum_{j\in J} q_j (z)F_j,
\eean
where
\bean
\label{qi}
q_i(z) = -z_i + \sum_{j\in J}\frac{a_j}{|a|}z_j.
\eean
This section is flat.
\end{lem}

\begin{proof}
The lemma follows from formulas \Ref{KJ}.
\end{proof}

This one-dimensional subbundle will be called the bundle of {\it conformal blocks} at level $|a|$.
A flat section of the subbundle of conformal blocks can be presented as an integral $I_\gamma(z)$,
where $\gamma$ is a small circle around infinity.

\subsection{Canonical isomorphism and period map}
\label{sec master and canonical}
The master function of the arrangement $\A(z)$  is
\bean
\label{def mast example}
\Phi(z,t) = \sum_{j\in J}\,a_j \log f_j = \sum_{j\in J}a_j\log (z_j+t).
\eean
The critical point equation is
$
\frac{\der\Phi}{\der t} = \sum_{j\in J}\frac{a_j}{z_j+t} = 0.
$
The critical set is
\bean
\label{crit z C}
C_\Phi(z) = \big\{ t\in U(\A(z))
\ \big| \ \sum_{j\in J}\frac{a_j}{z_j+t} = 0\big\} .
\eean
The algebra functions on the critical set is
\bean
A_\Phi(z) = \C(U(\A(z)))
/\big\langle \sum_{j\in J}\frac{a_j}{z_j+t} \big\rangle.
\eean
The identity element $[1](z)\in \ap$ equals $\frac 1{|a|}\sum_{j\in J}z_j\big[\frac{a_j}{f_j}\big]$.

\begin{lem}
\label{lem basis crit}
We have $\dim A_\Phi(z) = n-1$. Any $n-1$ elements of
$\big(\big[ \frac{a_j}{z_j+t}\big]\big)_{j\in J}$ are linearly independent.
\qed
\end{lem}

Let $p\in C_{\Phi}(z)$.
The  Grothendieck residue
 $\rho_{p} :  A_{p,\Phi}(z) \to \C$ is given by
\bean
\label{res map ex 1}
f \ \mapsto \ \frac 1{2\pi \sqrt{-1}}\,\Res_{p}
\ \frac{ f}{\frac{\der \Phi}{\der t}}
=\frac{1}{2\pi \sqrt{-1}}\int_{\Gamma_p}
\frac{f dt}{\frac{\der \Phi}{\der t }} ,
\eean
where
$\Gamma_p$ is a small circle around the critical point $p$ oriented clock-wise.
The   residue bilinear form  $(\,,\,)_z$ on $A_\Phi(z)$ is
 $\oplus_{p\in C_\Phi(z)} (\,,\,)_{p}$.

\begin{lem}
\label{lem res infty}
For $f,g\in \C(U(\A(z)))$, we have
\bean
\label{res infty}
([f],[g]) = - \frac 1{2\pi \sqrt{-1}} \Res_{t=\infty} \frac{fg}{\frac{\der \Phi}{\der t }}
 - \frac 1{2\pi \sqrt{-1}} \sum_{i\in J} \Res_{t=-z_i} \frac{fg}{\frac{\der \Phi}{\der t }}.
\eean
\qed
\end{lem}

The canonical element is
\bean
\label{can elt ex}
[E] = \sum_{j\in J} \Big[\frac 1{z_j+t}\Big]\otimes F_j\quad \in\quad A_\Phi(z)\otimes \sing V.
\eean
The canonical isomorphism $\al(z): A_\Phi(z) \to \sing V$ is given by the formula
\bean
\label{can ex p}
\phantom{aaa}
[f]\ \mapsto
- \frac 1{2\pi \sqrt{-1}} \sum_{j\in J}\Big(\Res_{t=\infty} \frac{f}{(z_j+t)\frac{\der \Phi}{\der t }}
+\sum_{i\in J}\Res_{t=-z_i} \frac{f}{(z_j+t)\frac{\der \Phi}{\der t }}
\Big) F_j.
\eean

\begin{thm}
\label{lem can map}
For $k\in J$, we have
\bean
\label{can ex}
\al(z) : \Big[\frac {a_k}{z_k+t}\Big] \mapsto v_k.
\eean

\end{thm}

\begin{proof}
Denote
\bean
g_{kj} = \frac{a_k}{(z_k+t)(z_j+t)}\ \frac 1{\frac{\der \Phi}{\der t }}
=
\frac{a_k}{(z_k+t)(z_j+t)}\ \frac{\prod_{m\in J}(z_m+t)}
{\sum_{m\in J} a_m\prod_{\ell\ne m}(z_\ell+t)}.
\eean
If $k\ne j$, then $\Res_{t=-z_i} g_{jk} = 0$ for all $i\in J$.
If $k=j$, then $\Res_{t=-z_i} g_{jj} = 0$ for  $i\ne j$ and
$\Res_{t=-z_j} g_{jj} = 2\pi \sqrt{-1}$. We also have $\Res_{t=\infty} g_{kj} = -2\pi \sqrt{-1}\frac {a_k}{|a|}$
for all $j\in J$. We obtain the theorem by comparing these formulas with formula
\Ref{ v_j}.
\end{proof}

\begin{cor}
Conjectures \ref{CB} and \ref{CB3} hold for this family of arrangements.

\end{cor}

\begin{proof}
By Theorem \ref{lem can map}, we have $\al(z)([1](z)) = q(z)$, where $q(z)$ is given by \Ref{CBl}. Lemma \ref{lem conf block}
implies Conjectures \ref{CB} and \ref{CB3}.
\end{proof}

\begin{cor}
\label{cor per 1}
For this family of arrangements the {\it period map} $q : \cd \to \sv$ is given by the formula
\bean
\label{pErioD}
q(z) = \frac 1{|a|}\sum_{j\in J} z_jv_j = \frac 1{|a|} \sum_{j\in J} q_j (z)F_j,
\eean
the potential function of first kind is
\bean
\label{pot1k}
P(z) = \frac 1{|a|^2} \sum_{j\in J} a_j q_j^2(z) = \sum_{1\leq i<j\leq n}\frac {a_1a_2}{|a|^3}(z_i-z_j)^2.
\eean

\end{cor}

\smallskip
By Corollary \ref{cor per 1},  the period map extends to a linear map $\C^n \to \Sing V$.
The linear  map is an epimorphism. The kernel is generated by the vector $(1,\dots,1)$.

\medskip
The standard basis $(H_j)_{j\in J}\in V^*$ induces linear functions  on $\sing V$,
 \bean
 \label{H}
h_j :
v_i \mapsto \frac {a_i}{|a|}, \qquad \on{if}\ j\ne i,
\qquad
v_j \mapsto
-1 +\frac {a_j}{|a|}.
 \eean
We have $\sum_{j\in J}a_jh_j=0$ and any $n-1$ of these functions form a basis of $(\sing V)^*$.

 For $i\ne j$, define the hyperplane  $\tilde H_{i,j}\subset \Sing V$ by the equation
 $h_i-h_j=0$.

\begin{lem}
For all $i,j$ we have $q^*(h_i-h_j) = z_j-z_i$  and $q(\Delta) = \cup_{i<j} \tilde H_{i,j}$.
\qed
\end{lem}

\subsection{Contravariant map as the inverse to the canonical map}
The canonical map $\al(z) : A_\Phi(z) \to \sing V$ is the isomorphism described in Theorem \ref{lem can map}.
The contravariant map $\mc S^{(a)} : V\to V^*$ is defined by the formula $F_{i} \mapsto a_i(H_i)$. By identifying
$a_i(H_i)$ with the differential form $\frac{a_i}{f_i}\,dt$ and then projecting the coefficient
to $A_\Phi(z)$ we obtain the map
\bean
[\mc S^{(a)}] : V \to A_\Phi(z), \qquad F_{i} \mapsto \Big[\frac{a_i}{f_i}\Big].
\eean

\begin{thm}
The composition $\al(z)\circ [\mc S^{(a)}] : V\to \sing V$ is the orthogonal projection multiplied by -1.
The composition $[\mc S^{(a)}]\circ \al(z) : A_\Phi(z) \to A_\Phi(z)$ is the identity map multiplied by -1.
\end{thm}

\begin{proof}
The composition $\al(z)\circ [\mc S^{(a)}]$ sends $F_{i}$ to $v_{i}$ which is the orthogonal projection multiplied by -1
 The composition  $[\mc S^{(a)}]\circ \al(z)$ sends $\big[\frac{a_i}{f_i}\big]$ to
\bean
-\Big[\frac{a_i}{f_i}\Big] + \frac{a_i}{|a|}\sum_{j\in J}\Big[\frac{a_j}{f_j}\Big].
\eean
The last sum is zero in $A_\Phi(z)$.
\end{proof}

\subsection{Multiplication on $\sing V$ and $(\sing V)^*$}
\begin{thm}
\label{thm mult on sing}
The canonical isomorphism $\al(z): A_\Phi(z) \to \sing V$ defines an algebra structure on $\sing V$,
\bean
\label{m v}
v_j*_zv_i &=& \frac {a_j}{z_j-z_i}v_i + \frac {a_i}{z_i-z_j}v_j, \qquad i\ne j,
\\
\notag
v_j*_zv_j &=& -\sum_{i\ne j} v_j*_zv_i.
\eean
The element
\bean
\label{IDEN}
\frac 1{|a|}\sum_{j\in J} z_jv_j
\eean
 is the identity element.
\qed
\end{thm}

The isomorphism $S^{(a)}|_{\sing V} : \sing V\to (\sing V)^*$ induces an algebra structure on $(\sing V)^*$.

\begin{lem}
The isomorphism $S^{(a)}|_{\sing V} : \sing V\to (\sing V)^*$ is given by the formula
$v_j\mapsto -{a_j} h_j$ for all $j$.
\end{lem}

\begin{proof} The lemma follows from formulas \Ref{H} and \Ref{SF jj}.
\end{proof}

\begin{cor}
The multiplication on $(\sing V)^*$ is given by the formula
\bean
\label{m h}
h_j*_zh_i &=& \frac {1}{z_i-z_j}h_i + \frac {1}{z_j-z_i}h_j, \qquad i\ne j,
\\
\notag
a_jh_j*_zh_j &=& -\sum_{i\ne j} a_i\,h_i*_zh_j.
\eean
The element
\bean
\label{ID *}
-\frac 1{|a|}\sum_{j\in J} a_jz_jh_j
\eean
 is the identity element.
\qed

\end{cor}

\subsection{Tangent morphism}

The tangent morphism $\beta$ of the tangent bundle $T(\C^n-\Delta)\to \C^n-\Delta$ to the bundle of algebras
$\sqcup_{z\in\C^n-\Delta} A_{\Phi}(z) \to \C^n-\Delta$ is given by the formula \Ref{tangent map},
\bean
\label{tang ex}
\beta(z) \ :\ \der_j \in T_z(\C^n-\Delta)\quad \mapsto\quad \Big[\frac{\der\Phi}{\der z_j}\Big]= \Big[\frac{a_j}{z_j+t}\Big]
\in A_\Phi(z) .
\eean

\begin{lem}
\label{lem tan map}
The map $\beta(z)$ is an epimorphism. The kernel of $\beta(z)$
is generated by the vector  $\sum_{j\in J}\der_j$.
\qed
\end{lem}

The residue form on the bundle of algebras induces a holomorphic symmetric bilinear form $\eta$ on $T(\C^n-\Delta)$,
see formula \Ref{ETA}. The bilinear form $\eta$ has rank $n-1$. Its kernel is generated by the vector  $\sum_{j\in J}\der_j$.

\begin{lem}
We have
\bean
\label{eta jj}
&&
\eta(\der_j, \der_j) = -a_j+\frac{a_j^2}{|a|},
\qquad j\in J,
\\
\notag
&&
\eta(\der_i, \der_j) =  \frac{a_ia_j}{|a|},
\qquad i,j\in J, \ i\ne j.\
\eean

\end{lem}

\begin{proof}
The lemma follows from Lemmas \ref{lem sing}, \ref{lem can map}
and Theorem \ref{thm alpha}. It can be checked also by a straightforward calculation.
\end{proof}

\subsection{Multiplication and potential function of second kind}
Let us define the multiplication on fibers of  $T(\C^n-\Delta)$ by the formulas
\bean
\label{MULT}
\der_i *_z \der_j & = &\frac {a_i}{z_i-z_j}\der_j
+ \frac {a_j}{z_j-z_i}\der_i,
\\
\notag
\der_i *_z \der_i &=& -\sum_{j\ne i} \der_i *_z \der_j,
\eean
cf. formula 5.25 in \cite{D2}.
The vector  $\sum_{i\in J} \der_i$ has zero product with everything.

\begin{lem}
For every $z\in \C^n-\Delta$, the morphism $\beta(z)$ defines an algebra epimorphism of $T_z(\C^n-\Delta)$ to
$A_\Phi(z)$, in particular, $\beta(z)(v) *_z \beta(z)(w) = \beta(z)(v*_z w)$  for all $v,w\in T_z(\C^n-\Delta)$.
\qed
\end{lem}

Consider   the ideal of
$T_z(\C^n-\Delta)$ generated by $\sum_{j\in J} \der_j$. Denote $B(z)$ the quotient algebra.
The morphism $\beta(z)$ induces an isomorphism $B(z)\simeq A_\Phi(z)$.

\begin{lem}
The element $\frac 1{|a|}\sum_{i\in J} z_j\der_j$ projects to the identity element of  $B(z)$.
 \qed
\end{lem}

The bilinear form $\eta$ defines a morphism $\tilde \eta$ of the tangent bundle $T(\C^n-\Delta)$ to the cotangent
bundle $T^*(\C^n-\Delta)$. For $j\in J$, denote $p_j(z)=a_j q_j(z)= a_j(-z_j + \sum_{i\in J}\frac{a_i}{|a|}z_i)$. We have $\sum_{j\in J}p_j=0$.

\begin{lem}
\label{lem te}
The morphism $\tilde \eta$ is given by the formula $\der_j \mapsto dp_j$ for all $j$.
The kernel of $\tilde \eta$ is generated by the vector $\sum_{j\in J}\der_j$.
\qed

\end{lem}

Consider the span of  differential one-forms $(dp_j)_{j\in J}$. This span equals the span of differential one-forms
 $(dz_i-dz_j)_{1\leq i< j\leq n}$.
The spans in the fibers define the subbundle
\bean
\label{B*}
\sqcup_{z\in\C^n-\Delta} B^*(z) \to\C^n-\Delta
\eean
 of the cotangent bundle
$T^*(\C^n-\Delta)$. The subbundle has rank $n-1$.

\begin{lem}
\label{leMma}
The form $\eta$ induces the algebra structure on $B^*(z)$ given by the formula
\bean
\label{p mult}
dp_i*_zdp_j
 & = &   \frac {a_i}{z_i-z_j}dp_j + \frac {a_j}{z_j-z_i}dp_i = \frac{a_ia_j} {z_i-z_j}d(z_i-z_j) ,
 \qquad i\ne j,
\\
\notag
dp_i*_zdp_i &=& -\sum_{j\ne i} dp_j *_z dp_i,
\eean
and the bilinear form
\bean
\label{p form}
&&
(dp_j, dp_j) = -a_j+\frac{a_j^2}{|a|},
\qquad j\in J,
\\
\notag
&&
(dp_i, dp_j) =  \frac{a_ia_j}{|a|},
\qquad i,j\in J, \ i\ne j.\
\eean
\qed
\end{lem}

Introduce the {\it potential function of second kind}
\bean
\label{pot}
\tilde P(z) = \frac 12 \sum_{1\leq i< j\leq n} a_ia_j\,(z_i-z_j)^2\log(z_i-z_j).
\eean

\begin{thm}
\label{thm pot}
We have
\bean
\label{WDW}
d\,\Big(\frac{\der^2\tilde P}{\der z_i\der z_j}\Big)\ =\ -\tilde \eta (\der_i) *_z \tilde \eta (\der_j)
\eean
for all $i,j$.

\end{thm}

\begin{proof}
The theorem follows from Lemma \ref{leMma}.
\end{proof}

Notice that equation \Ref{WDW} is the definition (3.5) in \cite{D2} of the potential function of an almost dual Frobenius structure.

The right hand side in \Ref{WDW} can be rewritten:
$
\tilde \eta (\der_i) *_z \tilde \eta (\der_j) = \tilde \eta (\der_i *_z \der_j).
$
For all $i,j,k$, we have
\bean
\tilde \eta (\der_i *_z \der_j)(\der_\ell) =  \eta (\der_i *_z \der_j, \der_\ell)
=(\beta(z)(\der_i) *_z \beta(z)(\der_j)*_z\beta(z)(\der_\ell), [1](z))_z,
\eean
where $(\,,\,)_z$ is the residue form on $A_\Phi(z)$.
Formula \Ref{WDW} says that for all $i,j,k$, we have
\bean
\label{WDW 3}
\frac{\der^3\tilde P}{\der z_i\der z_j\der_\ell}(z)= -
(\beta(z)(\der_i) *_z \beta(z)(\der_j)*_z\beta(z)(\der_\ell), [1](z))_z.
\eean
Hence Conjecture \ref{CB4} holds for this family of arrangements.

\subsection{Connections on the
bundle $\sqcup_{z\in\C^n-\Delta} B^*(z) \to\C^n-\Delta$ defined in \Ref{B*}}
\label{CONNE}

The combinatorial and Gauss-Manin connections on $(\C^n-\Delta)\times
\sing V\to \C^n-\Delta$ induce the combinatorial and Gauss-Manin connections on  bundle \Ref{B*}.

\begin{lem}
\label{p flat}
The differential one-forms $(dp_j)_{ j\in J}$ are flat sections of the combinatorial connection on  bundle \Ref{B*}.

\end{lem}

\begin{proof}
The vectors $v_j\in\Sing V$ give flat sections of the combinatorial connection on $\C^n\times \sing V \to\C^n$.
By formula \Ref{can ex}, the elements $ \Big[\frac {a_j}{z_j+t}\Big] \in A_\Phi(z)$ give flat sections of bundle
of algebras. Now formula \Ref{tang ex} and Lemma \ref{lem te} imply Lemma \ref{p flat}.
\end{proof}

Let $I(z) = \sum_{j\in J}I^j(z) dp_j$ be a  section of  bundle \Ref{B*}. For $i\in J$, we denote
$\frac{\der I}{\der z_i} =  \sum \frac{\der I^j}{\der z_i}dp_j$.

\begin{lem}
\label{lem GM B}

The differential equations for flat sections of the Gauss-Manin connection take the from
\bean
\kappa \frac {\der I}{\der z_i} = dp_i *_z I, \qquad i\in J,
\eean
see formula \Ref{p mult}. For generic $\kappa$ all the flat sections are given by the formula
\bean
\label{flaT}
I_{\gamma,\kappa} (z) =
\sum_{j\in J}
\Big(\int_{\gamma(z)} \prod_{i\in J}(z_i+t)^{a_i/\kappa} \frac {dt}{z_j+t}\Big) dp_j ,
\eean
where $\gamma(z) \in H_1(U(\A(z)), \mc L_\kappa\vert_{U(\A(z))})$ is a flat section of the Gauss-Manin connection
on $\sqcup_{z\in \C^n-\Delta} H_k(U(\A(z)), \mc L_\kappa\vert_{U(\A(z))})\to \C^n-\Delta$.

\end{lem}

\begin{proof}
The lemma follows from Theorem  \ref{thm ham normal} and formula \Ref{I(z)}.
\end{proof}

\begin{thm}
\label{THM TP}
For every flat section $I_{\gamma, \kappa}$, we have
$I_{\gamma, \kappa}  = -\kappa\,d p_{\gamma, \kappa}$ where
\bean
p_{\gamma,\kappa} = \int_{\gamma(z)} \prod_{i\in J}(z_i+t)^{a_i/\kappa} dt .
\eean
\end{thm}

\begin{proof}
The theorem follows from two formulas:
\bean
\kappa \frac{\der p_{\gamma,\kappa}}{\der z_j} = \int_{\gamma(z)} \prod_{i\in J}(z_i+t)^{a_i/\kappa}\,\frac {a_j dt}{z_j+t}
\eean
and $\sum_{j\in J}\frac{\der p_{\gamma,\kappa}}{\der z_j}=0$.
\end{proof}

Following Dubrovin \cite{D1, D2}, we will call the functions $p_{\gamma,\kappa}$
{\it twisted periods}. Notice that this definition agrees with the definition of twisted periods in Section \ref{sec diff forms},
namely,  the twisted periods of Theorem \ref{THM TP} can be also defined   by formula \Ref{PTW} of Theorem \ref{twised per}.

\begin{lem}
Given $\kappa\in \C^\times$, let $I(z,\kappa)$ be a flat section of the Gauss-Manin connection with the parameter $\kappa$.
Let $I(z,-\kappa)$ be a flat section of the Gauss-Manin connection with the parameter $-\kappa$. Then $(I(z,\kappa), I(z,-\kappa))_z$
does not depend on $z\in\C^n-\Delta$.
\qed
\end{lem}

\begin{proof}
The lemma follows from the fact that $B^*(z)$ is a Frobenius algebra.
\end{proof}

\subsection{Functoriality}

In this section we will discuss how our objects extend to strata of the discriminant $\Delta\subset\C^n$.

\subsubsection{}
A stratum $X$ of $\Delta$ is given by a partition $(J_1,\dots,J_m)$ of $J$,
\bean
\label{X}
X=\{(z_1,\dots,z_n)\in\C^n\ | \ z_i-z_j=0 \quad \text{for}\ i,j\in J_\ell,\ \ell = 1,\dots,m\},
\eean
$\dim X=m$. The coordinates on $X$ are functions $x_1,\dots,x_m$  where $x_\ell = z_j$ for $j\in J_\ell$.
Let $\iota: X \hookrightarrow  \C^n$ be the natural embedding. Then
\bean
\label{der der y}
\iota_* : \frac{\der}{\der x_\ell}\mapsto \sum_{j\in J_\ell} \frac{\der}{\der z_j},
\qquad
\iota^*: dz_j\mapsto dx_\ell  \quad \text{if} \ j\in J_\ell.
\eean
The remaining strata of $\Delta$ cut on $X$ the union of hyperplanes $x_i=x_j$, $1\leq i<j\leq m$ which we denote
$\Delta_X$.
For $\ell=1,\dots,m$, we denote $b_\ell =\sum_{j\in J_\ell} a_j$. We assume that $b_\ell\ne 0$ for all $\ell$.

We restrict our family of arrangements $\A(z), z\in \C^n$, to $X-\Delta_X$. For $x\in X-\Delta_X$ the corresponding
arrangement $\A(x)$ consists of points $-x_1,\dots,-x_m$ of weights $b_1,\dots,b_m$, respectively. For this new family
we will construct all the objects described in Sections \ref{An arrangement in C}-\ref{CONNE}
and relate them to the objects constructed for the arrangements $\A(z), z\in \C^n-\Delta$.
The objects corresponding to the new family will be provided with the index $X$.

\subsubsection{}
For $x\in X-\Delta_X$, the space $(V_X)^*=\OS^1(\A(x))$ has the standard basis $(H_{\ell,X}), \ell=1,\dots,m$.
Recall that the
space $V^*$ of Sections \ref{An arrangement in C}-\ref{CONNE} has the standard basis $(H_j), j\in J$. We have the canonical epimorphism
\bean
f^*: V^*\to (V_X)^*,
\qquad
(H_j) \mapsto (H_{\ell,X})\quad\text {if} \quad j \in J_\ell.
\eean
The space $V_X=\FF^1(\A(x))$ has the standard basis $F_{\ell,X}, \ell=1,\dots,m$.
The space $V$ of Sections \ref{An arrangement in C}-\ref{CONNE} has the standard basis $F_j, j\in J$. We have the canonical
embedding
\bean
f: V_X\hookrightarrow V,\qquad
F_{\ell,X} \mapsto \sum_{j\in J_\ell}F_j.
\eean
 The subspace of singular vector is defined
by the formula
\bean
\sing V_X = \Big\{ \sum_{\ell=1}^m c_\ell F_{\ell,X}\ | \ \sum_{\ell=1}^m b_\ell c_\ell = 0 \Big\}.
\eean
We have $f(\sing V_X) = f(V_X)\cap (\sing V)$. Consider the embedding
\bean
\tilde f : \sing V_X \hookrightarrow \sing V, \qquad v\mapsto f(v).
\eean
 For the contravariant form on $V_X$ we have
\bean
S^{(b)}_X(F_{\ell,X}, F_{k,X}) = S^{(a)}(f(F_{\ell,X}), f(F_{k,X}))=\delta_{\ell,k} b_k.
\eean
For $\ell=1,\dots,m$, we define a vector $v_{\ell,X}\in \sing V_X$ by the formula
\bean
v_{\ell, X} = -F_{\ell,X}+\frac{b_\ell}{|a|}\sum_{k=1}^m F_{k,X}.
\eean
 We have
$f: v_{\ell, X} \mapsto  \sum_{j\in J_\ell} v_j$.

The standard basis $(H_{\ell,X}),\ell=1,\dots,m$, induces linear functions  on $\sing V_X$,
\bean
h_{\ell,X} \ :\  v_{k,X}\ \mapsto\ \frac {b_k}{|a|} \quad\text{ if}\quad k\ne\ell, \qquad v_{\ell,X} \
\mapsto\
-1 +\frac {b_\ell}{|a|}.
\eean
 We have $\tilde f^* : h_j\mapsto h_{\ell,X}$ if $j\in J_\ell$.

\subsubsection{} For $\ell=1,\dots,m$ and $x\in X-\Delta_X$, the operators $K_{\ell,X}(x):\sing V_X\to\sing V_X$ are defined
by formulas \Ref{KJ},
\bean
K_{\ell,X}(x) v_{k,X} &=& \frac {b_\ell}{x_\ell-x_k}v_{k,X} + \frac {b_k}{x_k-x_\ell}v_{\ell,X}
\quad\text{ for}\quad \ell\ne k,
\\
\notag
K_{\ell,X}(x) v_{\ell,X} &=& - \sum_{k\ne\ell} K_{\ell,X}(x) v_{k,X}.
\eean
 For all $\ell,k$, we have
\bean
\label{fK}
f(K_{\ell,X}(x) v_{k,X}) = \sum_{j\in J_\ell} K_j(x) f(v_{k,X}).
\eean
Notice that the right hand side in \Ref{fK} is well-defined despite the fact that $K_j(x)v_i$
is not well-defined for all $v_i$, see formula \Ref{KJ}.

\subsubsection{} Multiplication on $\sing V_X$ is defined by formulas \Ref{m v},
\bean
v_{\ell,X}*_{x,X} v_{k,X} &=& \frac {b_\ell}{x_\ell-x_k}v_{k,X} + \frac {b_k}{x_k-x_\ell}v_{\ell,X}
\quad\text{ for}
\quad \ell\ne k,
\\
\notag
v_{\ell,X}*_{x,X} v_{\ell,X} &=&
 - \sum_{k\ne\ell} v_{\ell,X}*_{x,X} v_{k,X}.
 \eean
For all $\ell,k$, we have
\bean
\label{f mult}
f(v_{\ell,X} *_{x,X} v_{k,X}) = f(v_{\ell,X}) *_x f(v_{k,X}).
\eean
Notice that the right hand side in \Ref{f mult} is well-defined despite the fact that $v_i *_x v_j$
is not well-defined for all $v_i, v_j$, see formula \Ref{m v}.

\subsubsection{}
The multiplication on $(\sing V_X)^*$ is given by formula \Ref{m h},
\bean
\label{m h X}
h_{\ell,X}*_{x,X}h_{k,X} &=& \frac {1}{x_\ell-x_k}h_{\ell,X} + \frac {1}{z_k-z_\ell}h_{k,X}, \qquad \ell\ne k,
\\
\notag
b_\ell h_{\ell,X}*_{x,X}h_{k,\ell}  &=& -\sum_{k\ne \ell} b_k\,h_{k,X}*_{x,X}h_{\ell,X}.
\eean
If $\ell\ne k$, $i\in J_\ell, j\in J_k$, then
\bean
\label{f mult *}
\tilde f^*(h_{i} *_x h_{j}) = h_{\ell,X} *_{x,X} h_{k,X}.
\eean
Notice that $h_{i} *_x h_{j}$ is well-defined despite the fact that $h_i *_x h_j$
is not well-defined for all $h_i, h_j$, see formula \Ref{m v}.

\subsubsection{}
For $x\in X-\Delta_X$ the residue form on $A_\Phi(x)$
 induces a holomorphic bilinear form $\eta_X$ on $T_x(X-\Delta_X)$,
\bean
\label{eta jj X}
&&
\eta_X\Big(\frac{\der}{\der x_\ell}, \frac{\der}{\der x_\ell}\Big) = -b_\ell+\frac{b_\ell^2}{|a|},
\qquad j\in J,
\\
\notag
&&
\eta_X\Big(\frac{\der}{\der x_\ell}, \frac{\der}{\der x_k}\Big) =  \frac{b_\ell b_k}{|a|},
\qquad \ \ell\ne k.\
\eean
For all $\ell,k$, have
\bean
\eta_X\Big(\frac{\der}{\der x_\ell}, \frac{\der}{\der x_k}\Big) =
\eta\Big(\sum_{i\in J_\ell}\frac{\der}{\der z_i}, \sum_{j\in J_\ell}\frac{\der}{\der z_j}\Big).
\eean
For $\ell=1,\dots,m$, we define a linear function on $X$,
\bean
q_{\ell,X}(x) = - x_\ell + \sum_{k\ne\ell}\frac{b_k}{|a|}x_k.
\eean
We have
\bean
q_{\ell,X}(x) = q_i(x), \qquad \text{if}\ i\in J_\ell.
\eean
The period map $q_X: X-\Delta_X \to \sing V_X$ is defined by formula \Ref{pErioD},
\bean
\label{period map X}
q_X(x) = \frac 1{|a|}\sum_{\ell=1}^m q_{\ell,X} (x) F_{\ell,X} = \frac 1{|a|} \sum_{\ell=1}^m x_\ell\, v_{\ell,X}.
\eean
\begin{thm}
For all $x\in X$, we have
\bean
f(q_X(x)) = q(x)
\eean
and for the potential functions of first kind we have
\bean
P_X(x) = P(x).
\eean
\qed
\end{thm}

\subsubsection{}
For $x\in X-\Delta_X$, the multiplication on $T_x(X-\Delta_X)$ is defined by  formulas \Ref{MULT},
\bean
\label{MULT X}
\frac{\der}{\der x_\ell} *_{x,X} \frac{\der}{\der x_k} & = &
\frac {b_\ell}{x_\ell-x_k}\frac{\der}{\der x_k}
+ \frac {b_k}{x_k-x_\ell}\frac{\der}{\der x_\ell},
\\
\notag
\frac{\der}{\der x_\ell} *_{x,X} \frac{\der}{\der x_\ell} &=& -\sum_{k\ne \ell} \
\frac{\der}{\der x_k} *_{x,X} \frac{\der}{\der x_\ell}.
\eean
For all $\ell, k$, we have
\bean
\label{MULT X *}
\frac{\der}{\der x_\ell} *_{x,X} \frac{\der}{\der x_k}  = \Big(\sum_{i\in J_\ell}
\frac{\der}{\der z_i}\Big)  *_x \Big(\sum_{j\in J_k} \frac{\der}{\der z_j}\Big).
\eean
Notice that the right hand side in \Ref{MULT X *} is well-defined despite the fact that
$\frac{\der}{\der z_i} *_x \frac{\der}{\der z_j}$ is not well-defined for all $i,j$, see \Ref{MULT X}.

\subsubsection{}
For $\ell=1,\dots,m$, we denote $p_{\ell,X}=b_\ell q_{\ell,X}$. The map
\bean
\tilde \eta_X : T_x(X-\Delta_X)
\to T^*_x(X-\Delta_X),
\quad
 w\mapsto \eta_X(w,\cdot),
\eean
 sends $\frac{\der}{\der x_\ell} $ to $dp_{\ell,X}$.
 We have
 \bean
 dp_{\ell,X} = \iota^*\Big(\sum_{i\in J_\ell} dp_i\Big) .
 \eean
Denote $B^*_X(x)$ the span of $(dp_{\ell,X})_{\ell=1}^m$ in
$T^*_x(X-\Delta_X)$.
This span equals the span of differential forms
 $(dx_\ell-dx_k)_{1\leq \ell< k\leq m}$.
 The multiplication on $B^*_X(x)$ is given by the formulas \Ref{p mult},
\bean
\label{p mult X}
\phantom{aaaaaa}
dp_{\ell,X} *_{x,X} dp_{k,X}
 & = &
 \frac {b_\ell}{x_\ell-x_k} dp_{k,X} + \frac {b_k}{x_k-x_\ell} dp_{\ell,X}
  = \frac {b_\ell b_k}{x_\ell-x_k}d(x_\ell-x_k),
 \qquad \ell \ne k,
\\
\notag
dp_{\ell,X}*_{x,X}dp_{\ell,X}
 &=& -\sum_{k\ne \ell}\ dp_{k,X}*_{x,X} dp_{\ell,X} .
\eean
For all $\ell, k$, we have
\bean
\iota^*\Big(\sum_{i\in J_\ell} dp_i\Big) *_{x,X} \iota^*\Big(\sum_{j\in J_k} dp_j\Big)
=
\iota^*\Big(\sum_{i\in J_\ell}\sum_{j\in J_k} dp_i *_{x}  dp_j\Big).
\eean
The potential function of second kind
is defined by formula \Ref{pot},
\bean
\label{pot X}
\tilde P_X(x_1,\dots,x_m) = \frac 12 \sum_{1\leq \ell< k\leq m} b_\ell b_k\,(x_\ell-x_k)^2\log(x_\ell-x_k).
\eean
By formula \Ref{WDW},
\bean
\label{WDW X}
d\,\Big(\frac{\der^2\tilde P_X}{\der x_\ell\der x_k}\Big)\ =\ -\tilde \eta_X \Big(\frac{\der}{\der x_\ell}\Big) *_{x,X} \tilde \eta_X \Big(\frac{\der}{\der x_k}\Big)
\eean
for all $\ell,k$. For all $\ell, k$, we have
\bean
\frac{\der^2\tilde P_X}{\der x_\ell\der x_k}(x) = \lim_{z\to x\atop z\in \C^n-\Delta}\sum_{i\in J_\ell}\sum_{j\in J_k}\frac{\der^2\tilde P}{\der z_i\der z_j}(z).
\eean

\subsection{Frobenius like structure}

Consider the quotient $M$ of $\C^n$ by the one-dimensional subspace $\C(1,\dots,1)$ and the natural projection $\pi: \C^n\to M$.
Then all our objects
\,--\, the combinatorial bundle $(\cd)\times(\sv)\to\cd$
with the contravariant form $S^{(a)}$ and connections (combinatorial and Gauss-Manin);
the bundle of algebras $\sqcup_{z\in\C^n-\Delta} A_{\Phi}(z) \to \C^n-\Delta$;
the period map $q: \cd \to \sv$, potential functions $P(z)$ and $\tilde P(z)$, flat periods $p_j(z)$, twisted periods
$p_{\gamma,\kappa}(z)$\,--\, descend to  the quotient and form on $M-\pi(\Delta)$ a structure which we will also call a {\it Forbenius like structure}.

In particular, the functions  $p_1,\dots,p_{n-1}$ will form a coordinate systems on $M$ and $\eta$ will induce a holomorphic
metric on $M$ constant with respect to the coordinates $p_1,\dots,p_{n-1}$. If $y_i = \sum_{j=1}^{n-1}c_{j,i}p_j$,
$j=1,\dots,n-1$, is a linear change of coordinates with $c_{i,j}\in\C$ such that $\eta = \sum_{j=1}^{n-1} dy_j^2$.
Then equation \Ref{WDW} will take the form
\bean
\label{WDWV}
d\,\Big(\frac{\der^2\tilde P}{\der y_i\der y_j}\Big)\ =\ - dy_i *_y dy_j
\eean
for all $i,j$. The functions $-\frac{\der^3\tilde P}{\der y_i\der y_j\der y_k}$ will become the structure constants of the
multiplication on $T^*_y(M-\pi(\Delta))$ and the potential function $\tilde P$ will satisfies the WDVV equations with respect to the coordinates
$y_1,\dots,y_{n-1}$.

\section{Generic lines on plane}
\label{sec lines on plane}

\subsection{An arrangement in  $\C^n\times\C^2$}
\label{An arrangement in c2}
Consider $\C^2$ with coordinates $t_1,t_2$,\
$\C^n$ with coordinates $z_1,\dots,z_n$. Fix $n$
 linear functions on $\C^2$,
$g_j=b_j^1t_1+ b_j^2t_2,$\ $ j\in J,$
$b_j^i\in \C$. We assume that
\bean
d_{i,j} = \det \left( \begin{array}{clcr}
b^1_i  & b^2_i
\\
b^1_j  & b^2_j
\end{array} \right) \ne 0
\qquad \text{for all}\ i\ne j.
\eean
We define $n$ linear functions on $\C^n\times\C^2$,
$f_j = z_j+g_j,$\ $ j\in J.$
In $\C^n\times \C^2$ we define
 the arrangement $\tilde \A = \{ \tilde H_j\ | \ f_j = 0, \ j\in J \}$.

For every $z=(z_1,\dots,z_n)$ the arrangement $\tilde \A$
induces an arrangement $\A(z)$ in the fiber of the projection $\C^n\times\C^2\to\C^n$
over $z$. We
identify every fiber with $\C^2$. Then $\A(z)$ consists of
lines $H_j(z), j\in J$, defined in $\C^2$ by the  equations
$f_j=0$.
Denote $ U(\A(z)) = \C^2 - \cup_{j\in J} H_j(z)$,  the complement to the arrangement $\A(z)$.

The arrangement $\A(z)$ is with normal crossings if and only if $z \in \C^n-\Delta$, where
$\Delta = \cup_{1\leq i<j<k\leq n}H_{i,j,k}$ and the hyperplane $H_{i,j,k}$ is defined by the equation
$f_{i,j,k}=0$,
\bean
\label{ fijk}
f_{i,j,k}=d_{j,k}z_i+d_{k,i}z_j+d_{i,j}z_k.
\eean

\begin{lem}
\label{Pluk}
For any four distinct indices $i,j,k,\ell$ we have
\bean
&&
\phantom{aaaaaaaaa}
\frac{f_{i,j,k}}{d_{k,i}d_{i,j}} + \frac{f_{i,k,\ell}}{d_{\ell,i}d_{i,k}} + \frac{f_{i,\ell,j}}{d_{j,i}d_{i,\ell}} = 0,
\\
&&
\frac{f_{i,j,k}^2}{d_{i,j}d_{j,k}d_{k,i}}-\frac{f_{j,k,\ell}^2}{d_{j,k}d_{k,\ell}d_{\ell,j}} + \frac{f_{k,\ell,i}^2}{d_{k,\ell}d_{\ell,i}d_{i,k}}-
\frac{f_{\ell,i,j}^2}{d_{\ell,i}d_{i,j}d_{j,\ell}}=0.
\eean
\qed
\end{lem}

\subsection{Good fibers}
\label{sec Good ex2}

For any $z\in\C^n-\Delta$, the space $\OS^2(\A(z))$ has the standard basis
\linebreak
$(H_i(z),H_j(z))$,  $1\leq i<j\leq n$. The space $\FF^2(\A(z))$ has the standard dual basis
$F(H_i(z),H_j(z))$, $1\leq i<j\leq n$.
For $z^1, z^2\in \C^n-\Delta$, the combinatorial connection identifies the spaces
$\OS^2(\A(z^1))$, $\FF^2(\A(z^1))$   with the spaces
$\OS^2(\A(z^2))$, $\FF^2(\A(z^2))$, respectively, by identifying the standard bases.

Assume that nonzero weights $(a_j)_{j\in J}$ are given. Then the
arrangement $\A(z)$  is weighted. For $z\in\C^n-\Delta$, the arrangement $\A(z)$ is unbalanced
if $|a|\ne 0$.  We assume that $|a|\ne 0$.

For $z\in\C^n-\Delta$, we denote $V=\FF^2(\A(z))$, $V^*=(\FF^2(\A(z))^*=\OS^2(\A(z))$,
 $F_{i,j} = F(H_i(z),H_j(z))$. We have $F_{i,j}=-F_{j,i}$,
\bean
\label{sing S 2}
S^{(a)}(F_{i,j},F_{k,\ell}) &=& 0,
\qquad \text{if} \ i<j,\ k<\ell\ \text{and}\ (i,j)\ne(k,\ell),
\\
\notag
S^{(a)}(F_{i,j},F_{i,j}) &=& a_i a_j,
\eean
\bean
\sing V
& =&
\Big\{ \sum_{1\leq i<j\leq n} c_{i,j}F_{i,j}\ | \ \sum_{j=1}^{i-1} a_jc_{j,i}
- \sum_{j=i+1}^{n} a_jc_{i,j} = 0, \ i=1,\dots,n\Big\}.
\eean
By Corollary \ref{cor nondeg}, the restriction of $S^{(a)}$ to $\sing V$ is nondegenerate.
Denote $(\sing V)^\perp$ the orthogonal complement to $\sing V$ with respect to $S^{(a)}$. Then
$V = \sing V \oplus (\sing V)^\perp$. Denote $\pi : V\to \sing V$ the orthogonal projection.

\begin{lem}
\label{elm orth}
The space $(\sing V)^\perp$ is generated by vectors
\bean
\label{wj C two}
\sum_{i\in J} F_{i,j}, \qquad j\in J.
\eean
\qed
\end{lem}

For $i\ne j$, we define the vector $v_{i,j}\in V$ by the formula
\bean
\label{ v ij}
v_{i,j} = F_{i,j} -\frac{a_j}{|a|}\sum_{k\in J} F_{i,k}
- \frac{a_i}{|a|}\sum_{\ell\in J} F_{\ell,j}.
\eean
We have $v_{i,j}=-v_{j,i}$. Set $v_{i,i}=0$.

\begin{lem}
\label{lem siNg}
We have the following properties.

\begin{enumerate}
\item[(i)]
$\dim \Sing V= {n-1\choose 2}$.

\item[(ii)]

We have $v_{i,j}\in\Sing V$ and $v_{i,j} = \pi(F_{i,j})$.

\item[(iii)]
For $j\in J$, we have $\sum_{i\in J} v_{i,j}=0$.
\item[(iv)]
For any $k\in J$, the set $v_{i,j}$, $ 1\leq i<j\leq n$, $k\notin\{i,j\}$, is a
basis of $\sing V$.

\end{enumerate}
\qed
\end{lem}

\begin{lem} We have
\bean
\label{SF jji}
&&
S^{(a)}(v_{i,j},v_{k,\ell}) = 0,
\qquad \text{if}\ i,j,k,\ell \ \text{are distinct},
\\
\notag
\label{SF ijii}
&&
S^{(a)}(v_{i,j},v_{i,k}) = -\frac{a_ia_ja_k}{|a|},
\qquad \text{if}\ i,j,k \ \text{are distinct},
\\
\notag
\label{SF ijiii}
&&
S^{(a)}(v_{i,j},v_{i,j}) = - \sum_{k\ne j}
S^{(a)}(v_{i,j},v_{i,k}) = a_ia_j - \frac{a_ia_j(a_i+a_j)}{|a|}.
\eean
\qed
\end{lem}

\subsection{Operators  $K_i(z) : V\to V$}
\label{sec ham ex}

For any  subset $\{i,j,k\}\subset J$, we
define the linear operator
$L_{i,j,k} : V\to V$ by the formula
\bean
\label{def L}
F_{i,j}
& \mapsto &
a_kF_{i,j}+a_iF_{j,k}+a_jF_{k,i},
\\
\notag
F_{j,k}
& \mapsto &
a_kF_{i,j}+a_iF_{j,k}+a_jF_{k,i},
\\
\notag
F_{k,i}
& \mapsto &
a_kF_{i,j}+a_iF_{j,k}+a_jF_{k,i},
\\
\notag
F_{\ell,m}
& \mapsto & 0,
\qquad
\text{if}\ \{\ell,m\} \ \text{is not a subset of}\ \{i,j,k\}.
\eean
see formula \Ref{L_C}. Notice that $L_{i,j,k}$ does not depend on the order of $i,j,k$.

 We define the operators $K_i(z): V\to V$, $i\in J$, by the formula
\bean
\label{K ex CC}
K_i(z) = \sum \frac{d_{j,k}}{f_{i,j,k}}\,L_{i,j,k} ,
\eean
where the sum is over all unordered subsets $\{j,k\}\subset J-\{i\}$,
see formula \Ref{K_j}.
For any $i\in J$ and $z\in \C^n-\Delta$, the operator $K_i(z)$ preserves
the subspace $\Sing V\subset V$ and is a symmetric operator,
$S^{(a)}(K_i(z)v,w)= S^{(a)}(v, K_i(z)w)$ for all $v,w\in V$, see Theorem \ref{thm K sym}.

\begin{lem}
\label{lem Kv CC}
For $i\in J$, we have
\bean
\label{KJ 2}
K_i(z) v_{j,k} &=& \frac{d_{j,k}}{f_{i,j,k}}\,(a_iv_{j,k}+ a_jv_{k,i}+a_kv_{i,j}),
\qquad \text{if}\ i\notin\{j,k\},
\\
K_i(z) v_{j,i} &=& - \sum_{k\notin\{i,j\}} K_i(z) v_{j,k}.
\eean
\end{lem}

\begin{proof}
The restriction of $K_i(z)$ to $(\sing V)^\perp$ is zero by formula \Ref{def L} and Lemma \ref{elm orth}.
We have
\bean
\label{KJ e}
K_i(z) v_{j,k} = K_i(z) F_{j,k}= \frac{d_{j,k}}{f_{i,j,k}}\,(a_iF_{j,k}+ a_jF_{k,i}+a_kF_{i,j}).
\eean
The right hand side in \Ref{KJ e} equals the right hand side in \Ref{KJ 2}
by Lemma \ref{lem siNg}.
\end{proof}

The differential equations \Ref{dif eqn} for flat sections of the Gauss-Manin connection on
\linebreak
 $(\C^n-\Delta)\times \sing V\to \C^n-\Delta$
take the form
\bean
\label{dif eqn ex}
\kappa \frac{\der I}{\der z_j}(z) = K_j(z)I(z),
\qquad
j\in J.
\eean
For generic $\kappa$ all the  flat sections are given by the formula
\bean
\label{I(z) 2}
I_\gamma(z) =
\sum_{1\leq i<j\leq n}
\Big(\int_{\gamma(z)} \prod_{m\in J}f_m^{a_m/\kappa} \frac {d_{i,j}}{f_if_j}\,dt_1\wedge dt_2\Big) F_{i,j} ,
\eean
see formula \Ref{Ig}. These generic $\kappa$ can be determined more precisely from the determinant formula in \cite{V1}.

\subsection{Conformal blocks}
Define the map $q : \C^n \to \sing V$ by the formula
\bean
\label{CB2}
q  : z \mapsto -\frac 1{|a|^2}\sum_{1\leq i<j\leq n} \frac{(z_ib^1_j-z_jb^1_i)^2}{d_{i,j}b^1_ib^1_j}\,v_{i,j} .
\eean

\begin{lem}
\label{ q inv}
For any $k\in J$, we have
\bean
\label{ij ne k}
q(z) = \frac 1{|a|^2} \sum' \frac{ f_{i,j,k}^2}{ d_{i,j}d_{j,k}d_{k,i}}\,v_{i,j},
\eean
where the sum is over all pairs $i<j$ such that $k\notin\{i,j\}$.
\end{lem}

\begin{proof}
Denote $A_{i,j}=-\frac{(z_ib^1_j-z_jb^1_i)^2}{d_{i,j}b^1_ib^1_j}$, then
\bean
\label{qAv}
q(z) = \frac 1{|a|^2} \sum_{1\leq i<j\leq n} A_{i,j}v_{i,j}.
\eean
We replace in \Ref{qAv} each $v_{i,k}$ with $-\sum_{j\ne k} v_{i,j}$ and
each $v_{k,j}$ with $-\sum_{i\ne k} v_{i,j}$. Then
\bean
\label{qAv 1}
q(z) = \frac 1{|a|^2} \sum' (A_{i,j}+A_{j,k}+A_{k,i})v_{i,j},
\eean
where the sum is the same as in \Ref{ij ne k}. The lemma follows from the identity
\bean
A_{i,j}+A_{j,k}+A_{k,i} =  \frac{ f_{i,j,k}^2}{ d_{i,j}d_{j,k}d_{k,i}}.
\eean
\end{proof}

By Lemma \ref{ q inv}, the map $q$ can be defined in terms of the determinants $d_{i,j}$, $ 1\leq i< j\leq n$, only
without using  the individual numbers $b^1_i$.

\begin{thm}
\label{thm c 2}
If $\kappa = |a|/2$, then the Gauss-Manin connection on $(\C^n-\Delta)\times (\sing V)\to \C^n-\Delta$
has a one-dimensional invariant subbundle,
generated by the section $q: z\mapsto q(z)$, see \Ref{ij ne k}. This section is flat.
\end{thm}

This one-dimensional subbundle will be called the bundle of {\it conformal blocks} at level $|a|/2$.

\begin{proof}
We check that
\bean
\frac {|a|} 2\frac{\der q}{\der z_1}(z) = K_1(z) q(z).
\eean
The other  differential equations are proved similarly.
By Lemma \ref{ q inv}   we have
\bean
q(z) =  \frac 1{|a|^2} \sum_{1< i<j\leq n} \frac{ f_{1,i,j}^2}{ d_{i,j}d_{j,1}d_{1,i}}\,v_{i,j}.
\eean
Then
\bean
\frac {|a|} 2\frac{\der q}{\der z_1}(z) = \frac 1{|a|} \sum_{1< i<j\leq n} \frac{ f_{1,i,j}}{ d_{j,1}d_{1,i}}\,v_{i,j}
\eean
and
\bean
K_1(z) q(z) = \frac 1{|a|^2} \sum_{1< i<j\leq n} \frac{ f_{1,i,j}}{ d_{j,1}d_{1,i}}\,(a_1v_{i,j}+ a_iv_{j,1}+ a_jv_{1,i})
\eean
by formula \Ref{KJ 2}. By replacing $v_{j,1}$ with $-\sum_{i\ne 1}v_{j,i}$ and $v_{1,i}$ with $-\sum_{j\ne 1} v_{i,j}$
and using Lemma \ref{Pluk}
we obtain
\bean
K_1(z) q(z) =    \frac 1{|a|} \sum_{1< i<j\leq n} \frac{ f_{1,i,j}}{ d_{j,1}d_{1,i}}\,v_{i,j}.
\eean

\end{proof}

\subsection{Algebra $A_\Phi(z)$}
\label{sec master 2}

The master function of the arrangement $\A(z)$  is
\bean
\label{def mast 2}
\Phi(z,t) = \sum_{j\in J}\,a_j \log f_j = \sum_{j\in J}a_j\log (z_j+ b^1_jt_1+b^2_jt_2).
\eean
The critical point equations are
\bean
\frac{\der\Phi}{\der t_1} = \sum_{j\in J} a_j\frac{b^1_j}{f_j} = 0,
\qquad
\frac{\der\Phi}{\der t_2} = \sum_{j\in J} a_j\frac{b^2_j}{f_j} = 0.
\eean
Introduce $H_i$, $i=1,2,$ by the formula
\bean
\frac{\der\Phi}{\der t_i} = \frac{H_i}{\prod_{j\in J} f_j}.
\eean
We have
\bean
\label{HH}
t_1\frac{\der\Phi}{\der t_1}+ t_2\frac{\der\Phi}{\der t_2}
=
 |a| - \sum_{i\in J} z_i \frac{a_i}{f_i}.
\eean
In other words, we have
\bean
t_1H_1+t_2H_2
=
 |a|\prod_{j\in J} f_j - \sum_{i\in J} z_i a_i \prod_{j\ne i} f_j.
\eean
The critical set is
\bean
\label{crit z CC}
\phantom{aaa}
C_\Phi(z) = \Big\{ t\in U(\A(z))
 \big|   \frac{\der\Phi}{\der t_1}= 0,  \frac{\der\Phi}{\der t_2}=0 \Big\} =
\big\{ t\in U(\A(z))
 \big|   H_1= 0, H_2=0 \big\} .
\eean
The algebra of functions on the critical set is
\bean
A_\Phi(z) = \C(U(\A(z)))
/\Big\langle \frac{\der\Phi}{\der t_1}, \frac{\der\Phi}{\der t_2} \Big\rangle =
\C(U(\A(z)))
/\big\langle H_1, H_2 \big\rangle.
\eean

\begin{lem}
\label{lem basis crit CC}
We have $\dim A_\Phi(z) = {n-1\choose 2}$.
\qed
\end{lem}

Introduce elements $w_{i,j} \in A_\Phi(z)$ by the formula
\bean
w_{i,j}= a_ia_j\,\Big[\frac{d_{i,j}}{f_if_j}\Big].
\eean

\begin{lem}
\label{w rel}

We have $w_{i,j}= - w_{j,i}$ and
$
\label{w relation}
\sum_{j\in J}w_{i,j}=0.
$

\end{lem}

\begin{proof}
The lemma follows from the identity
\bean
a_i\frac{df_i}{f_i} \wedge \frac{d\Phi}{\Phi} = \sum_{j\in J} a_ia_j\,\frac{d_{i,j}}{f_if_j} dt_1\wedge dt_2 +
\sum_{m\in J} dz_j\wedge \mu_j,
\eean
where $\mu_j$ are suitable one-forms.
\end{proof}

The elements $\big[ \frac{a_i}{f_i}\big]$, $i\in J$, generate $A_\Phi(z)$ by Lemma \ref{lem f_j generate}.

\begin{lem}
For $j\in J$, we have the following identity in $A_\Phi(z)$,
\bean
\label{a/f k=2}
\sum_{i\in J} d_{i,j} \Big[\frac{a_i}{f_i}\Big] =0.
\eean
\qed
\end{lem}

\begin{lem}
We have
\bean
\label{umn}
\Big[\frac{a_i}{f_i}\Big]*_z\Big[\frac{a_j}{f_j}\Big] &=& \frac 1{d_{i,j}} \,w_{i,j}, \qquad i\ne j,
\\
\notag
\Big[\frac{a_j}{f_j}\Big]*_z\Big[\frac{a_j}{f_j}\Big]
&=&
\sum_{i\notin\{j,k\}}
\frac {d_{k,i}}{d_{j,k}d_{i,j}} w_{i,j}, \qquad k\ne j.
\eean
\bean
\label{MuL}
\Big[ \frac{a_i}{f_i}\Big]*_z w_{j,k}
 &=& \frac{d_{j,k}}{f_{i,j,k}}\,(a_iw_{j,k}+ a_jw_{k,i}+a_kw_{i,j}),
\qquad \text{if}\ i\notin\{j,k\},
\\
\notag
\Big[ \frac{a_i}{f_i}\Big]*_z w_{j,i} &=& - \sum_{k\notin\{i,j\}} \Big[ \frac{a_i}{f_i}\Big]*_z w_{j,k}.
\eean
\qed
\end{lem}

\begin{cor}
The elements $w_{i,j}$, $1\leq i<j\leq n$, span $A_\Phi(z)$.
\end{cor}

\begin{lem}
\label{lem 1 in A}
The identity element $[1](z) \in A_\Phi(z)$ satisfies the equations
\bean
\label{1^2}
&&
\phantom{aaaaaqaaaa}
\\
\notag
&&
[1](z) = \frac 1{|a|} \sum_{i\in J} z_i\Big[ \frac{a_i}{f_i}\Big] =
 \frac 1{|a|^2} \Big(\sum_{i\in J} z_i\Big[\frac{a_i}{f_i}\Big]\Big)^2
 = -\frac 1{|a|^2} \sum_{1\leq i<j\leq n} \frac{(z_ib^1_j-z_jb^1_i)^2}{d_{i,j}b^1_ib^1_j}\,w_{i,j} .
 \eean
\end{lem}

\begin{proof}
To obtain the last expression in \Ref{1^2} we replace $\Big[\frac{a_i}{f_i}\Big]\Big[\frac{a_i}{f_i} \Big]$ with
$\Big[-\frac{a_i}{f_i}\sum_{j\ne i}\frac{b^1_j}{b^1_i}\frac{a_j}{f_j} \Big]$.
\end{proof}

\begin{thm}
\label{lem 1 1 in A}
For any $k\in J$, the identity element $[1](z) \in A_\Phi(z)$ satisfies the equation
\bean
\label{1^2 1}
[1](z) = \frac 1{|a|^2}\sum' \frac{ f_{i,j,k}^2}{ d_{i,j}d_{j,k}d_{k,i}}\,w_{i,j},
\eean
where the sum is over all pairs $i<j$ such that $k\notin\{i,j\}$.
\end{thm}

\begin{proof}
The proof is the same as the proof of Lemma \ref{ q inv}.
\end{proof}

The canonical element is
\bean
\label{can elt 2}
[E] = \sum_{1\leq i<j\leq n} \Big[\frac {d_{i,j}}{f_if_j}\Big]\otimes F_{i,j} \quad \in\quad A_\Phi(z)\otimes \sing V.
\eean

\begin{thm}
\label{thm can 2}
The canonical isomorphism
\bean
\al(z) : A_\Phi(z) \to \sing V
\eean
is given by the formula
\bean
w_{i,j} \mapsto v_{i,j} .
\eean
\end{thm}

\begin{cor}
\label{cor 1}
We have $\al(z)[1]) =  q(z)$, where $q(z)$ is the conformal block of Theorem \ref{thm c 2}.
\end{cor}

\subsection{Proof of Theorem \ref{thm can 2} }
 Introduce the coefficients $B_{i,j}$ by the formula
\bean
\al(z)(w_{k,\ell}) = \sum_{1\leq i<j\leq n} B_{i,j} F_{i,j}.
\eean
We have
\bean
\label{B}
-\frac{4\pi^2}{a_ka_\ell d_{k,\ell}d_{i,j}} B_{i,j} = \sum_{p\in C_\Phi(z)}
\Res_p \frac 1{f_kf_\ell f_i f_j} \frac {\prod_{m\in J}f_m^2}{H_1H_2}.
\eean

\begin{lem}
\label{lem1}
We have $B_{i,j} = 0$, if $\{i,j\}\cap\{k,\ell\}=\emptyset$.
\end{lem}

\begin{proof}

The differential form
\bean
\omega_{k,\ell,i,j} = \frac {\prod_{m\in J}f_m^2}{f_kf_\ell f_i f_j H_1H_2}\,dt_1\wedge dt_2
\eean
has poles only on the curves $H_1=0$ and $H_2=0$. The poles are of first order. To calculate the right hand side in
\Ref{B}, we need to take the residue $\Res_{H_1=0}\omega_{k,\ell,i,j}$ of the form $\omega_{k,\ell,i,j}$ at the curve $H_1=0$
and then take the residue of that form on the curve $H_1=0$ at the points where $H_2=0$. This is the same as if we
took with minus sign the residue of $\Res_{H_1=0}\omega_{k,\ell,i,j}$ on the curve $H_1=0$ at infinity. That residue at infinity
with minus sign could be obtained differently in two steps. First we may take the residue $\Res_{\infty}\omega_{k,\ell,i,j}$ of $\omega_{k,\ell,i,j}$
at the line at infinity and then take the residue of that one-form on the line at infinity  at the points where $H_1=0$.

So to calculate the right hand side in \Ref{B} we first calculate $\Res_{\infty}\omega_{k,\ell,i,j}$.
The coordinates at infinity are $u_1=t_1/t_2$, $u_2= 1/t_2$. We have $f_m=(b^1_mu_1 + b^2_m+u_2z_m)/u_2$. Denote $\tilde f_m (u_1) = b^1_mu_1 + b^2_m $.
For $i=1,2$, we have $H_i(u_1/u_2,1/u_2) = \hat H_i(u_1,u_2)/u_2^{n-1}$, where $\hat H_i(u_1,u_2)$ are some polynomials.
Denote $\tilde H_i(u_1) = \hat H_i(u_1,0)$.
We have $dt_1\wedge dt_2 = -\frac 1{u_2^3} du_1\wedge du_2$. Then the residue of $\omega_{k,\ell,i,j} $ at the line at infinity equals
\bean
2\pi\sqrt{-1} \,\tilde \omega_{k,\ell,i,j} = \frac {\prod_{m\in J}\tilde f_m(u)^2}{\tilde f_k(u)\tilde f_\ell(u) \tilde f_i(u) \tilde f_j(u)\tilde H_1(u)\tilde H_2(u)}\,du,
\eean
where $u=u_1$. On the line at infinity this one-form is holomorphic at $u=\infty$.
The number $\frac{2\pi\sqrt{-1}}{a_ka_\ell d_{k,\ell}d_{i,j}} B_{i,j}$ equals the sum of  residues of the form $\tilde \omega_{k,\ell,i,j}$ at the points where $\tilde H_1(u)=0$.

By formula \Ref{HH}, we have
\bean
u\tilde H_1(u)+ \tilde H_2(u) = |a| \prod_{m \in J}\tilde f_m(u).
\eean
Thus $\tilde H_2(u) = |a| \prod_{m \in J}\tilde f_m(u)$ at the point where $H_1(u)=0$. Therefore,
the sum of residues of the form $\tilde \omega_{k,\ell,i,j}$ at the points where $\tilde H_1(u)=0$ equals the sum of residues of the form
\bean
\tilde \omega_{k,\ell,i,j} =  \frac {\prod_{m\notin \{k,\ell,i,j\}}\tilde f_m(u)}{|a| \,\tilde H_1(u)}\,du
\eean
at the points where $\tilde H_1(u)=0$. This sum is zero.
\end{proof}

\begin{lem}
\label{two}
We have $B_{k,j} = -\frac{a_\ell}{|a|}$, if $j\notin \{k,\ell\}$.
\end{lem}

\begin{proof}
On the line at infinity we consider the differential one form
\bean
\tilde \omega_{k,\ell,i,j} = \frac {\prod_{m\in J}\tilde f_m(u)^2}{\tilde f_k^2(u)\tilde f_\ell(u) \tilde f_j(u)\tilde H_1(u)\tilde H_2(u)}\,du,
\eean
As in Lemma \ref{lem1} we observe that $\frac{2\pi\sqrt{-1}}{a_ka_\ell d_{k,\ell}d_{k,j}} B_{k,j}$ equals the sum of
residues of that one-form at the points where $H_1=0$. Consider the differential one-form
\bean
\mu = \frac {\prod_{m\notin \{k,\ell,j\}}\tilde f_m(u)}{|a|\,\tilde f_k(u)\tilde H_1(u)}\,du.
\eean
As in Lemma \ref{lem1} we observe that $\frac{2\pi\sqrt{-1}}{a_ka_\ell d_{k,\ell}d_{k,j}} B_{k,j}$ equals the sum of residues
of $\mu$ at the points where $\tilde H_1(u)=0$ and this sum equals
\bean
-\Res_{\tilde f_k=0} \mu = -\frac{2\pi\sqrt{-1}}{a_k|a|d_{k,\ell}d_{k,j}}.
\eean
The lemma is proved.

\end{proof}

By Lemmas \ref{lem1} and \ref{two} we know that $\al(z)(w_{k,\ell}) = B_{k,\ell} F_{k,\ell}- \frac{a_k}{|a|} \sum_{i\ne k} F_{i,\ell}
- \frac{a_\ell}{|a|} \sum_{j\ne \ell} F_{k,j}$. From the condition that $\al(z)(w_{k,\ell})\in \sing V$
we conclude that $B_{k,\ell}=\frac{|a|-a_k-a_\ell}{|a|}$. The theorem is proved.

\subsection{Contravariant map as the inverse to the canonical map}
The canonical map $\al(z) : A_\Phi(z) \to \sing V$ is the isomorphism described in Theorem \ref{thm can 2}.
The contravariant map $\mc S^{(a)} : V\to V^*$ is defined by the formula $F_{i,j} \mapsto a_ia_j(H_i,H_j)$. By identifying
$a_ia_j(H_i,H_j)$ with the differential form $a_ia_j\frac{d_{i,j}}{f_if_j}\,dt_1\wedge dt_2$ and then projecting the coefficient
to $A_\Phi(z)$ we obtain the map
\bean
[\mc S^{(a)}] : V \to A_\Phi(z), \qquad F_{i,j} \mapsto w_{i,j} = a_ia_j\Big[\frac{d_{i,j}}{f_if_j}\Big].
\eean

\begin{thm}
The composition $\al(z)\circ [\mc S^{(a)}] : V\to \sing V$ is the orthogonal projection.
The composition $[\mc S^{(a)}]\circ \al(z) : A_\Phi(z) \to A_\Phi(z)$ is the identity map.
\end{thm}

\begin{proof}
The composition $\al(z)\circ [\mc S^{(a)}]$ sends $F_{i,j}$ to $v_{i,j}$ which is the orthogonal projection
by Lemma \ref{lem siNg}. The composition  $[\mc S^{(a)}]\circ \al(z)$ sends $w_{i,j}$ to
\bean
w_{i,j}- \frac{a_i}{|a|} \sum_{k\in J} w_{k,j}
- \frac{a_j}{|a|} \sum_{\ell\in J} w_{i,\ell}.
\eean
The last two sums are equal to zero in $A_\Phi(z)$ by Lemma \ref{w rel}.
\end{proof}

\subsection{Corollaries of Theorem \ref{thm can 2}  }

\begin{thm}
For any $j\in J$, the $\sv$-valued function $\frac{\der q}{\der z_j}(z)$ satisfies the Gauss-Manin
differential equations with $\kappa=|a|$,
\bean
|a| \frac{\der }{\der z_i}\frac{\der q}{\der z_j}(z) = K_i(z) \frac{\der q}{\der z_j}(z), \qquad i\in J.
\eean

\end{thm}

\begin{proof}
By Theorems \ref{lem 1 1 in A} and \ref{thm can 2}, we have
\bea
q(z) = \al(z)([1](z)) = \frac 1{|a|^2} \al(z) \big((\sum_{m\in J}z_m\Big[\frac{a_m}{f_m}\Big])^2\Big) =
\frac 1{|a|^2} \sum_{m,\ell\in J}z_mz_\ell \al(z)\big(\Big[\frac{a_m}{f_m}\Big]\Big[\frac{a_\ell}{f_\ell}\Big]\big).
\eea
By Theorem \ref{thm can 2}, for any $m,\ell\in J$, the element $\al(z)(\big[\frac{a_m}{f_m}\big]\big[\frac{a_\ell}{f_\ell}\big])\in\sv$
is a linear combination of vectors $v_{i,j}$ with constant coefficients. Hence
\bea
\frac{\der^2 q}{\der z_j\der z_i}(z)= \frac 2{|a|^2} \al(z)\big(\Big[\frac{a_i}{f_i}\Big]\Big[\frac{a_j}{f_j}\Big]\big),
\eea
\bea
\frac{\der q}{\der z_j}(z)= \frac 2{|a|}\al(z)\big(\Big[\frac{a_j}{f_j}\Big] *_z \frac 1{|a|}\sum_{m\in J}z_m\Big[\frac{a_m}{f_m}\Big]\big) =
\frac 2{|a|}\al(z)\big(\Big[\frac{a_j}{f_j}\Big]\big),
\eea
\bea
K_i(z)\frac{\der q}{\der z_j}(z)=   \frac 2{|a|}\al(z)\big(\Big[\frac{a_i}{f_i}\Big]\Big[\frac{a_j}{f_j}\Big]\big).
\eea
This implies the theorem.
\end{proof}

\begin{cor}
Conjectures \ref{CB} and \ref{CB3} hold for this family of arrangements.
\qed
\end{cor}

The tangent morphism $\beta$ and the residue form on the bundle of algebras induce a holomorphic bilinear form $\eta$
on fibers of the tangent bundle,
\bean
\label{ETA 2}
\phantom{aaaaa}
\eta(\der_i, \der_j)_z
&=&
 (\beta(z)(\der_i),\beta(z)(\der_j))_z =  (-1)^k S^{(a)}(\al(z)\beta(z)(\der_i),\al(z)\beta(z)(\der_j))=
\\
\notag
&=&
 \Big(
\Big[\frac{a_i}{f_i}\Big],\Big[\frac{a_j}{f_j}\Big]\Big)_z
 = (-1)^k S^{(a)}\Big(\al(z)\big(\Big[\frac{a_i}{f_i}\Big]\big),\al(z)\big(\Big[\frac{a_j}{f_j}\Big]\big)
\Big).
\eean
By Theorem \ref{thm ETAA}, we have
\bean
\eta(\der_i,\der_j)_z = \frac{|a|^2}{4} S^{(a)}(\frac{\der q}{\der z_i}(z), \frac{\der q}{\der z_j}(z)).
\eean
Theorems \ref{flat co} and \ref{twised per} also hold for this family of arrangements.

\begin{thm}
Recall the potential function of first kind $P(z) = S^{(a)}(q(z), q(z))$. We have
\bean
\label{pot k=2}
P(z) =  \sum_{1\leq i<j<k\leq n} \frac{a_ia_ja_k}{ |a|^5}\, \frac{f_{i,j,k}^4}{d_{i,j}^2d_{j,k}^2d_{k,i}^2}.
\eean
\end{thm}

\begin{proof}
By Theorem \ref{thm MULTI}, for any $r\leq 4$, we have
\bean
\label{deriP 2}
(\beta(z)(\der_{m_1})*_z\dots *_z\beta(z)(\der_{m_{r}}), [1](z))_z = \frac{ |a|^r}{A_{2,r}} \frac{\der^rP}{\der z_{m_1}\dots\der z_{m_r}}(z),
\eean
for all $m_1,\dots,m_r\in J$. In particular,
\bean
\label{deriP 21}
(\beta(z)(\der_k)*_z\beta(z)(\der_{\ell})*_z\beta(z)(\der z_m), [1](z))_z = \frac{ |a|^3}{4!} \frac{\der^3P}{\der z_k\der z_\ell \der z_m}(z)
\eean
for all $k,\ell,m\in J$.  Introduce the function
\bean
\label{tP}
\hat P(z) = \frac 1{4!} \sum_{1\leq i<j<k\leq n} \frac{a_ia_ja_k}{ |a|^2}\, \frac{f_{i,j,k}^4}{d_{i,j}^2d_{j,k}^2d_{k,i}^2}.
\eean

\begin{prop}
\label{lem tP}
We have
\bean
\label{tP 21}
(\beta(z)(\der_j)*_z\beta(z)(\der_{\ell})*_z\beta(z)(\der z_m), [1](z))_z = \frac{\der^3\hat P}{\der z_j\der z_\ell \der z_m}(z)
\eean
for all $k,\ell,m\in J$.
\end{prop}

\begin{proof}

For $k,\ell\in J$, we define the differential one-form $\psi_{k,\ell}$ on $\C^n-\Delta$ by the formula
\bean
\psi_{k,\ell}(\der_m) =  ( \beta(z)(\der_k)*_z \beta(z)(\der_\ell), \beta(z)(\der_m))_z.
\eean
The canonical isomorphism identifies the residue
form and the contravariant form and therefore we may write
\bean
\psi_{k,\ell}(\der_m) =  S^{(a)}( \al(z)\beta(z)(\der_k)*_z\al(z)\beta(z)(\der_\ell), \al(z)\beta(z)(\der_m)).
\eean

\begin{lem}

The form $\psi_{k,\ell}$ is the differential of the function
\bean
\phi_{k,\ell}(z)=\frac {|a|}{2} \frac 1{d_{k,\ell}} S^{(a)}(v_{k,\ell},q(z)),
\eean
if $k\ne \ell$, and of the function
\bean
\phi_{k,\ell}(z)=\frac {|a|}{2}\sum_{i\notin\{j,k\}}\frac{d_{j,i}}{d_{k,j}d_{i,k}} S^{(a)}( v_{i,k}, v),
\eean
if $k=\ell$, where $j$ is any number in $J$ such that $j\ne k$.
\end{lem}

\begin{proof}
The vector $\al(z)\beta(z)(\der_k)*_z\al(z)\beta(z)(\der_\ell) = \al(z)(\big[\frac{a_k}{f_k}\big]\big[\frac{a_\ell}{f_\ell}\big])\in \sv$
equals $\frac 1{d_{k,\ell}} v_{k,\ell}$ if $k\ne\ell$ and equals $\sum_{i\notin\{j,k\}}\frac{d_{j,i}}{d_{k,j}d_{i,k}}  v_{i,k}$ if $k=\ell$.
We also have $\al(z)\beta(z)(\der_m) = \frac {|a|}2 \frac{\der q}{\der z_m}$. This implies the lemma.
\end{proof}

Proposition \ref{lem tP} is equivalent to the formula
\bean
\label{M2}
\frac{\der^2\hat P}{\der z_k\der z_\ell} = \phi_{k,\ell}
\eean
for all $k,\ell\in J$.
The proof  of  \Ref{M2} is by direct verification. Namely, assume that $k<\ell$. Then
\bean
\label{k<l P2}
&&
\\
\notag
&&
\frac{\der^2\hat P}{\der z_k\der z_\ell} = \frac{a_ia_ka_\ell}{2|a|^2} \Big(
\sum_{i<k} \frac{f_{i,k,l}^2}{d_{i,k}d_{k,\ell}d_{\ell,i}}\frac 1{d_{k,\ell}}
+\sum_{k<i<\ell} \frac{f_{k,i,l}^2}{d_{k,i}d_{i,\ell}d_{\ell,k}}\frac 1{d_{\ell,k}}
+\sum_{i>\ell} \frac{f_{k,\ell,i}^2}{d_{k,\ell}d_{\ell,i}d_{i,k}}\frac 1{d_{k,\ell}}\Big).
\eean
We also have
\bean
\label{k<l p2}
\phantom{aa}
&&
\\
\notag
\phi_{k,\ell}
&=& \frac 1{2|a|}S^{(a)}\Big(\frac 1{d_{k,\ell}}v_{k,\ell}, \sum_{i<j  \atop k\notin\{i,j\}} \frac{f_{i,j,k}^2}{d_{i,j}d_{j,k}d_{k,i}}
v_{i,j}\Big)
=
\\
\notag
&=&
\frac 1{2|a|d_{k,\ell}}S^{(a)}\Big(v_{k,\ell},
\sum_{i<k} \frac{f_{i,\ell,k}^2}{d_{i,\ell}d_{\ell,k}d_{k,i}}v_{i,\ell}
+ \sum_{k<i<\ell} \frac{f_{i,\ell,k}^2}{d_{i,\ell}d_{\ell,k}d_{k,i}}v_{i,\ell}
+ \sum_{i>\ell} \frac{f_{\ell,i,k}^2}{d_{\ell,i}d_{i,k}d_{k,\ell}}v_{\ell,i}\Big)
\\
\notag
&=&
\frac {a_ia_\ell a_k}{2|a|^2d_{k,\ell}}\Big(-
\sum_{i<k} \frac{f_{i,\ell,k}^2}{d_{i,\ell}d_{\ell,k}d_{k,i}}
- \sum_{k<i<\ell} \frac{f_{i,\ell,k}^2}{d_{i,\ell}d_{\ell,k}d_{k,i}}
+ \sum_{i>\ell} \frac{f_{\ell,i,k}^2}{d_{\ell,i}d_{i,k}d_{k,\ell}}\Big).
\eean
Comparing \Ref{k<l P2} and \Ref{k<l P2} we conclude that \Ref{M2} holds if $k<\ell$.
Assume that $k=\ell$. Then
\bean
\label{k P2}
\frac{\der^2\hat P}{\der z_k^2} =  \sum_{i<j  \atop k\notin\{i,j\}} \frac{a_ia_ja_k}{2|a|^2}\frac{f_{i,j,k}^2}{d^2_{j,k}d^2_{k,i}}
\eean
We also have
\bean
\label{k p2}
\phi_{k,\ell}
&=& \frac 1{2|a|}S^{(a)}\Big(\sum_{m\notin\{k,\ell\}}  \frac {d_{\ell,m}}{d_{k,\ell}d_{m,k}}v_{m,k}, \sum_{i<j  \atop k\notin\{i,j\}} \frac{f_{i,j,k}^2}{d_{i,j}d_{j,k}d_{k,i}}
v_{i,j}\Big)
=
\\
\notag
&=&
\sum_{i<j  \atop k\notin\{i,j\}} \frac{1}{2|a|}\frac{f_{i,j,k}^2}{d_{i,j}d_{j,k}d_{k,i}}
S^{(a)}\Big(\frac {d_{\ell,i}}{d_{k,\ell}d_{i,k}}v_{i,k} +\frac {d_{\ell,j}}{d_{k,\ell}d_{j,k}}v_{j,k}, v_{i,j}\Big)=
\\
\notag
&=&
\sum_{i<j  \atop k\notin\{i,j\}} \frac{a_ia_ja_k}{2|a|^2d_{k,\ell}}\frac{f_{i,j,k}^2}{d_{i,j}d_{j,k}d_{k,i}}\Big(
-\frac {d_{\ell,i}}{d_{i,k}} +\frac {d_{\ell,j}}{d_{j,k}}\Big)=
\\
\notag
&=&
\sum_{i<j  \atop k\notin\{i,j\}} \frac{a_ia_ja_k}{2|a|^2d_{k,\ell}}\frac{f_{i,j,k}^2}{d_{i,j}d_{j,k}d_{k,i}}
\frac {d_{\ell,k}d_{i,j}}{d_{i,k}d_{j,k}}=\sum_{i<j  \atop k\notin\{i,j\}} \frac{a_ia_ja_k}{2|a|^2}\frac{f_{i,j,k}^2}{d^2_{j,k}d^2_{k,i}} .
\eean
Comparing \Ref{k P2} and \Ref{k p2} we conclude that \Ref{M2} holds for $k=\ell$.
The proposition is proved.
\end{proof}

Both functions $|a|^3P(z)/4!$ and $\hat P(z)$ satisfy the same equation and both functions are homogeneous polynomials in $z$
of degree four. Hence $|a|^3P(z)/4! = \hat P(z)$. Thus the proposition implies the theorem.
\end{proof}

The period map $q:\cd\to\cd$ is a polynomial map in $z$ of degree two with respect to the combinatorial connection.

The space $\sing V$ has distinguished bases labeled by $k\in J$. The basis corresponding to $k$
consists of the
vectors $v_{i,j}$ such that $1\leq i<j\leq n$ and $k\notin\{i,j\}$. Such a basis defines coordinate hyperplanes in
$\sing V$.

\begin{lem}
The period map sends the discriminant $\Delta\subset\C^n$ to the union $\Delta_V\subset \sing V$
of all coordinate hyperplanes of all distinguished bases in
$\sing V$.

\end{lem}
\begin{proof}
The period map is given by the formula $q(z) = \frac 1{|a|^2}\sum' \frac{ f_{i,j,k}^2}{ d_{i,j}d_{j,k}d_{k,i}}\,v_{i,j},$
where the sum is over all pairs $i<j$ such that $k\notin\{i,j\}$. Thus the functions $\frac{ f_{i,j,k}^2}{ d_{i,j}d_{j,k}d_{k,i}}$
are the coordinate functions of the period map in this (combinatorially flat) basis. The lemma follows from this description of the
 coordinate functions.
\end{proof}

\begin{lem}
\label{lem kernel}
For $z\in\C^n-\Delta$, the kernel of the differential of the period map is two dimensional. The kernel is spanned by
the vectors
\bean
\sum_{j\ne i} d_{j,i}\der_j, \qquad i\in J.
\eean
Any two of these vectors are linearly independent.
\qed
\end{lem}

Introduce the {\it potential function of second kind}
\bean
\label{POT2k CC}
\tilde P(z_1,\dots,z_n) = \frac1{4!}\!\sum_{1\leq i<j<k\leq n}\! \frac{a_ia_ja_k}{d_{i,j}^2d_{j,k}^2d_{k,i}^2} f_{i,j,k}^4\log f_{i,j,k}.
\eean

\begin{thm}
\label{pot2k thm}
For any $m_0,\dots,m_4\in J$ we have
\bean
\label{5der}
\frac{\der^5\tilde P}{\der z_{m_0}\dots\der z_{m_4}}(z) = \Big(\Big[\frac{a_{m_0}}{f_{m_0}}\Big]*_z\dots*_z \Big[\frac{a_{m_4}}{f_{m_4}}\Big], [1](z)\Big)_z.
\eean

\end{thm}

Theorem \ref{pot2k thm} proves Conjecture \ref{CB4} for this family of arrangements.

If $m_1\ne m_2$ and $m_3\ne m_4$, equation \Ref{5der} takes the form
\bean
\label{dERpk2 CC}
S^{(a)}(K_{m_0}(z)v_{m_1,m_2},v_{m_3,m_4}) = d_{m_1,m_2}d_{m_3,m_4} \frac{\der^5\tilde P}{\der z_{m_0}\dots\der z_{m_4}}(z).
\eean

\begin{cor}
\label{an comb2}
The matrix elements of the operators $K_i(z)$ with respect to the
(combinatorially constant) vectors $v_{i,j}$ are described by the fifth
derivatives of the potential function of second kind.

\end{cor}

 Notice that
\bean
S^{(a)}(v_{m_1,m_2},v_{m_3,m_4}) = d_{m_1,m_2}d_{m_3,m_4} \frac{|a|^4}{4!}\frac{\der^4P}{\der z_{m_1}\dots\der z_{m_4}}(z),
\eean
where $P(z)$ is the potential function  of first kind, see Theorem \ref{thm MULTI}.

\medskip
{\it Proof of Theorem \ref{pot2k thm}.}
We have the relation $\sum_{j\in J}d_{i,j}\big[\frac{a_j}{f_j}  \big] = 0$ for any $i\in J$, see \Ref{a/f k=2},
and the relation
\bean
\label{kerPotk2}
\sum_{j\in J}d_{i,j}\frac{\der}{\der z_j}   \frac{\der^4P}{\der z_{m_1}\dots\der z_{m_4}}(z)
=0
\eean
 for any $m_1,\dots,m_4,i\in J$. By using these two relations and by reordering the set $J$ if necessary,
 we can reduce formula  \Ref{5der} to three cases in which $(m_0,\dots,m_4)$ equals $(5,1,2,3,4)$ or
$(3,1,2,3,4)$ or $(3,1,2,1,2)$.

Let $(m_0,\dots,m_4)=(5,1,2,3,4)$. Then $\frac{\der^5\tilde P}{\der z_{m_0}\dots\der z_{m_4}}(z)=0$ and
\bean
&&
\Big(\Big[\frac{a_{m_0}}{f_{m_0}}\Big]*_z\dots*_z \Big[\frac{a_{m_4}}{f_{m_4}}\Big], [1](z)\Big)_z =
\frac 1{d_{1,2}d_{3,4}}S^{(a)}(K_5(z)v_{1,2},v_{3,4})=
\\
\notag
&&
\phantom{aa}
= \frac {d_{1,2}}{d_{1,2}d_{3,4}f_{5,1,2}}
S^{(a)}(a_5v_{1,2}+a_1v_{2,5}+a_2v_{5,1},v_{3,4}) = 0.
\eean
Let $(m_0,\dots,m_4)=(3,1,2,3,4)$. Then $\frac{\der^5\tilde P}{\der z_{m_0}\dots\der z_{m_4}}(z)=0$ and
\bean
&&
\Big(\Big[\frac{a_{m_0}}{f_{m_0}}\Big]*_z\dots*_z \Big[\frac{a_{m_4}}{f_{m_4}}\Big], [1](z)\Big)_z =
\frac 1{d_{1,2}d_{3,4}}S^{(a)}(K_3(z)v_{1,2},v_{3,4})=
\\
\notag
&&
\phantom{aa}
= \frac {d_{1,2}}{d_{1,2}d_{3,4}f_{3,1,2}}
S^{(a)}(a_3v_{1,2}+a_1v_{2,3}+a_2v_{3,1},v_{3,4}) =
\\
\notag
&&
\phantom{aa}
=
\frac {d_{1,2}}{d_{1,2}d_{3,4}f_{3,1,2}}
\big(0 + a_1\frac{a_2a_3a_4}{|a|} - a_2\frac{a_1a_3a_4}{|a|}\big)=0.
\eean
Let $(m_0,\dots,m_4)=(3,1,2,1,2)$. Then $\frac{\der^5\tilde P}{\der z_{m_0}\dots\der z_{m_4}}(z)=\frac {a_1a_2a_3} {d_{1,2}f_{1,2,3}}$ and
\bean
&&
\Big(\Big[\frac{a_{m_0}}{f_{m_0}}\Big]*_z\dots*_z \Big[\frac{a_{m_4}}{f_{m_4}}\Big], [1](z)\Big)_z =
\frac 1{d_{1,2}d_{1,2}}S^{(a)}(K_3(z)v_{1,2},v_{1,2})=
\\
\notag
&&
\phantom{aa}
= \frac {d_{1,2}}{d_{1,2}d_{1,2}f_{3,1,2}}
S^{(a)}(a_3v_{1,2}+a_1v_{2,3}+a_2v_{3,1},v_{1,2}) =
\\
\notag
&&
\phantom{aa}
=
\frac {d_{1,2}}{d_{1,2}d_{1,2}f_{3,1,2}}
\big(a_3\frac{a_1a_2\sum_{j\notin\{1,2\}}a_j}{|a|}
+ a_1\frac{a_1a_2a_3}{|a|} + a_2\frac{a_1a_2a_3}{|a|}\big)= \frac {a_1a_2a_3} {d_{1,2}f_{1,2,3}}.
\eean
The theorem is proved.
 \qed

\subsection{Frobenius like structure}

Consider the quotient $M$ of $\C^n$ by the two-dimensional subspace, which is the kernel of the period map,
see Lemma \ref{lem kernel}. Let $\pi: \C^n\to M$ be  the natural projection.
Then all our objects  descend to  the quotient and form on $M-\pi(\Delta)$ a structure which we will also
call a {\it Forbenius like structure}.

\section{Generic arrangements  in $ \C^k$}
\label{sec arrs k}

\subsection{An arrangement in  $\C^n\times\C^k$}
\label{An arrangement in ck}
Consider $\C^k$ with coordinates $t_1,\dots,t_k$,\
$\C^n$ with coordinates $z_1,\dots,z_n$. Fix $n$
 linear functions on $\C^k$,
$g_j=\sum_{m=1}^k b^m_jt_m,$\ $ j\in J,$
$b_j^m\in \C$. For ${i_1,\dots,i_k}\subset J$,
denote
\bean
d_{i_1,\dots,i_k} = \text{det}_{\ell,m=1}^k (b^m_{i_\ell}).
\eean
We assume that all the numbers $d_{i_1,\dots,i_k}$ are nonzero if ${i_1,\dots,i_k}$ are distinct. In other words we assume that the collection of functions $g_j, j\in J$, is generic.
We define $n$ linear functions on $\C^n\times\C^k$,
$f_j = z_j+g_j,$\ $ j\in J.$
In $\C^n\times \C^k$ we define
 the arrangement $\tilde \A = \{ \tilde H_j\ | \ f_j = 0, \ j\in J \}$.

For every $z=(z_1,\dots,z_n)$ the arrangement $\tilde \A$
induces an arrangement $\A(z)$ in the fiber of the projection $\C^n\times\C^k\to\C^n$
over $z$. We
identify every fiber with $\C^k$. Then $\A(z)$ consists of
hyperplanes  $H_j(z), j\in J$, defined in $\C^k$ by the equations
$f_j=0$.
Denote $ U(\A(z)) = \C^k - \cup_{j\in J} H_j(z)$,  the complement to the arrangement $\A(z)$.

The arrangement $\A(z)$ is with normal crossings if and only if $z \in \C^n-\Delta$,
\bean
\Delta = \cup_{\{i_1<\dots<i_{k+1}\}\subset J}H_{i_1,\dots,i_{k+1}},
\eean
where $H_{i_1,\dots,i_{k+1}}$ is the hyperplane defined by the equation
$f_{i_1,\dots,i_{k+1}}=0$,
\bean
\label{ i_1,...,i_k}
f_{i_1,\dots,i_{k+1}}= \sum_{m=1}^{k+1} (-1)^{m-1} z_{i_m} d_{i_1,\dots,\widehat{i_m},\dots,i_{k+1}}.
\eean

\subsection{Good fibers}
\label{sec Good exk}

For any $z\in\C^n-\Delta$, the space $\OS^k(\A(z))$ has the standard basis
\linebreak
$(H_{i_1}(z),\dots,H_{i_k}(z))$,  $1\leq i_1<\dots<i_k\leq n$. The space $\FF^k(\A(z))$ has the standard dual basis
$F(H_{i_1}(z),\dots,H_{i_k}(z))$.
For $z^1, z^2\in \C^n-\Delta$, the combinatorial connection identifies the spaces
$\OS^k(\A(z^1))$, $\FF^k(\A(z^1))$   with the spaces
$\OS^k(\A(z^2))$, $\FF^k(\A(z^2))$, respectively, by identifying the corresponding standard bases.

Assume that nonzero weights $(a_j)_{j\in J}$ are given. Then each
arrangement $\A(z)$  is weighted. For $z\in\C^n-\Delta$, the arrangement $\A(z)$ is unbalanced
if $|a|\ne 0$.  We assume that $|a|\ne 0$.

For $z\in\C^n-\Delta$, we denote $V=\FF^k(\A(z))$, $V^*=(\FF^k(\A(z))^*=\OS^k(\A(z))$,
 $F_{i_1,\dots,i_k} = F(H_{i_1}(z),\dots,H_{i_k}(z))$. For any permutation $\sigma\in \Sigma_k$,
 we have $F_{i_{\sigma(1)},\dots,i_{\sigma(k)}}=(-1)^\sigma F_{i_1,\dots,i_k}$.
 If $v = \sum_{1\leq i_1<\dots<i_k\leq n} c_{i_1,\dots,i_k}F_{i_1,\dots,i_k}$ is a vector of $V$,
 we introduce $c_{i_1,\dots,i_k}$ for all $ i_1,\dots,i_k\in J$ by the rule:
 $c_{i_{\sigma(1)},\dots,i_{\sigma(k)}}=(-1)^\sigma c_{i_1,\dots,i_k}$.
 The contravariant form on $V$ is defined by
\bean
\label{sing S k}
S^{(a)}(F_{i_1,\dots,i_k},F_{j_1,\dots,j_k}) &=& 0,
\qquad \text{if} \ \{i_1,\dots,i_k\}\ne \{i_1,\dots,i_k\},
\\
\notag
S^{(a)}(F_{i_1,\dots,i_k},F_{i_1,\dots,i_k}) &=& \prod_{m=1}^k a_{i_m},
\eean
the singular subspace is defined by
\bean
&&
{}
\\
\notag
&&
\sing V
 =
\Big\{ \sum_{1\leq i_1<\dots<i_k\leq n} c_{i_1,\dots,i_k}F_{i_1,\dots,i_k}\ | \ \sum_{j\in J} \,a_j\,c_{j,j_1,\dots,j_{k-1}}=0\
\text{for all}\ \{j_1,\dots,j_{k-1}\}\subset J \Big\}.
\eean
By Corollary \ref{cor nondeg}, the restriction of $S^{(a)}$ to $\sing V$ is nondegenerate.
Denote $(\sing V)^\perp$ the orthogonal complement to $\sing V$ with respect to $S^{(a)}$. Then
$V = \sing V \oplus (\sing V)^\perp$. Denote $\pi : V\to \sing V$ the orthogonal projection.

\begin{lem}
\label{elm orth k}
The space $(\sing V)^\perp$ is generated by vectors
\bean
\label{wj a}
\sum_{j\in J} F_{j, j_1,\dots,j_{k-1}},
\eean
labeled by subsets $\{j_1,\dots,j_{k-1}\}\subset J.$
\qed
\end{lem}

For distinct  ${i_1,\dots,i_k}$, we define the vector $v_{i_1,\dots,i_k}\in V$ by the formula
\bean
\label{ v ij k}
v_{i_1,\dots,i_k} = F_{i_1,\dots,i_k} -\sum_{m=1}^k\frac{a_{i_m}}{|a|}\sum_{j\in J} F_{i_1,\dots,i_{m-1},j,i_{m+1},\dots,i_k}.
\eean
We have $v_{i_{\sigma(1)},\dots,i_{\sigma(k)}}=(-1)^\sigma v_{i_1,\dots,i_k}$. Set $v_{i_1,\dots,i_k}=0$ if ${i_1,\dots,i_k}$ are not distinct.

\begin{lem}
\label{lem siNg k}
We have the following properties.

\begin{enumerate}
\item[(i)]
$\dim \Sing V= {n-1\choose k}$.

\item[(ii)]

For distinct  ${i_1,\dots,i_k}$, we have $v_{i_1,\dots,i_k}\in\Sing V$ and $v_{i_1,\dots,i_k} = \pi(F_{i_1,\dots,i_k})$.

\item[(iii)]
For $\{j_1,\dots,j_{k-1}\}\subset J$,  we have $\sum_{j\in J} v_{j_1,\dots,j_{k-1},j}=0$.

\item[(iv)]
For any $m\in J$, the set $v_{i_1,\dots,i_k}$, $ 1\leq i_1<\dots<i_k\leq n$, $m\notin\{{i_1,\dots,i_k}\}$, is a
basis of $\sing V$.
\end{enumerate}
\qed
\end{lem}

\begin{lem}
\label{lem Shap}
We have
\bean
\label{SF jji k}
&&
S^{(a)}(v_{i_1,\dots,i_k},v_{j_1,\dots,j_k}) = 0,
\qquad \text{if}\ |\{i_1,\dots,i_k\}\cap\{j_1,\dots,j_k\}|<k-1,
\\
\notag
\label{SF ijii k}
&&
S^{(a)}(v_{i_1,\dots,i_{k-1},i_k},v_{i_1,\dots,i_{k-1},i_{k+1}})
 = -\frac{\prod_{\ell=1}^{k+1}a_{i_\ell}}{|a|}\qquad \text{for distinct}\ i_1,\dots,i_{k-1},i_k,i_{k+1},
\\
\notag
\label{SF ijiii k}
&&
S^{(a)}(v_{i_1,\dots,i_k},v_{i_1,\dots,i_k})=
\frac{(\sum_{\ell\notin\{i_1,\dots,i_k\}}a_{i_\ell})\prod_{m=1}^{k}a_{i_m}}{|a|}.
\eean
\qed
\end{lem}

\begin{proof}
The lemma is a straightforward corollary of \Ref{sing S k}  and \Ref{ v ij k}.
\end{proof}

\subsection{Operators  $K_i(z): V\to V$}
\label{sec ham ex Ck}

For any  subset $\{i_1,\dots,i_{k+1}\}\subset J$, we
define the linear operator
$L_{i_1,\dots,i_{k+1}} : V\to V$ by the formula
\bean
\label{def L k}
\phantom{aaa}
F_{i_1,\dots,\widehat{i_m},\dots,i_{k+1}}
& \mapsto &
(-1)^m\sum_{\ell=1}^{k+1} (-1)^\ell a_{i_\ell} F_{i_1,\dots,\widehat{i_\ell},\dots,i_{k+1}},
\qquad m=1,\dots,k+1,
\\
\notag
F_{j_1,\dots,j_{k}}
& \mapsto & 0, \qquad \text{if}\ \{j_1,\dots,j_{k}\} \ \text{is not a subset of}\
\{i_1,\dots,i_{k+1}\},
\eean
see formula \Ref{L_C}. Notice that $L_{i_1,\dots,i_{k+1}}$ does not depend on the order of ${i_1,\dots,i_{k+1}}$.

 We define the operators $K_i(z): V\to V$, $i\in J$, by the formula
\bean
\label{K ex}
K_i(z) = \sum \frac{d_{i_1,\dots,i_k}}{f_{i,i_1,\dots,i_k}}\,L_{i,i_1,\dots,i_k} ,
\eean
where the sum is over all unordered subsets $\{i_1,\dots,i_k\}\subset J-\{i\}$,
see formula \Ref{K_j}.
For any $i\in J$ and $z\in \C^n-\Delta$, the operator $K_i(z)$ preserves
the subspace $\Sing V\subset V$ and is a symmetric operator,
$S^{(a)}(K_i(z)v,w)= S^{(a)}(v, K_i(z)w)$ for all $v,w\in V$, see Theorem \ref{thm K sym}.

\begin{lem}
\label{lem Kv}
We have
\bean
\label{KJ k}
\phantom{aaaaaa}
K_{i_1}(z)v_{i_2,\dots,i_{k+1}}  &=&
\frac{d_{i_2,\dots,i_{k+1}}}{f_{i_1, i_2,\dots,i_{k+1}}}\,\sum_{\ell=1}^{k+1} (-1)^{\ell+1} a_{i_\ell} v_{i_1,\dots,\widehat{i_\ell},\dots,i_{k+1}},
\qquad \text{if}\ i_1\notin\{i_2,\dots,i_{k+1}\},
\\
\notag
K_{i_1}(z) v_{i_1,i_2,\dots,i_k} &=& - \sum_{m\notin\{i_1,\dots,i_k\}} K_{i_1}(z) v_{m,i_2,\dots,i_k}.
\eean
\end{lem}

\begin{proof}
The operator $K_i(z)$ preserve the decomposition $\sing V\oplus (\sing V)^\perp$. Hence
\bea
\label{KJ e k}
&&
K_{i_1}(z) v_{i_2,\dots,i_{k+1}} = K_{i_1}(z) \pi (F_{i_2,\dots,i_{k+1}}) =\pi(K_{i_1}(z) F_{i_2,\dots,i_{k+1}}) =
\\
&&
= \pi \Big(\frac{d_{i_2,\dots,i_{k+1}}}{f_{i_1, i_2,\dots,i_{k+1}}}\,\sum_{\ell=1}^{k+1} (-1)^{\ell+1} a_{i_\ell} F_{i_1,\dots,\widehat{i_\ell},\dots,i_{k+1}}\Big)=
\frac{d_{i_2,\dots,i_{k+1}}}{f_{i_1, i_2,\dots,i_{k+1}}}\,\sum_{\ell=1}^{k+1} (-1)^{\ell+1} a_{i_\ell} v_{i_1,\dots,\widehat{i_\ell},\dots,i_{k+1}}.
\eea
\end{proof}

The differential equations \Ref{dif eqn} for flat sections of the Gauss-Manin connection on
\linebreak
 $(\C^n-\Delta)\times \sing V\to \C^n-\Delta$
take the form
\bean
\label{dif eqn ex k}
\kappa \frac{\der I}{\der z_j}(z) = K_j(z)I(z),
\qquad
j\in J.
\eean
For generic $\kappa$ all the  flat sections are given by the formula
\bean
\label{I(z) k}
I_\gamma(z) =
\sum_{1\leq i_1<\dots<i_k\leq n}
\Big(\int_{\gamma(z)} \prod_{m\in J}f_m^{a_m/\kappa} \frac {d_{i_1,\dots,i_k}}{f_{i_1}\dots f_{i_k}}\,dt_1\wedge \dots\wedge dt_k\Big) F_{i_1,\dots,i_k} ,
\eean
see formula \Ref{Ig}. These generic $\kappa$ can be determined more precisely from the determinant formula in \cite{V1}.

\subsection{Algebra $A_\Phi(z)$}
\label{sec master k}

The master function of the arrangement $\A(z)$  is
\bean
\label{def mast k}
\Phi(z,t) = \sum_{j\in J}\,a_j \log f_j.
\eean
The critical point equations are
\bean
\frac{\der\Phi}{\der t_i} = \sum_{j\in J} b^i_j\frac{a_j}{f_j} = 0, \qquad i=1,\dots,k.
\eean
Introduce $H_i, i=1,\dots,k$, by the formula
\bean
\frac{\der \Phi}{\der z_i}= \frac{H_i}{\prod_{j\in J}f_j}.
\eean
The critical set is
\bean
C_\Phi(z)
&=&
\{t\in U(\A(z))\ |\ \frac{\der \Phi}{\der z_i}(z,t)=0, i=1,\dots,k\}
=
\\
\notag
&=& \{t\in U(\A(z))\ |\ H_i(z,t)=0, i=1,\dots,k\}.
\eean
The algebra of functions on the critical set is
\bean
A_\Phi(z) &=& \C(U(\A(z)))
/\Big\langle \frac{\der\Phi}{\der t_1},\dots,\frac{\der\Phi}{\der t_k}  \Big\rangle=
\C(U(\A(z)))
/\big\langle H_1,\dots, H_k \big\rangle.
\eean

\begin{lem}
\label{lem basis cr k}
We have $\dim A_\Phi(z) = {n-1\choose k}$.
\qed
\end{lem}

Introduce elements $w_{i_1,\dots,i_k} \in A_\Phi(z)$ by the formula
\bean
w_{i_1,\dots,i_k}= a_{i_1}\dots a_{i_k}\,\Big[\frac{d_{i_1,\dots,i_k}}{f_{i_1}\dots f_{i_k}}\Big].
\eean

\begin{lem}
\label{lem w rel k}

We have $w_{i_{\sigma(1)},\dots,i_{\sigma(k)}}=(-1)^\sigma w_{i_1,\dots,i_k}$  for $\sigma \in \Sigma_k$ and
\bean
\label{w rel k}
\sum_{i\in J}  w_{i_1,\dots,i_{k-1},i} = 0.
\eean
\qed
\end{lem}
\begin{proof}
The lemma follows from the identity
\bean
&&
\frac{a_{i_1}df_{i_1}}{f_{i_1}} \wedge\dots\wedge\frac{a_{i_{k-1}}df_{i_{k-1}}}{f_{i_{k-1}}}\wedge \frac{d\Phi}{\Phi} =
\\
\notag
&&
\phantom{aaa}
= \sum_{i\in J} a_{i_1}\dots a_{k-1} a_i\,\frac{d_{i_1,\dots,i_{k-1},i}}{f_{i_1}\dots f_{k-1}f_i} dt_1\wedge\dots\wedge dt_k +
\sum_{m\in J} dz_j\wedge \mu_j,
\eean
where $\mu_j$ are suitable $k-1$-forms.
\end{proof}

The elements $\big[ \frac{a_i}{f_i}\big]$, $i\in J$, generate $A_\Phi(z)$ by Lemma \ref{lem f_j generate}.

\begin{lem}
For $i_1,\dots,i_{k-1}\in J$, we have the following identity in $A_\Phi(z)$,
\bean
\label{a/f k}
\sum_{i\in J} d_{i,i_1,\dots,i_{k-1}} \Big[\frac{a_i}{f_i}\Big] =0.
\eean
\qed
\end{lem}

Denote $I = \{i_1,\dots,i_{k-1}\}$. Relation \Ref{a/f k} will be called the $I$-{\it relation.}

\begin{lem}
\label{wf mult}
We have in $A_\Phi(z)$,
\bean
&&{}
\\
\notag
\Big[\frac{a_{i_1}}{f_{i_1}}\Big]*_z w_{i_2,\dots,i_{k+1}}
&=&
\frac{d_{i_2,\dots,i_{k+1}}}{f_{i_1, i_2,\dots,i_{k+1}}}\,\sum_{\ell=1}^{k+1} (-1)^{\ell+1} a_{i_\ell} w_{i_1,\dots,\widehat{i_\ell},\dots,i_{k+1}},
\qquad \text{if}\ i_1\notin\{i_2,\dots,i_{k+1}\},
\\
\notag
\Big[\frac{a_{i_1}}{f_{i_1}}\Big]*_z w_{i_1,i_2,\dots,i_k} &=& - \sum_{m\notin\{i_1,\dots,i_k\}} \Big[\frac{a_{i_1}}{f_{i_1}}\Big]*_z w_{m,i_2,\dots,i_k}.
\eean
\qed
\end{lem}

\begin{proof}
The first formula follows from the identity
\bean
\phantom{aaa}
\frac {df_{i_1}}{f_{i_1}}\wedge\dots\wedge \frac {df_{i_{k+1}}}{f_{i_{k+1}}}
= \frac {df_{i_1,\dots,i_{k+1}}}{f_{i_{k+1}}}\wedge \sum_{m=1}^{k+1}(-1)^{m-1}
\frac {df_{i_1}}{f_{i_1}}\wedge\dots\wedge \widehat{\frac {df_{i_{m}}}{f_{i_{m}}}}\wedge\dots\wedge
\frac {df_{i_{k+1}}}{f_{i_{k+1}}}.
\eean
\end{proof}

\begin{lem}
\label{MONOM k}
Fix $i_0\in J$. Then every monomial $M = \prod_{j\in J}\big[\frac{a_{j}}{f_{j}}\big]^{s_j}\in A_\Phi(z)$ with $\sum_{j\in J} s_j=k$
can be written as a linear combination of elements $w_{i_1,\dots,i_k}$ where $i_1<\dots<i_k$ and $i_0\notin\{i_1,\dots,i_k\}$ with coefficients
independent of $z$.
\end{lem}

\begin{proof}
Let us write
\bean
M = \Big[\frac{a_{i_0}}{f_{i_0}}\Big]^{\ell_{i_0}} \Big[\frac{a_{j_1}}{f_{j_1}}\Big]^{\ell_{j_1}}\dots\Big[\frac{a_{j_m}}{f_{j_m}}\Big]^{\ell_{j_m}},
\eean
where $i_0,j_1,\dots,j_m$ are distinct, $\ell_{i_0},\ell_{j_1},\dots,\ell_{j_m}$ are positive and
 $\ell_{i_0}+\ell_{j_1}+\dots+\ell_{j_m}=k$.

If $\ell_{i_0}>0$, then let us decrease $\ell_{i_0}$ by one. For that let us use an $I$-relation of formula \Ref{a/f k}, where
$I=\{p_1,\dots,p_{k-1}\}$ is any subset which contains $j_1,\dots,j_m$ but does not contain $i_0$.
By using \Ref{a/f k}, we can write
\bean
 \Big[\frac{a_{i_0}}{f_{i_0}}\Big]^{\ell_{i_0}}=-\Big[\frac{a_{i_0}}{f_{i_0}}\Big]^{\ell_{i_0}-1}\Big(\sum_{i\notin \{i_0,p_1,\dots,p_{k-1}\}}
 \frac{d_{i,p_1,\dots,p_{k-1}}}{d_{i_0,p_1,\dots,p_{k-1}}} \Big[\frac{a_i}{f_i}\Big]\Big).
 \eean
Substituting this expression into $M$, we will present $M$ as a sum of monomials $M'$
with the degree of $\big[\frac{a_{i_0}}{f_{i_0}}\big]$ equal to $\ell_{i_0}-1$. In any monomial $M'$ the degrees of initial factors
$\big[\frac{a_{j_s}}{f_{j_s}}\big]$ are the same and one new factor appears. Now to each of the monomials $M'$ we will apply the same
procedure until the monomial $ \big[\frac{a_{i_0}}{f_{i_0}}\big]$ will not appear in each of the constructed monomial. Then we will decrease those degrees
of $\ell_{j_1},\dots, \ell_{j_m}$ which are greater than one. In the end we will present $M$ as a sum of monomials of the form
$\big[\frac{a_{i_1}}{f_{i_1}}\big]\dots\big[\frac{a_{i_k}}{f_{i_k}}\big],$
where $i_1<\dots<i_k$ and $i_0\notin\{i_1,\dots,i_k\}$. Such a monomial equals $\frac 1{d_{i_1,\dots,i_k}}w_{i_1,\dots,i_k}$.
\end{proof}

\begin{lem}
\label{MONOM}
Fix $i_0\in J$. Then every monomial $M = \prod_{j\in J}\big[\frac{a_{j}}{f_{j}}\big]^{s_j}\in A_\Phi(z)$ can be written
as a linear combination of elements $w_{i_1,\dots,i_k}$ where $i_1<\dots<i_k$ and $i_0\notin\{i_1,\dots,i_k\}$.
If $\sum_{j\in J}s_j\ne k$, then the coefficients of the linear combination may depend on $z$.

\end{lem}

\begin{proof}
The lemma follows from Lemmas \ref{MONOM k}, \ref{wf mult},  \ref{lem on ONE}.
\end{proof}

\begin{thm}
\label{basis th}
Fix $i_0\in J$. Then the ${n-1\choose k}$  elements $w_{i_1,\dots,i_k}$ with $i_1<\dots<i_k$ and $i_0\notin\{i_1,\dots,i_k\}$, form a basis of $A_\Phi(z)$.
\end{thm}

\begin{proof}
The elements $\big[ \frac{a_i}{f_i}\big]$, $i\in J$, generate $A_\Phi(z)$ by Lemma \ref{lem f_j generate}.
Each polynomial in $\big[ \frac{a_i}{f_i}\big]$, $i\in J$,  is a linear combination of ${n-1\choose k}$ elements
$w_{i_1,\dots,i_k}$ with $i_1<\dots<i_k$ and $i_0\notin\{i_1,\dots,i_k\}$ by Lemma \ref{MONOM}. But $\dim A_\Phi(z) = {n-1\choose k}$.
The theorem follows.
\end{proof}

\begin{thm}
\label{thm id k}
For $i_{0} \in J$, the identity element  $[1](z)\in \Ap$ satisfies the equation
\bean
\label{O}
[1](z) &=& \frac {1}{|a|^k} \sum_{i_1<\dots<i_k
\atop i_{0}\notin\{i_1,\dots,i_k\}} \frac{f_{i_0,i_1,\dots,i_k}^k}{\prod_{m=0}^{k} (-1)^md_{i_0,\dots,\widehat{i_m},\dots,i_{k}}} w_{i_1,\dots,i_k} =
\\
\notag
&=&
\frac {1}{|a|^k} \sum_{i_1<\dots<i_k
\atop i_{0}\notin\{i_1,\dots,i_k\}} \frac{(\sum_{m=0}^{k} (-1)^{m} z_{i_m} d_{i_0,\dots,\widehat{i_m},\dots,i_{k}})^k}
{\prod_{m=0}^{k} (-1)^md_{i_0,\dots,\widehat{i_m},\dots,i_{k}}} w_{i_1,\dots,i_k}.
\eean

\end{thm}

\begin{proof}
Our goal is to prove that the decomposition of
\bean
\label{b}
[1](z) = \frac 1{|a|^k} \Big(\sum_{j\in J} z_j\Big[\frac{a_j}{f_j}\Big]\Big)^k =
\frac 1{|a|^k} \sum_{s_1+\dots+s_n=k}{k \choose s_1,\dots,s_n} \prod_{j\in J} z_j^{s_j}\Big[\frac{a_j}{f_j}\Big]^{s_j}
\eean
with respect to the basis $(w_{i_1,\dots,i_k}, \,i_1<\dots<i_k, \,i_0\notin\{i_1,\dots,i_k\})$
equals the right hand side in \Ref{O}. For every monomial $\prod_{j\in J} z_j^{s_j}$ we need to show that
${k \choose s_1,\dots,s_n}\prod_{j\in J} \big[\frac{a_j}{f_j}\big]^{s_j} $ equals the coefficient of that monomial in \Ref{O}.
For that we need to express $\prod_{j\in J} \big[\frac{a_j}{f_j}\big]^{s_j} $ as a linear combination of basis vectors. To obtain
this  linear combination we will eliminate from this product the factor $\big[\frac{a_{i_0}}{f_{i_0}}\big]^{s_{i_0}} $ and will make
the powers of all other factors not greater than 1. This will be done by using the $I$-relations of formula \Ref{a/f k} like in the proof of
Lemma \ref{MONOM k}. At every step of that simplification we will use one of the $I$-relations. Although the steps of this procedure are not unique,
 the resulting linear combination is unique. To prove that the linear combination of basis vectors representing
${k \choose s_1,\dots,s_n}\prod_{j\in J} \big[\frac{a_j}{f_j}\big]^{s_j} $ equals the coefficient of
$\prod_{j\in J} z_j^{s_j} $ in \Ref{O}, we will fix an arbitrary basis vector $w_{i_1,\dots,i_k}$  and choose a particular sequence of $I$-relations so
that the coefficient of $w_{i_1,\dots,i_k}$ in the decomposition of ${k \choose s_1,\dots,s_n}\prod_{j\in J} \big[\frac{a_j}{f_j}\big]^{s_j} $
will be equal to the coefficient of  $\prod_{j\in J} z_j^{s_j} w_{i_1,\dots,i_k}$ in \Ref{O}.

By comparing the coefficients of a monomial $\prod_{j\in J} z_j^{s_j}$ in \Ref{O} and \Ref{b}, we observe  that the coefficients have common factors
$\frac 1{|a|^k}{k\choose s_1,\dots,s_k}$, so we will ignore these common factors in our next reasonings.

Before explaining the choice of the $I$-relations for an arbitrary pair $(\prod_{j\in J} z_j^{s_j}, w_{i_1,\dots,i_k}) $  let us consider two examples.

As the first example, we consider a monomial $M=z_{i_1}\dots z_{i_k}$, $i_1<\dots<i_k$, $i_0\notin\{i_1,\dots z_{i_k}\}$. The coefficient of $M$ in
\Ref{b} is
\bean
&&
\Big[\frac{a_{i_1}}{f_{i_1}}\Big]\dots\Big[\frac{a_{i_k}}{f_{i_k}}\Big] =
\frac 1{d_{i_1,\dots, i_k}} w_{i_1,\dots, i_k} =
\frac 1{(-1)^0d_{\widehat{i_0},i_1,\dots, i_k}} w_{i_1,\dots, i_k} =
\\
\notag
&&
\phantom{aaaa}
=
\frac{\prod_{m=1}^k (-1)^m d_{i_0,\dots, \widehat{i_m},\dots,i_k}}
{\prod_{m=0}^k (-1)^m d_{i_0,\dots,\widehat{i_m},\dots,i_k}} w_{i_1,\dots,i_k},
\eean
which is  the coefficient of $M$ in \Ref{O}.

As the second example we consider a monomial $M= z_{i_0}z_{j_i}\dots z_{j_{k-1}}$, where $i_0,j_i,\dots,j_{k-1}$
are distinct. The monomial $M$ appears in \Ref{O} in the coefficient  of a basis vector $w_{i_1,\dots,i_k}$ if
$\{j_1,\dots,j_{k-1}\}\subset \{i_1,\dots,i_k\}$. So we may assume that for some $1\leq \ell\leq k$, we have
$M= z_{i_0} \prod_{m=1\atop m\ne \ell}^kz_{i_m}$.
In \Ref{b}, the coefficient of  $M$ is
$\big[\frac{a_{i_0}}{f_{i_0}}\big] \prod_{m=1\atop m\ne \ell}^k\big[\frac{a_{i_m}}{f_{i_m}}\big]$.
By using the $I$-relation for $I=\{\widehat{i_0},i_1,\dots,\widehat{i_\ell},\dots,i_k\}$, we transform it into
\bean
\label{summand}
-\prod_{m=1\atop m\ne \ell}^k\Big[\frac{a_{i_m}}{f_{i_m}}\Big] \sum_{m \notin \{i_0,i_1,\dots,\widehat{i_\ell},\dots,i_k\}}
\frac {d_{m,i_1,\dots,\widehat{i_\ell},\dots,i_k}}{d_{i_0,i_1,\dots,\widehat{i_\ell},\dots,i_k}}
\Big[\frac{a_{m}}{f_{m}}\Big].
\eean
We choose the summand in \Ref{summand} corresponding to  $m=i_\ell$. This summand is
\bea
&&
-\prod_{m=1\atop m\ne \ell}^k\Big[\frac{a_{i_m}}{f_{i_m}}\Big]
\frac {d_{i_\ell,i_1,\dots,\widehat{i_\ell},\dots,i_k}}{d_{i_0,i_1,\dots,\widehat{i_\ell},\dots,i_k}}
\Big[\frac{a_{i_\ell}}{f_{i_\ell}}\Big] =
 \prod_{m=1}^k\Big[\frac{a_{i_m}}{f_{i_m}}\Big]
\frac {d_{i_1,\dots, {i_\ell},\dots,i_k}}{(-1)^\ell d_{i_0,i_1,\dots,\widehat{i_\ell},\dots,i_k}}
 =
\\
&&
\phantom{aaaaaa}
= \frac{\prod_{m=0\atop m\ne \ell}^k (-1)^m d_{i_0,\dots, \widehat{i_m},\dots,i_k}}
{\prod_{m=0}^k (-1)^m d_{i_0,\dots,\widehat{i_m},\dots,i_k}}   w_{i_1,\dots, {i_\ell},\dots,i_k},
\eea
which is the coefficient of $Mw_{i_1,\dots, {i_\ell},\dots,i_k}$ in \Ref{O}.

Now let $M$ be an arbitrary monomial of degree $k$ in variables $z_1,\dots,z_n$.
The monomial $M$ appears in \Ref{O} in the coefficient  of a vector $w_{i_1,\dots,i_k}$ if
there is a subset $\{p_1,\dots,p_r\} \subset \{1,\dots,k\}$ such that $M=z_{i_0}^{s_{0}}\prod_{m=1}^r z_{i_{p_m}}^{s_m}$,
$\sum_{m=0}^rs_m=k$, and all numbers $s_1,\dots,s_r$ are positive. Denote $\{q_1,\dots,q_{k-r}\} = \{1,\dots,k\} - \{p_1,\dots,p_r\}$, the complement.

In \Ref{b}, the coefficient of  $M$ is
$P=\big[\frac{a_{i_0}}{f_{i_0}}\big]^{s_{0}}\prod_{m=1}^r \big[\frac{a_{i_{p_m}}}{f_{i_{p_m}}}\Big]^{s_{m}}$.
To express this product as a linear combination of the basis vectors we need to apply to this product  $I$-relations $k-r$ times.
To calculate the coefficient of $w_{i_1,\dots,i_k}$
we first apply the $I$-relation with $I=\{\widehat{i_0},i_1,\dots,\widehat{i_{q_1}},\dots,i_k\}\subset
\{i_0,\dots,i_k\}$  and decrease the degree of $\big[\frac{a_{i_0}}{f_{i_0}}\big]$ by 1,
\bea
P= -\Big[\frac{a_{i_0}}{f_{i_0}}\Big]^{s_{0}-1}\prod_{m=1}^r \Big[\frac{a_{i_{p_m}}}{f_{i_{p_m}}}\Big]^{s_{m}}
\!\!\!\!\!
\sum_{m\notin\{i_0,\dots,\widehat{i_{q_1}},\dots,i_k\}}\!\!
 \frac{d_{m,\widehat{i_0},i_1,\dots,\widehat{i_{q_1}},\dots,i_k}}{d_{i_0,\widehat{i_0},i_1,\dots,\widehat{i_{q_1}},\dots,i_k}} \Big[\frac{a_{m}}{f_{m}}\Big].
\eea
In the next steps we will simplify further the first factors of this expression. After the future simplifications the only term of this sum that can give
$w_{i_1,\dots,i_k}$ is the term with $m={i_{q_1}}$, which is
\bean
\label{ex}
&&
-\Big[\frac{a_{i_0}}{f_{i_0}}\Big]^{s_{0}-1}\prod_{m=1}^r \Big[\frac{a_{i_{p_m}}}{f_{i_{p_m}}}\Big]^{s_{m}}
 \frac{d_{{i_{q_1}},\widehat{i_0},i_1,\dots,\widehat{i_{q_1}},\dots,i_k}}{d_{i_0,\widehat{i_0},i_1,\dots,\widehat{i_{q_1}},\dots,i_k}} \Big[\frac{a_{{i_{q_1}}}}{f_{{i_{q_1}}}}\Big]=
\\
\notag
&&
=
\Big[\frac{a_{i_0}}{f_{i_0}}\Big]^{s_{0}-1}\prod_{m=1}^r \Big[\frac{a_{i_{p_m}}}{f_{i_{p_m}}}\Big]^{s_{m}}
 \frac{d_{\widehat{i_0}, i_1,\dots,i_k}}{(-1)^{q_1} d_{i_0,i_1,\dots,\widehat{i_{q_1}},\dots,i_k}} \Big[\frac{a_{{i_{q_1}}}}{f_{{i_{q_1}}}}\Big].
\eean
We will call this term the {\it main term}.
Now we apply the $I$-relation with
$I=\{\widehat{i_0},i_1,\dots,$ $\widehat{{i_{q_2}}},\dots,i_k\}\subset
\{i_0,\dots,i_k\}$  and again decrease the degree of $\big[\frac{a_{i_0}}{f_{i_0}}\big]$ by one. After the second step
the only term of the obtained sum that may produce the vector $w_{i_1,\dots,i_k}$ is the term
\bea
\Big[\frac{a_{i_0}}{f_{i_0}}\Big]^{s_{0}-1}\prod_{m=1}^r \Big[\frac{a_{i_{p_m}}}{f_{i_{p_m}}}\Big]^{s_{m}}
 \frac{d_{\widehat{i_0}, i_1,\dots,i_k}}{(-1)^{{i_{q_1}}}d_{i_0,i_1,\dots,\widehat{{i_{q_1}}},\dots,i_k}} \Big[\frac{a_{{i_{q_1}}}}{f_{{i_{q_1}}}}\Big]
 \frac{d_{\widehat{i_0}, i_1,\dots,i_k}}{(-1)^{q_2}d_{i_0,i_1,\dots,\widehat{{i_{q_2}}},\dots,i_k}} \Big[\frac{a_{{i_{q_2}}}}{f_{{i_{q_2}}}}\Big].
\eea
This will be our {\it main term} after two steps of the simplifying procedure. We will repeat this procedure  to kill all factors  $\big[\frac{a_{i_0}}{f_{i_0}}\big]$.
After $s_{i_0}$ steps the {\it main term} will be
\bea
\prod_{m=1}^r \big[\frac{a_{i_{p_m}}}{f_{i_{p_m}}}\Big]^{s_{m}}
\prod_{m=1}^{s_0}
 \frac{d_{\widehat{i_0}, i_1,\dots,i_k}}{(-1)^{q_m}d_{i_0,i_1,\dots,\widehat{{i_{q_m}}},\dots,i_k}} \Big[\frac{a_{{i_{q_m}}}}{f_{{i_{q_2}}}}\Big].
\eea
Now     we will apply the $I$-relation with
$I=\{i_0,\dots,$ $\widehat{i_{p_1}},\dots,\widehat{i_{q_{s_0+1}}},\dots, i_k\}\subset
\{i_0,\dots,i_k\}$  and decrease the degree of $\big[\frac{a_{i_{p_1}}}{f_{i_{p_1}}}\big]$ by one.
 After this procedure the main term will  be
\bea
\Big[\frac{a_{i_{p_1}}}{f_{i_{p_1}}}\Big]^{s_{1}-1}\!\!
\prod_{m=2}^r \Big[\frac{a_{i_{p_m}}}{f_{i_{p_m}}}\Big]^{s_{m}} \!\!
\prod_{m=1}^{s_0}
 \frac{d_{\widehat{i_0}, i_1,\dots,i_k}}{(-1)^{q_m}d_{i_0,i_1,\dots,\widehat{{i_{q_m}}},\dots,i_k}} \Big[\frac{a_{{i_{q_m}}}}{f_{{i_{q_2}}}}\Big]\!
\times\! \frac{(-1)^{p_1}d_{{i_0}, i_1,\dots,\widehat{i_{p_1}}, \dots,i_k}}{(-1)^{q_{s_0+1}}d_{i_0,i_1,\dots,\widehat{i_{q_{s_0+1}}},\dots,i_k}} \Big[\frac{a_{{i_{q_{s_0+1}}} }}{f_{{i_{q_{s_0+1}}}}}\Big].
\eea
Now we will be decreasing the degree of $\big[\frac{a_{j_1}}{f_{j_1}}\big]$ to make it 1. Then we will continue this procedure of simplification, which will end with the main term
\bea
\Big(d_{\widehat{i_0}, i_1,\dots,i_k}\prod_{m=1}^k \Big[\frac{a_{i_m}}{f_{i_m}}\Big]\Big)
 \frac{((-1)^0d_{\widehat{i_0}, i_1,\dots,i_k})^{s_0}\prod_{m=1}^r ((-1)^{p_m}d_{{i_0}, i_1,\dots,\widehat{i_{p_m}}, \dots,i_k})^{s_m}
}{\prod_{m=0}^k(-1)^md_{i_0,\dots,\widehat{i_m},\dots,i_k}}.
\eea
After replacing the first factor $d_{\widehat{i_0}, i_1,\dots,i_k}\prod_{m=1}^k \Big[\frac{a_{i_m}}{f_{i_m}}\Big]$ with $w_{i_1,\dots,i_k}$ we observe that the second factor equals
the coefficient of $Mw_{i_1,\dots,i_k}$ in \Ref{O}. The theorem is proved.
\end{proof}

\subsection{Canonical isomoprhism}

The set of vectors $v_{i_1,\dots,i_k}, 1<i_1<\dots<i_k\leq n,$ is a basis of $\sv$, by Lemma \ref{lem siNg k}.
For $z\in \cd$, the set of vectors $w_{i_1,\dots,i_k}, 1<i_1<\dots<i_k\leq n,$ is a basis of $\ap$, by Theorem \ref{basis th}.

\begin{thm}
\label{thm cOnst}
For $z\in \cd$, the matrix of the canonical isomorphism $\al(z): \ap\to\sv$ with respect to these bases does not depend on $z$.

\end{thm}
\begin{proof}
For $1<m_1<\dots<m_k\leq n$ and $1\leq i_1<\dots<i_k\leq n$ denote
\bean
B_{i_1,\dots,i_k} = \sum_{p\in C_\Phi(z)} \Res_p  \frac 1{\prod_{j=1}^k f_{m_j}}\frac 1{\prod_{j=1}^k f_{i_j}}  \frac {\prod_{\ell\in J} f_{\ell}^k}{\prod_{i=1}^kH_i}.
\eean
Then
\bean
\al(z) (w_{m_1,\dots,m_k}) = \sum_{1\leq i_1<\dots<i_k\leq n} \frac{ d_{m_1,\dots,m_k} d_{i_1,\dots,i_k} \prod_{j=1}^ka_{m_j}}{(2\pi\sqrt{-1})^k}B_{i_1,\dots,i_k}F_{i_1,\dots,i_k},
\eean
see formulas \Ref{res map}, \Ref{map alpha}. In order to prove the theorem we need to show that every $B_{i_1,\dots,i_k}$ does not depend on $z$.

The differential form
\bean
\omega = \frac 1{\prod_{j=1}^k f_{m_j}}\frac 1{\prod_{j=1}^k f_{i_j}}  \frac {\prod_{\ell\in J} f_{\ell}^k}{\prod_{i=1}^kH_i}\,dt_1\wedge\dots\wedge dt_k
\eean
has poles only on the hypersurfaces $H_i=0$, $i=1,\dots,k$. The poles are of first order. To calculate $B_{i_1,\dots,i_k}$,
we need to take the residue $\psi=\Res_{H_i=0, i=1,\dots, k-1}\omega$ of the form $\omega$ at the curve $\mc C=\{H_i=0, i=1,\dots,k-1\}$
and then take the residue of the form $\psi$ on the curve $\mc C$ at the points where $H_k=0$. This is the same as if we
took with minus sign the residue at infinity of the form $\psi$  on the curve $\mc C$. That residue at infinity
(up to sign) can be obtained differently in two steps. First we may take the residue of  $\omega$
at the hyperplane at infinity (denote the residue by  $\phi$) and then take the residue of $\phi$ at the points of the set $\{H_i=0, i=1,\dots, k-1\}$.

So to calculate $B_{i_1,\dots,i_k}$ we first calculate $\phi$.
The coordinates at infinity are $u_1=t_1/t_k,$ \dots, $
u_{k-1}=t_{k-1}/t_k$, $u_k= 1/t_k$. We have $f_m=(b^1_mu_1  + \dots + b^{k-1}_mu_{k-1}+b^k_m+ z_mu_k)/u_k$. Denote
 $\tilde f_m (u_1,\dots,u_{k-1}) = b^1_mu_1 +\dots +b^{k-1}_mu_{k-1}+ b^k_m $.
For $i=1,\dots,k$, we have $H_i(u_1/u_k,$\dots, $u_{k-1}/u_k,1/u_k) = \hat H_i(u_1,\dots,u_{k-1},u_k)/u_k^{n-1}$, where $\hat H_i(u_1,\dots,u_{k-1},u_k)$
is a  polynomial. Denote $\tilde H_i(u_1,\dots,u_{k-1}) = \hat H_i(u_1,\dots,u_{k-1}, 0)$.
The polynomial $\tilde H_i(u_1,\dots,u_{k-1})$ does not depend on $z$.

We have $dt_1\wedge\dots\wedge dt_k = -\frac {1}{u_k^{k+1}}du_1\wedge\dots\wedge {du_k}$. By counting all orders of $u_k$ in factors of $\omega$ we conclude that
the  form $\omega$ has the first order
pole at the hyperplane at infinity. The residue $\phi$ of $\omega $ at the infinite hyperplane
equals
\bean
\pm 2\pi\sqrt{-1}\frac 1{\prod_{j=1}^k \tilde f_{m_j}}\frac 1{\prod_{j=1}^k \tilde f_{i_j}}  \frac {\prod_{\ell\in J} \tilde f_{\ell}^k}{\prod_{i=1}^k\tilde H_i}
\, du_1\wedge\dots\wedge du_{k-1}.
\eean
This form does not depend on $z$. Now we  are supposed to take the sum of residues of $\phi$ at the points of the set $\{ \tilde H_i=0, i=1,\dots,k-1\}$
and the polynomials $\tilde H_i$ also do not depend on $z$. Hence $B_{i_1,\dots,i_k}$ does not depend on $z$. The theorem is proved.
\end{proof}

Define the {\it naive isomorphism} $\nu(z) : \ap\to\sv$ by the formula
\bean
\nu(z) : w_{i_1,\dots,i_k}  \mapsto v_{i_1,\dots,i_k}
\eean
for all $i_1,\dots,i_k\in J$.

\begin{lem}
\label{lem nu}
The map $\nu(z)$ is an isomorphism of vector spaces and for every $i\in J$ and $w\in\Ap$ we have
\bean
\label{nK}
\nu(z) \Big[\frac{a_i}{f_i}\Big] *_z w = K_i(z) \nu(z) (w).
\eean

\end{lem}

\begin{proof}
The map $\nu(z)$ is an isomorphism by Lemma \ref{lem siNg k} and  Theorem \ref{basis th}. Formula $\Ref{nK}$
holds by Lemmas \ref{lem Kv} and \ref{wf mult}.
\end{proof}

Introduce the linear isomorphism
\bean
\zeta = \al(z) \nu(z)^{-1} : \sv\to\sv.
\eean
By Theorem \ref{thm cOnst} and Lemma \ref{lem nu} the isomorphism $\zeta$ does not depend on $z$.
By Theorems \ref{thm cOnst} and \ref{K/f} the isomorphism $\zeta$ commutes with the action of operators
$K_i(z)$ for all $i\in J$ and $z\in \cd$, $[K_i(z),\zeta]=0$.

\begin{thm}
\label{thm const}
The isomorphism $\zeta$ is a scalar operator.
\end{thm}

\begin{proof}
By Lemma 4.3 in \cite{V5}, the eigenvalues of the operators $K_i(z)$
separate the eigenvectors. The theorem follows from the fact that the operators $K_i(z)$
have too many eigenvectors,  and $\zeta$ must preserve all of them.
More precisely, let $i_1,\dots,i_{k+1}\in J$ be distinct. Assume that $z\in\cd$ tends to a generic point $z^0$ of the hyperplane
 defined by the equation $f_{i_1,\dots,i_{k+1}}=0$. It follows from Lemma \ref{lem Kv} that  the vector
\bean
x_{i_1,\dots,i_{k+1}} = \sum_{\ell=1}^{k+1}(-1)^{\ell+1}a_{i_\ell} v_{i_1,\dots,\widehat{i_\ell},\dots,i_{k+1}}\in \sv
\eean
is the limit of an eigenvector of operators $K_i(z)$ as $z\to z^0$. Hence, $x_{i_1,\dots,i_{k+1}}$ is an eigenvector of $\zeta$.
It is easy to see that the vectors  $x_{i_1,\dots,i_{k+1}}$ generate $\sv$ and for distinct $i_1,\dots,i_{k+2}$ we have
\bean
\sum_{\ell=1}^{k+2}(-1)^\ell a_\ell x_{i_1,\dots,\widehat{i_\ell},\dots,i_{k=2}}=0.
\eean
This equation implies that $\zeta$ is a scalar operator on the subspace generated by the vectors
$x_{i_1,\dots,\widehat{i_\ell},\dots,i_{k=2}}$, $\ell=1,\dots,k+2$, and this fact implies that $\zeta$ is a scalar operator on $\sv$.
\end{proof}

\begin{cor}
\label{cor ac}
There exists  $c\in\C^\times$ such that $\al(z) = c\, \nu(z)$, that is,
\bean
\al(z) : w_{i_1,\dots,i_k} \mapsto c\,v_{i_1,\dots,i_k}
\eean
for all $i_1,\dots,i_k\in J$.
\qed
\end{cor}

One may expect that $c=(-1)^k$, see Theorems \ref{lem can map} and \ref{thm can 2}$^\star$.

{\let\thefootnote\relax
\footnotetext{\vsk-.8>\noindent
$^\star$\, $c=(-1)^k$ by \cite[Theorem 2.16]{V7}.}

The canonical
isomorphism $\al(z)$ induces an algebra structure on $\sv$ depending on $z\in \cd$.

\begin{cor}
\label{co1}
For any $i_0\in J$, the identity element $\{1\}(z)$ of that algebra structure satisfies the equation
\bean
\{1\}(z) =
\frac {c}{|a|^k} \sum_{i_1<\dots<i_k
\atop i_{0}\notin\{i_1,\dots,i_k\}} \frac{f_{i_0,i_1,\dots,i_k}^k}{\prod_{m=0}^{k} (-1)^md_{i_0,\dots,\widehat{i_m},\dots,i_{k}}} v_{i_1,\dots,i_k},
\eean
where $c$ is defined in Corollary \ref{cor ac}.
\qed
\end{cor}

\begin{thm}
\label{thm con true generic}
Conjectures \ref{CB} and \ref{CB3} hold for this family of arrangements.
\end{thm}

\begin{proof}
Conjecture \ref{CB3} is a direct corollary of Theorem \ref{thm cOnst}.

\begin{lem}
For  $r\leq k$ and $m_1,\dots,m_r\in J$, we have
\bean
 \frac{\der^r\{1\}}{\der z_{m_1}\dots\der z_{m_r}}(z)
 = \frac{k(k-1)\dots(k-r+1)} {|a|^r} \,\al(z) \big(\prod_{i=1}^r \Big[ \frac {a_{m_i}}   {f_{m_i}}\Big]\big).
\eean
\end{lem}
\begin{proof}
The proof is by induction on $r$. For $r=0$, the statement is true: $\{1\}=\{1\}$. Assuming   the statement is true for
some $r$, we prove the statement for $r+1$. We have
\bean
\notag
&&
\frac{\der}{\der z_j} \frac{\der^{r}\{1\}}{\der z_{m_1}\dots\der z_{m_r}}(z) =
 \frac{\der}{\der z_j}\frac{k(k-1)\dots(k-r+1)} {|a|^r} \,\al(z) \big(\prod_{i=1}^r \Big[ \frac {a_{m_i}}   {f_{m_i}}\Big]\big)=
\\
\notag
&&
\phantom{}
= \frac{\der}{\der z_j}\frac{k(k-1)\dots(k-r+1)} {|a|^r} \,\al(z) \big(\prod_{i=1}^r
 \Big[ \frac {a_{m_i}}   {f_{m_i}}\Big]\,\frac 1{|a|^{k-r}}(\sum_{i\in J}z_i\Big[\frac{a_i}{f_i}\Big])^{k-r}\big) =
\\
\notag
&&
\phantom{}
=\frac {k-r}{|a|} \al(z)\big(\Big[\frac{a_j}{f_j}\Big]\big) *_z
\frac{k(k-1)\dots(k-r+1)} {|a|^r} \,\al(z) \big(\prod_{i=1}^r \Big[ \frac {a_{m_i}}   {f_{m_i}}\Big] \times
\\
&&
\notag
\times\,\frac 1{|a|^{k-r-1}}
(\sum_{i\in J}z_i\Big[\frac{a_i}{f_i}\Big])^{k-r-1}\big)
=
 \frac{k(k-1)\dots(k-r+1)(k-r)} {|a|^{r+1}} \al(z)
 \big(\Big[\frac{a_j}{f_j}\Big]\prod_{i=1}^r \Big[ \frac {a_{m_i}}   {f_{m_i}}\Big]\big).
\eean
\end{proof}

Let us finish the proof of Conjecture \ref{CB}. We have
\bean
\notag
&&
\frac{\der}{\der z_j} \frac{\der^{r}\{1\}}{\der z_{m_1}\dots\der z_{m_r}}(z) =\frac {k-r}{|a|} \al(z)\big(\Big[\frac{a_j}{f_j}\Big] \big)*_z \frac{k(k-1)\dots(k-r+1)} {|a|^r} \,\al(z) \big(
\prod_{i=1}^r \Big[ \frac {a_{m_i}}   {f_{m_i}}\Big]\big)=
\\
\notag
&&
\phantom{aaaaaaaaaaaaaaaaaaa}
=\frac {k-r}{|a|} K_j(z) \frac{\der^{r}\{1\}}{\der z_{m_1}\dots\der z_{m_r}}(z).
\eean
\end{proof}

Introduce the {\it potential function of second kind}
\bean
\label{POT2k}
\tilde P(z_1,\dots,z_n) = \frac{c^2}{(2k)!} \sum_{1\leq i_1<\dots <i_{k+1}\leq n}
\frac{\prod_{\ell=1}^{k+1}a_{i_\ell}}
{\prod_{\ell=1}^{k+1}d^2_{i_1,\dots,\widehat{i_\ell},\dots, i_{k+1}}} f_{i_1,\dots,i_{k+1}}^{2k}\log f_{i_1,\dots,i_{k+1}},
\eean
where $c$ is the constant defined in Corollary \ref{cor ac}.

\begin{thm}
\label{pot2k generic thm}
For any $m_0,\dots,m_{2k}\in J$, we have
\bean
\label{2k+1deR}
\frac{\der^{2k+1}\tilde P}{\der z_{m_0}\dots\der z_{m_{2k}}}(z) = (-1)^k
\Big(\Big[\frac{a_{m_0}}{f_{m_0}}\Big]*_z\dots*_z \Big[\frac{a_{m_{2k}}}{f_{m_{2k}}}\Big], [1](z)\Big)_z.
\eean

\end{thm}

Theorem \ref{pot2k generic thm} proves Conjecture \ref{CB4} for this family of arrangements.

If $m_1,\dots,m_k$ are distinct and $m_{k+1},\dots, m_{2k}$ are distinct, equation \Ref{2k+1deR} takes the form
\bean
\label{dERpk2}
&&
\\
\notag
&&
c^2 S^{(a)}(K_{m_{0}}(z)v_{m_1,\dots,m_k},v_{m_{k+1},\dots,m_{2k}}) =
 d_{m_1,\dots,m_k}d_{m_{k+1},\dots,m_{2k}} \frac{\der^{2k+1}\tilde P}{\der z_{m_0}\dots\der z_{m_{2k}}}(z).
\eean

\begin{cor}
\label{an combk}
The matrix elements of the operators $K_i(z)$ with respect to the (combinatorially constant) vectors $v_{i_1,\dots,i_k}$ are described by the $(2k+1)$-st
derivatives of the potential function of second kind.

\end{cor}

 Notice that
\bean
c^2 S^{(a)}(v_{m_1,\dots,m_k},v_{m_{k+1},\dots,m_{2k}}) = d_{m_1,\dots,m_k}d_{m_{k+1},\dots,m_{2k}} \frac{|a|^{2k}}{(2k)!}\frac{\der^{2k}P}{\der z_{m_1}\dots\der z_{m_{2k}}}(z),
\eean
where $P(z)$ is the potential function  of first kind, see Theorem \ref{thm MULTI}.

\begin{proof} We have the $I$-relation $\sum_{j\in J}d_{j,i_1,\dots,i_{k-1}}\big[\frac{a_j}{f_j}  \big] = 0$ for any
$i_1,\dots,i_{k-1}\in J$, see \Ref{a/f k},
and the relation
\bean
\label{kerPotk gener}
\sum_{j\in J} d_{i,i_1,\dots,i_{k-1}} \frac{\der}{\der z_j}   \frac{\der^{2k}\tilde P}{\der z_{m_1}\dots\der z_{m_{2k}}}(z)
=0
\eean
 for any $m_1,\dots,m_{2k},i_1,\dots,i_{k-1}\in J$. By using these two relations and by reordering the set $J$ if necessary,
 we can reduce formula  \Ref{2k+1deR} to the case in which $(m_1,\dots,m_k)$ are distinct,
$(m_{k+1},\dots,m_{2k})$ are distinct, and $m_{0}\notin\{m_1,\dots,m_k\}$. After that we need to check
identity \Ref{dERpk2}. That is done by direct calculation of the left and right hand sides, cf. the proof
of  Theorem \ref{pot2k thm}.

For example, the most difficult case is if  $(m_0,\dots,m_{2k})=(k+1,1,\dots,k,1,\dots,k)$. Then
\bean
\frac{\der^{2k+1}\tilde P}{\der z_{m_0}\dots\der z_{m_{2k}}}(z)= c^2 \frac{\prod_{m=1}^{k+1}a_m}{(-1)^k d_{1,\dots,k}f_{1,\dots,k+1}}
\eean
and
\bean
&&
(-1)^k\Big(\Big[\frac{a_{m_0}}{f_{m_0}}\Big]*_z\dots*_z \Big[\frac{a_{m_{2k}}}{f_{m_{2k}}}\Big], [1](z)\Big)_z =
c^2 \frac 1{d^2_{1,\dots,k}}S^{(a)}(K_{k+1}(z)v_{1,\dots,k},v_{1,\dots,k})=
\\
\notag
&&
\phantom{aa}
= c^2 \frac {1}{d_{1,\dots,k}f_{k+1,1,\dots,k}}
S^{(a)}(a_{k+1}v_{1,\dots,k}+\sum_{\ell=1}^k(-1)^\ell a_\ell v_{k+1,1,\dots,\widehat{\ell},\dots,k},v_{1,\dots,k}) =
\\
\notag
&&
\phantom{aa}
= c^2 \frac {1}{d_{1,\dots,k}f_{k+1,1,\dots,k}}
\Big( \frac{\prod_{m=1}^{k+1}a_{m}}{|a|}\sum_{\ell\notin\{1,\dots,k\}}a_\ell + \frac{\prod_{m=1}^{k+1}a_{m}}{|a|}\sum_{\ell\in\{1,\dots,k\}}a_\ell\Big) =
\\
\notag
&&
\phantom{aa}
= c^2 \frac {1}{d_{1,\dots,k}f_{k+1,1,\dots,k}}
\prod_{m=1}^{k+1}a_{m}.
\eean
 \end{proof}


\end{document}